\documentclass[11pt,a4paper,english,reqno]{amsart}
\usepackage[T1]{fontenc}
\usepackage[utf8]{inputenc}
\usepackage{color}
\usepackage{verbatim}
\usepackage{amstext}
\usepackage{amsthm}
\usepackage{stmaryrd}
\usepackage{setspace}
\usepackage{amsfonts}
\usepackage{epsfig}
\usepackage{mathrsfs}
\usepackage{esint}
\usepackage{amssymb,amsmath}
\usepackage{cite}

\makeatletter

\numberwithin{equation}{section}
\numberwithin{figure}{section}
\theoremstyle{plain}
\newtheorem{thm}{\protect\theoremname}[section]
  \theoremstyle{remark}
  \newtheorem{rem}[thm]{\protect\remarkname}
  \theoremstyle{plain}
  \newtheorem{lem}[thm]{\protect\lemmaname}
  \theoremstyle{definition}

\@ifundefined{definecolor}
 {\usepackage{color}}{}

\makeatletter
\newcommand{\Rmnum}[1]{\expandafter\@slowromancap\romannumeral#1@}\makeatother

\numberwithin{equation}{section}

\allowdisplaybreaks

\newcommand{\set}[1]{\left\{#1\right\}}

\newcommand{\defs}{:=}

\newcommand{\pt}[1]{\partial_{#1}}

\DeclareMathOperator{\di}{div}

\newcommand{\dif}{\mathrm{d}}

\DeclareSymbolFont{lettersA}{U}{pxmia}{m}{it}
\DeclareMathSymbol{\piup}{\mathord}{lettersA}{"19}

\newcommand{\Real}{\mathbb R}

\newcommand{\mb}[1]{\mathbf{#1}}

\makeatother

\makeatother

\usepackage{babel}
  \providecommand{\definitionname}{Definition}
  \providecommand{\lemmaname}{Lemma}
  \providecommand{\remarkname}{Remark}
\providecommand{\theoremname}{Theorem}

\begin{document}

\title[3-d shocks with large swirl]{Transonic shock solutions for steady 3-d axisymmetric full Euler flows with large swirl velocity in a finite cylindrical nozzle}

\author{Beixiang Fang}

\author{Xin Gao}

\author{Wei Xiang}

\author{Qin Zhao}

\address{B.X. Fang: School of Mathematical Sciences, MOE-LSC, and SHL-MAC,
Shanghai Jiao Tong University, Shanghai 200240, China }
\email{\texttt{bxfang@sjtu.edu.cn}}

\address{X. Gao: School of Mathematical Sciences, Key Laboratory of MEA (Ministry of Education) \& Shanghai Key Laboratory of PMMP, East China Normal University, Shanghai 200241, China}

\email{\texttt{xgao@math.ecnu.edu.cn}}

\address{W. Xiang: Department of Mathematics, City University of Hong Kong, Kowloon, Hong Kong, China}

\email{\texttt{weixiang@cityu.edu.hk}}

\address{Q. Zhao: School of Mathematics and Statistics, Wuhan University of Technology, Wuhan 430070, China}

\email{\texttt{qzhao@whut.edu.cn}}


\date{\today}

\begin{abstract}

This paper concerns the existence and location of three-dimensional axisymmetric transonic shocks with large swirl velocity for shock solutions of the steady compressible full Euler system in a cylindrical nozzle with prescribed receiver pressure.
As far as we know, it is the first mathematical result on the three-dimensional transonic shock with either large vorticity or large swirl velocity.
One of the key difficulties is the fact that the Euler system is elliptic-hyperbolic composite for the flow behind the shock front, and its elliptic part and hyperbolic part are strongly coupled in the lower order terms because of the large swirl velocity, such that they cannot be simply decoupled in the principal parts as the case for the flow without swirls or with small swirl velocity.
New decomposition techniques for the elliptic-hyperbolic composite system are developed to deal with this difficulty, and the solvability condition for the boundary value problem is deduced to determine the location of the shock front.
It finally turns out that the non-zero swirl velocity, which brings new challenging difficulties in the analysis, plays an essential and fundamental role in determining the location of the shock front.
Another key difficulty in the analysis is that there are no readily established shock solutions available, which is different from the flows without swirls, for which trivial piecewise constant shock solutions can be easily constructed.
Non-trivial special shock solutions are first constructed as background solutions under the assumption that the flow parameters depend only on the radial distance to the symmetric axis.
Necessary and sufficient conditions for the existence of such solutions are also found.
Even though the shock location can be arbitrarily shifted for the special shock solutions, it will be shown that the shock solution with a determined shock position can be established as the boundary data are perturbations of one of the established special shock solutions under certain conditions.

\end{abstract}

\maketitle
\tableofcontents{}

\section{Introduction}
In this paper, we are concerned with the existence and shock location of the transonic shock solutions for steady compressible three-dimensional (3-d) axisymmetric full Euler flows with large vorticity via large swirl velocity in a finite cylindrical nozzle with prescribed receiver pressure at the exit (see Figure \ref{fig:1}).
The steady compressible 3-d Euler flow is governed by the following system
\begin{align}
&\di(\rho \mathbf{u})=0,\label{E1}\\
&\di(\rho \mathbf{u}\otimes \mathbf{u})+\nabla p=0,\label{E2}\\
&\di(\rho(e+\frac12 |\mathbf{u}|^2+\displaystyle\frac{p}{\rho})\mathbf{u})=0,\label{E3}
\end{align}
where  $\mb{x}\defs(x_1,x_2,x_3)$ are the space variables, ``$\di $'' is the divergence operator with respect to $\mb{x}$, $\mathbf{u}=(u_1,u_2,u_3)^\top $ is the velocity field, $\rho$, $p$ and $e$ stand for the density, pressure, and the internal energy respectively.
Moreover, the fluid is assumed to be polytropic, which satisfies the following thermal relation
\[
p= A e^{\frac{s}{c_v}}{\rho}^{\gamma},
\]
where $s$ is the entropy, $\gamma>1$ is the adiabatic exponent, $c_v$ is the specific heat at constant volume and $A$ is a positive constant.

It is a longstanding open problem to establish the existence of transonic shock solutions to the system \eqref{E1}-\eqref{E3} in a finite 3-d nozzle and catch the location of the shock front under the boundary conditions proposed by Courant-Friedrichs in \cite{CR} as follows:
\begin{quotation}
	{For given supersonic states at the nozzle entry and receiver pressure at the nozzle exit, try to determine the steady flow pattern with a single transonic shock front, 
including the location of the shock front as well as the states of the flow ahead of and behind it, in a finite 3-d nozzle, satisfying the slip boundary condition on the nozzle wall.}
\end{quotation}

Therefore, roughly speaking, we will impose the following boundary conditions in this paper:
\begin{equation}\label{bou:rough}
\begin{array}{l}
\mbox{all incoming states at the entrance,}\\
\mbox{the slip boundary condition on the nozzle wall,}\\
\mbox{and receiver pressure at the outlet.}
\end{array}
\end{equation}
The shock problem is mathematically formulated as a free boundary problem to the Euler system \eqref{E1}-\eqref{E3} under the boundary conditions \eqref{bou:rough} with prescribed data, which will be specified in detail in Section 2.
The formulated free boundary problem is a strong nonlinear problem, and the location of the shock front is the very free boundary to be determined together with the solution.
In this paper, the flow is assumed to swirl around the symmetric axis with a large speed, and we are going to determine the location of the shock front under the prescribed boundary data and establish the existence of the shock solution to the formulated free boundary problem.

\begin{figure}[!ht]
	\centering
	\includegraphics[width=0.8\textwidth]{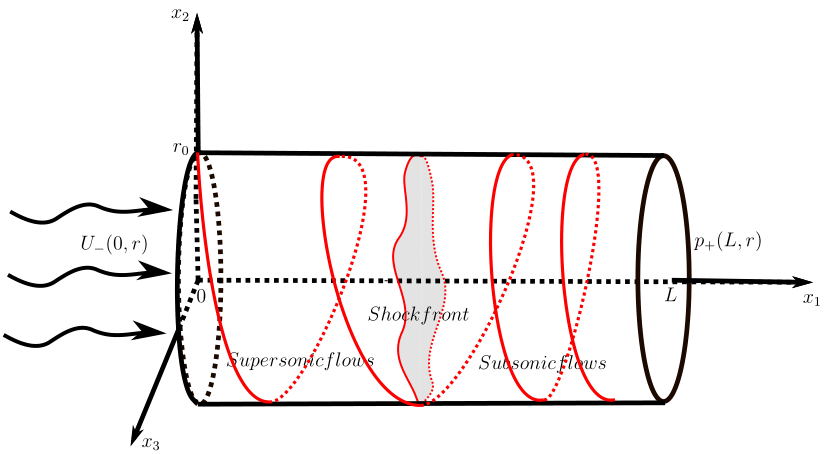}
	\caption{Transonic shock flow with large swirl velocity in a cylinder.\label{fig:1}}
\end{figure}
Even though there has been lots of literature on the existence of steady axisymmetric transonic shock solutions with zero-swirl velocity or with sufficiently small swirl velocity in a 3-d cylindrical nozzle (for instance, see  \cite{
FG1,LiXY2010,LiXinYin2011PJM,ParkRyu_2019arxiv,Yong,HPark,WS94} and references therein),
as far as we know, there is no rigorous result on the transonic flow with general swirl velocity.
In fact, it is still an open problem whether there exist shock solutions with general swirl velocity under the boundary conditions proposed by Courant-Friedrichs in \cite{CR}, and how we determine the admissible shock locations if such shock solutions exist.
This paper is devoted to this problem by establishing the existence of axisymmetric shock solutions to the system \eqref{E1}-\eqref{E3} with large swirl velocity in a finite cylindrical nozzle (see Figure \ref{fig:1}), where the location of the shock front is determined by the boundary data \eqref{bou:rough}.
In this paper, two major steps will be carried out as follows to deal with key difficulties and establish the existence of the shock solution with a determined location of the shock front.
\begin{enumerate}
\item[$\underline{Step\ 1:}$] The first step is to establish necessary and sufficient conditions for the existence of special shock solutions to the system \eqref{E1}-\eqref{E3} with large swirl velocity under the assumption that the states of the flow both ahead of and behind the shock front depend only on the radial distance to the symmetry axis.

One difficulty arising immediately from the large swirl velocity is that there is no readily established background shock solution available, which is different from the flows without swirls, for which trivial piecewise constant shock solutions can be easily established.
Such shock solutions cannot be trivially established because, across the shock front, the flow may fail to satisfy the boundary conditions proposed above and the Rankine-Hugoniot conditions derived from the Euler system \eqref{E1}-\eqref{E3} simultaneously.
Moreover, the shock front location of these special solutions can be arbitrarily shifted in the nozzle, so that it cannot be used as a background solution directly.
    
\item[$\underline{Step\ 2:}$] Based on the special shock solutions established in \emph{Step $1$}, the second step is devoted to determining the location of the shock front and establishing the existence of the shock solution to the system \eqref{E1}-\eqref{E3} for steady axisymmetric flows under small perturbation of the states at the nozzle entry as well as the receiver pressure at the exit.
\begin{itemize}
\item Since the location of the shock front for the special shock solution can be arbitrarily shifted, we have to determine the location of the shock front for the perturbed flow directly from the perturbed boundary data.
Actually, it is a longstanding mathematical challenge to find plausible mechanisms to determine the location of the shock front.
\item Because of the large swirl velocity, the elliptic part and the hyperbolic part of the Euler system for the flow behind the shock front are strongly coupled in the lower order terms, so that they cannot simply be decoupled in the principal parts as the case for the flow without swirl velocity or with small swirl velocity.
Hence, classical theory for elliptic PDEs and hyperbolic PDEs cannot be directly employed to determine the location of the shock front and solve the free boundary problem.
\item  The singularity along the symmetry axis brings additional difficulties in solving the free boundary problem.
\end{itemize}
New decomposition techniques for the Euler system will be developed in this paper to deal with these difficulties. The solvability conditions and approximate solutions for the boundary value problem of the elliptic-hyperbolic composite system will be derived. 
It turns out that, for the perturbation problem in this step, the non-zero swirl velocity of the unperturbed shock solution plays an essential role and contributes the key terms in determining the location of the shock front.
This provides an alternative mechanism, different from the one observed in \cite{FB63}, where the geometry of the nozzle boundary plays a key role.
\end{enumerate}

To overcome these challenging difficulties, in \emph{Step $1$}, we find that the Mach number function of the incoming flow ahead of the shock front plays an important role, and the special shock solutions can be constructed if and only if the Mach number function satisfies certain conditions.
In \emph{Step $2$}, approximate shock solutions (see \S\ref{sec:6.1}) are constructed to capture the location of the shock front directly from the perturbed boundary data. 
The key step is to solve the free boundary problem of the steady Euler system for the flow behind the shock front, which is an elliptic-hyperbolic composite system.
However, the large swirl velocity of the unperturbed flow, strongly coupled nonlinear lower-order terms, and the singularity along the symmetric axis bring new challenging difficulties that the methods and techniques developed in \cite{FB63,FG1} for the shock problem for the flow without swirl velocity cannot deal with.

As far as we know, there is no general mathematical theory for such problems, and we have to develop techniques dealing with them.
By observing that the deduced elliptic-hyperbolic composite system has special structural properties, we develop a partly-decouple technique by introducing two auxiliary problems described as follows.
One auxiliary problem is a boundary value problem of an elliptic system of first order, which is reduced from the elliptic part of the whole composite system by taking out the coupling $0$-th order terms.
In order to make sure this problem admits a solution, the boundary data should satisfy a solvability condition.
This solvability condition, where the large swirling velocity contributes essentially, serves to catch the location of the shock front.
It shows a different mechanism in determining the location of the shock front from the ones in \cite{FB63,FG1}.
The other auxiliary problem is reduced by subtracting the obtained solution to the previous auxiliary problem from the problem for the whole composite system with the strongly coupled non-local terms.
Techniques and methods are developed to deal with the difficulty coming from these coupling non-local terms such that the unique existence of the solution as well as the \textit{a priori} estimates can be established.

For the singularity along the symmetry axis, firstly, the solution to the steady Euler system for the flow ahead of the shock front cannot be directly considered as readily solved by applying the classical theory of hyperbolic systems because of the singularities along the symmetric axis, which is a part of the boundary.
Techniques are developed to solve initial-boundary value problems of hyperbolic systems with singularities on a part of the boundary.
Moreover, the singularity along the symmetry axis also brings difficulties for both auxiliary problems introduced above to deal with the boundary value problem of the elliptic-hyperbolic system for the flow behind the shock front.
We develop a method by considering the problems in domains of higher dimensions to overcome them.
Careful estimates are established to deal with the singularities of the solution along the symmetric axis.

Thanks to steady efforts made by many mathematicians, there have been plenty of results on gas flows with shocks in nozzles, for instance, see \cite{BF,Chen_S2008TAMS,GCM2007,GK2006,GM2003,GM2004,
GM2007,SC2005,SC2009,CY2008,FG1,FG3,FHXX,FL,FB63,FXX,FX,LXY2009,
LiXY2010,
LiXinYin2011PJM,
LXY2013,LiuYuan2009SIAMJMA,LiuXuYuan2016AdM,ParkRyu_2019arxiv,Yong,HPark,WS94, WengXin,
XinYanYin2011ARMA,XinYin2005CPAM,XinYin2008JDE,
Yuan2012NA,HZ1} and references therein.
For steady multi-dimensional flows with a shock in nozzles, in order to determine the position of the shock front, Courant and Friedrichs pointed out in \cite{CR} that, without rigorous mathematical analysis, additional conditions are needed to be imposed at the exit of the nozzle and the pressure condition is preferred among different possible options.
From then on, many mathematicians have been working on this issue and there have been many substantial progresses.
In particular, Chen-Feldman proved in \cite{GM2003} the existence of transonic shock solutions in a finite flat
nozzle for multi-dimensional potential flows with given potential value at the exit and an assumption that the shock front passes through a given point. Later in \cite{GK2006}, with given vertical component of the velocity at the exit, Chen-Chen-Song established the existence of the shock solutions for 2-d steady Euler flows. Both existence results are established under the assumption that the shock front passes through a given point.
Without this assumption and with a given pressure at the exit, recently in \cite{FB63}, Fang-Xin established the existence of the transonic shock solutions in an almost flat (but non-flat) nozzle. It is interesting that the results in \cite{FB63} indicate that, in a flat nozzle and for given pressure condition at the exit, there may exist more than one admissible positions of the shock front. So it is still open to determine the shock location in a flat nozzle. This is one of the main reason that we work on the flat nozzle in this paper.
For the problems in a diverging nozzle,
the position of the shock front can be uniquely determined for a given pressure at the exit under the assumption that the flow states depend only on the radius (see \cite{CR}). And the structural stability of this shock solution
has been established for 2-d case  in \cite{SC2009} by Chen and in \cite{LXY2009,
LXY2013} by Li-Xin-Yin.
For 3-d axisymmetric case, in \cite{WS94}, Weng-Xie-Xin proved the structure stability of the transonic shock solutions under a small perturbation of the swirl velocity; in \cite{Yong}, Park proved the structure stability under a small perturbation of the supersonic flow. Finally, Fang-Gao in \cite{FG1} established the existence of the 3-d axisymmetric flow in a small perturbed cylinder with zero-swirl velocity. 
See also \cite{
DWX,DW} for the 3-d compressible axisymmetric smooth Euler flow with non-zero swirl velocity in a nozzle and \cite{HKWX2021} for the stability of two dimentional transonic contact discontinuity in a finite nozzle.

Finally, for convenience of the reader, we briefly describe the structure of the paper.
In Section 2, the mathematical formulation of the problem for the existence of shock solutions for steady compressible 3-d axisymmetric full Euler flows with non-zero swirl velocity will be described.
Then the previously mentioned two steps, which will be carried out in this paper to establish the existence of the shock solutions with determined location of the shock front, will be described with more details and main theorems will be stated.
Section 3 is devoted to carrying out \emph{Step 1} to establish the existence and uniqueness of special shock solutions with non-zero swirl velocity.
Moreover, the properties of these established solutions will also be described.
Based on the special solutions established in Section 3, the remaining sections are devoted to carrying out \emph{Step 2}.
Several preliminary theorems will be firstly established in Section 4 and Section 5, which deal with the previously mentioned challenging difficulties in \emph{Step 2}. 
More precisely,
in Section 4,
we develop techniques to solve the initial-boundary value problem of the hyperbolic system with singularities on a part of the boundary, and then establish the existence of the solution for the axisymmetric Euler flow ahead of the shock front.
In Section 5,
we develop techniques to deal with the boundary value problem of the elliptic-hyperbolic composite system deduced from the linearized Euler system for the flow behind the shock front and establishes the existence of the solutions to the associating linearized problems.
With the help of the preliminary theorems in Section 4 and Section 5, we establish the existence of the shock solution with determined location of the shock front for the axisymmetric Euler flow with large swirl velocity under certain sufficient conditions and conclude the proof of the main theorem in this paper in Section 6.

\section{Mathematical formulation and main theorems}
In this section, we first formulate the transonic shock problem in a finite cylinder as a free boundary problem for the steady 3-d axisymmetric full Euler flows.
Then we state several lemmas, theorems and finally the main theorem one by one.

\subsection{Free boundary value problem for 3-d steady axisymmetric full Euler flows in a finite cylinder}

Let the nozzle be defined by
\begin{equation}
\mathcal {N}: = \left\{(x_1,x_2,x_3)\in {\mathbb{R}}^3: 0 < x_1 < L,\  0 \leq x_2^2 + x_3^2  < r_0^2\right\},
\end{equation}
with the entrance $ \mathcal{N}_{en} $, the exit $ \mathcal{N}_{ex}$, and the nozzle wall $ \mathcal{N}_w $ (see Figure \ref{fig:1})
\begin{equation*}
\begin{split}
\mathcal{N}_{en} &\defs \overline{\mathcal{N}} \cap \{(x_1,x_2,x_3)\in {\mathbb{R}}^3 : x_1 = 0\},\\
\mathcal{N}_{ex} &\defs \overline{\mathcal{N}} \cap \{(x_1,x_2,x_3)\in {\mathbb{R}}^3 : x_1 = L \}, \\
\mathcal{N}_w &\defs \overline{\mathcal{N}} \cap \{(x_1 ,x_2 ,x_3 )\in {\mathbb{R}}^3 : x_2^2 + x_3^2 = r_0^2\}.
\end{split}
\end{equation*}
Let $ (z, r, \varpi) $ be the cylindrical coordinate in $ \Real^{3} $ such that
\begin{align*}
	(x_1, x_2, x_3) = (z, r \cos \varpi, r\sin \varpi).
\end{align*}
Then the domain $ \mathcal {N} $ and its boundaries can be expressed in $(z,r)$-coordinates as  
\begin{equation}\label{eq801}
	\Omega:  = \{ (z, r)\in \mathbb{R}^2 : 0 < z < L,\ 0 \leq r < r_0 \}
\end{equation}
and
\begin{align*}
	&\Gamma_{en} = \{ (z, r)\in \mathbb{R}^2 :  z = 0,\,0 \leq r < r_0 \},\\
	&\Gamma_{a}= \{ (z, r)\in \mathbb{R}^2 : 0 < z < L,\, r = 0\},\\
	&\Gamma_{ex}= \{ (z, r)\in \mathbb{R}^2 :  z = L,\,0 \leq r < r_0\},\\
	&\Gamma_w = \{ (z, r)\in \mathbb{R}^2 : 0 < z < L,\,r = r_0\}.
\end{align*}
In the cylinder $ \mathcal {N} $, we expect the flow is axisymmetric, that is, if we define
\begin{align*}
	u &\defs u_1,\\
	v &\defs u_2 \cos \varpi + u_3 \sin \varpi,\\
	w &\defs u_3 \cos \varpi - u_2 \sin \varpi,
\end{align*}
then $(u,v,w)$ and $(p,\rho)$ are independent on $\varpi$. Here $w$ is called the swirl velocity.
Then the steady 3-d full Euler system \eqref{E1}-\eqref{E3} becomes the following steady axisymmetric Euler system for $(u,v,w,p,\rho)$,
\begin{align}
	&\partial_z(r \rho u)  + \partial_r (r \rho v) = 0,\label{CE1}\\
	& \partial_z (p+ \rho u^2) + \partial_r (\rho u v) + \frac{\rho u v}{r} =0,\label{CE2}\\
	& \partial_z (\rho u v) + \partial_r (p+\rho v^2)+\frac{\rho (v^2 -w^2)}{r} =0,\label{CE3}\\
	& \partial_z (\rho u w) + \partial_r (\rho v w) + 2\frac{\rho v w}{r} =0,\label{CE4}\\
	& \partial_z \Big(r \rho u (\frac12  (q^2 +w^2)  + i)\Big) +\partial_r \Big(r \rho v (\frac12  (q^2 +w^2)  + i)\Big) =0,\label{CE5}
\end{align}
where
$q^2 \defs u^2 + v^2$ and $i\defs \displaystyle\frac{\gamma p}{(\gamma -1)\rho}$ are the speed modulo the swirl velocity and enthalpy, respectively.
Because the flow is axisymmetric, we know that
\begin{equation}\label{bou:sym}
v(z,0)=w(z,0)=0.
\end{equation}

Straightforward computations yield that the vorticity along $x_1$-direction is
$$
\pt{x_2}u_3-\pt{x_3}u_2=\pt{r}w+\frac{w}{r}.
$$
Therefore, for the flow with large swirl velocity concerned in this paper, its vorticity is \emph{large} in general.
Finally, equation $\eqref{CE5}$ can be rewritten as the Bernoulli's law:
\begin{equation}
\partial_z \big(r \rho u B\big) +\partial_r \big(r \rho v B\big) =0,
\end{equation}
where  $B \defs \displaystyle\frac12  (q^2 +w^2) + i$ is the Bernoulli function.

The eigenvalues of equations \eqref{CE1}-\eqref{CE5} are
\begin{equation*}
\begin{array}{lrl}
\lambda_1 &=&\displaystyle\frac{uv}{u^2 - c^2} - \frac{c\sqrt{u^2 + v^2 -c^2}}{u^2 - c^2},\\[3mm]
\lambda_{2,3,4} &=&\displaystyle\frac{v}{u},\\[3mm]
\lambda_5 &= &\displaystyle\frac{uv}{u^2 - c^2} + \frac{c\sqrt{u^2 + v^2 -c^2}}{u^2 - c^2}.
\end{array}
\end{equation*}
Let $c^2 = \displaystyle\frac{\gamma p}{\rho}$ be the sonic speed. $M = \displaystyle\frac{\sqrt{u^2 + v^2}}{c}$ is called \emph{``quasi-Mach number''} in this paper.
When $M >1$,  $ \lambda_1 $ and $\lambda_5$ are real, so the equations \eqref{CE1}-\eqref{CE5} are hyperbolic. We call the flow \emph{``quasi-supersonic flow''}. When $M <1$, $ \lambda_1 $ and $ \lambda_5 $ are a pair of conjugate complex numbers, so the equations \eqref{CE1}-\eqref{CE5} are of elliptic-hyperbolic composited. We call the flow \emph{``quasi-subsonic flow''}.

Let
$\Gamma_s\defs\set{(z, r)\in \mathbb{R}^2 :  z=\varphi(r),\, 0 \leq r < r_0}$
be the shock front position, then the following Rankine-Hugoniot (R.-H.) conditions hold across $ \Gamma_s $:
\begin{align}
&[\rho u]-\varphi'[\rho v]=0,\label{CRH1}\\
&[p+\rho u^2]-\varphi'[\rho u v]=0,\label{CRH2}\\
&[\rho uv]-\varphi'[p+\rho v^2]=0,\label{CRH3}\\
&[\rho uw]-\varphi'[\rho v w]=0,\label{CRH4}\\
&[\frac12 (q^2 + w^2) + i]=0,\label{CRH5}
\end{align}
where $[\cdot]$ denotes the jump of concerned quantities across the shock front $\Gamma_s$. Moreover, it follows from \eqref{CRH1}, \eqref{CRH4} and \eqref{CRH5} that
\begin{align}
[w]=[\frac12 q^2+ i]=0.
\end{align}

Denote the states of the flow by
$$
U \defs (p,\theta, w, q, s)^\top,
$$
where $\theta = \arctan\displaystyle\frac{v}{u}$ is the flow angle and $q = \sqrt{u^2 + v^2}$. Then in light of \eqref{bou:rough} and \eqref{bou:sym} for reasonable boundary conditions, we will study the following free boundary condition governed by 3-d steady axisymmetric Euler equations with a transonic shock.

\medskip
{\bf {$\llbracket \textit{FBP}\rrbracket $}}:
Given a supersonic state $U_-(r)$ at the entrance $ \Gamma_{en}$, and a relatively high pressure $ p_+(r) $ at the exit $ \Gamma_{ex}$, to find a transonic shock solution $ \left( U_{-}(z,r);\ U_{+}(z,r);\ \varphi(r) \right) $ such that
\begin{enumerate}
\item The curve $ \Gamma_s=\set{z=\varphi(r):0\leq r<r_0} $ is the location of the shock front, which separates $ \Omega $ into two subdomains:
		\begin{align}
			& \Omega_- = \{ (z, r)\in \mathbb{R}^2 : 0 < z < \varphi(r), \, 0 \leq r < r_0\},\\
			& \Omega_+ = \{ (z, r)\in \mathbb{R}^2 : \varphi (r) < z < L,\, 0 \leq r < r_0\},
		\end{align}
	where $ \Omega_- $ denotes the region of the \emph{quasi-supersonic} flow ahead of the shock front, and $ \Omega_+ $ is the region of the \emph{quasi-subsonic} flow behind it.
\item $ U_{-}(z,r) $ satisfies \eqref{CE1}-\eqref{CE5} in $\Omega_{-}$, and the following boundary conditions:
		\begin{align}
			& U_{-}(0,r) = U_-(r) &\text{ on } & \Gamma_{en},\label{U-E0r}\\
			& \theta_{-}(z,r_0) = 0 &\text{ on } & \Gamma_{w}\cap\overline{\Omega_-},\\
			& \theta_-(z,0) =w_-(z,0)=0 &\text{ on } & \Gamma_{a}\cap\overline{\Omega_-}.\label{vwaxis}
		\end{align}
\item $ U_{+}(z,r) $ satisfies \eqref{CE1}-\eqref{CE5} in $\Omega_{+}$, and the following boundary conditions:
		\begin{align}
		& p_{+}(L,r) = p_+(r) &\text{ on } & \Gamma_{ex},\label{p+r}\\
		& \theta_{+}(z,r_0) = 0 &\text{ on } & \Gamma_{w}\cap\overline{\Omega_+},\\
		&\theta_+(z,0)=w_+(z,0) =0 &\text{ on } & \Gamma_{a}\cap\overline{\Omega_+}.
		\end{align}	
\item $ \big( U_{-}(z,r);\ U_{+}(z,r);\ \varphi(r) \big) $ satisfies the R.-H. conditions \eqref{CRH1}-\eqref{CRH5} across the shock front $\Gamma_s$.
\end{enumerate}

\begin{rem}
In the definition of {\bf {$\llbracket \textit{FBP}\rrbracket $}}, we say the shock solution is transonic when the flow is \emph{quasi-supersonic} ahead of the shock front and \emph{quasi-subsonic} behind of the shock front, instead of the real supersonic and subsonic flow with respect to real Mach number $\displaystyle\frac{\sqrt{u^2+v^2+w^2}}{c}$, because the type of the equations \eqref{CE1}-\eqref{CE5} depends on whether the quasi-Mach number is larger than $1$ or not and then the shock solutions in this definition are similar as the transonic shock solutions in the two-dimensional steady Euler flows.
\end{rem}

This paper is devoted to establishing the existence of shock solutions
$\big(U_{-}(z,r)$, $U_{+}(z,r)$; $\varphi(r)\big)$ to the problem {\bf {$\llbracket \textit{FBP}\rrbracket $}}, which are a perturbation of special shock solutions with large swirl velocity, and determining their shock locations by carrying out the following steps, which are quickly mentioned in section 1 too.

\subsection{Step 1: Existence of special shock solutions with large non-zero swirl}
As previously described, the first step
is to find special shock solutions only depending on $r$.
Namely, we are going to construct shock solutions $ \big( \overline{U}_{-}(r);\ \overline{U}_{+}(r);\ \overline{\varphi}(r) \big) $, where the location of the shock front $\overline{\varphi}(r)\equiv \bar{z}_{s} $ with $ \bar{z}_{s}\in (0,L) $ being an arbitrary constant, and
\begin{align}
&\overline{U}_-(r) \defs\big(\bar{p}_-(r), \,0, \,\bar{w}_-(r),\, \bar{q}_-(r), \,\bar{s}_-(r)\big)^\top,\label{U--E}\\
&\overline{U}_+(r)\defs \big(\bar{p}_+(r),\, 0,\, \bar{w}_+(r), \, \bar{q}_+(r), \,\bar{s}_+(r)\big)^\top.\label{U++E}
\end{align}

It is easy to check that $ \big( \overline{U}_{-}(r);\ \overline{U}_{+}(r)\big) $ is a solution of the Euler equations \eqref{CE1}-\eqref{CE5} if and only if $ \big( \overline{U}_{-}(r);\ \overline{U}_{+}(r)\big) $ satisfies the following ordinary differential equations
\begin{align}
    &\partial_r \bar{p}_-(r) = \frac{1}{r} \bar{\rho}_-(r) \bar{w}_-^2(r),\label{RHS-}\\
    & \partial_r \bar{p}_+(r) = \frac{1}{r} \bar{\rho}_+(r) \bar{w}_+^2(r).\label{RHS+}
  \end{align}
Moreover, since the shock front is a straight line ($ \bar{\varphi}'(r)\equiv 0 $), the R.-H. conditions $\eqref{CRH1}$-$\eqref{CRH5}$ yield that
\begin{align}
&\bar{p}_+=\Big((1+{\mu}^2)\overline{M}_-^2-{\mu}^2\Big)\bar{p}_-,\label{BCRH1}\\
&\bar{q}_+  = {\mu}^2\Big(\bar{q}_- + \frac{2}{\gamma-1}\frac{\bar{c}_{-}^2}{\bar{q}_-}\Big),\label{BCRH2}\\
&\bar{\rho}_+ =\frac{{\bar{\rho}_- \bar{q}_-}}{\bar{q}_+}=\frac{{\bar{\rho}_- \bar{q}_-^2}}{\bar{c}_{*}^2},\label{BCRH3}\\
&\bar{w}_+ = \bar{w}_-,\label{BCRH4}
\end{align}
where
\begin{align*}
\bar{c}_-^2 = \displaystyle\frac{\gamma \bar{p}_-}{\bar{\rho}_-},\quad
\overline{M}_-^2 = \displaystyle\frac{\bar{q}_-^2}{\bar{c}_-^2}=\frac{q_-^2p_-^{\frac{1}{\gamma}-1}}{\gamma A^{\frac{1}{\gamma}}e^{\frac{s_-}{\gamma c_v}}}, \quad
\bar{c}_{*}^2 = {\mu}^2\Big(\bar{q}_-^2 + \frac{2}{\gamma-1}\bar{c}_-^2\Big),\quad
{\mu}^2=\frac{\gamma-1}{\gamma+1}.
\end{align*}

We will prove the following theorem in section 3.
\begin{thm}\label{BGthm}
Either for given $C^7$-functions $(\bar{w}_-(r), \bar{q}_-(r))$ with $\bar{w}_-(0)=0$ and $\bar{q}_-(r)>0$ and for given boundary values $\overline{M}_-(0)$ and $\bar{p}_-(0)$ on the symmetry axis with $\overline{M}_-(0)>1$ and $\bar{p}_-(0)> 0$, or for given $C^7$-functions $(\bar{p}_-(r), \bar{s}_-(r))$ with $\bar{p}_-(0)>0$ and for given boundary value $\bar{q}_-(0)>0$ on the symmetry axis with $\overline{M}_-(0)>1$,
there exists a unique special piecewise $C^7$-shock solution $\big(\overline{U}_- (r);\ \overline{U}_+ (r);\ \overline{\varphi}(r)\big)$ of the forms \eqref{U--E} and \eqref{U++E} to equations \eqref{RHS-}-\eqref{RHS+} in domain $\Omega$ with R.-H. conditions \eqref{BCRH1}-\eqref{BCRH4} up to an arbitrary shift of the shock front location, that is, $ z=\overline{\varphi}(r)\equiv \bar{z}_{s} $ is the location of the shock front  for an arbitrary constant $ \bar{z}_{s}\in (0,L) $.
\end{thm}

\begin{rem}
It is easy to see that solutions $\big(\overline{U}_- (r);\ \overline{U}_+ (r)\big)$ of equations \eqref{RHS-}-\eqref{RHS+} with $\theta_{\pm}=0$ are solutions of the full Euler system \eqref{CE1}-\eqref{CE5}.
\end{rem}

\begin{rem}
In order to establish the existence of quasi-supersonic solutions in Theorem \ref{mainthmU-}, the $C^7$-regularity of background solutions is required.
\end{rem}

It should be pointed out that for the state $\overline{U}_- (r)$ ahead of the shock front satisfying the equation \eqref{RHS-}, one can obtain the state $\overline{U}_+ (r)$ behind the shock front by the R.-H. conditions \eqref{BCRH1}-\eqref{BCRH4}. But such state $\overline{U}_+ (r)$ does not necessarily satisfies the equation \eqref{RHS+}. Therefore, we need to derive an equation for a quantity which is equivalent to equation \eqref{RHS+} if \eqref{RHS-} and \eqref{BCRH1}-\eqref{BCRH4} hold.
Fortunately, we find that $\displaystyle\frac{1}{\overline{M}_-^2(r)}$ plays a key role
in finding the special shock solution, as showed in the following lemma.
\begin{lem}\label{lem:1}
Assume \eqref{RHS-} and \eqref{BCRH1}-\eqref{BCRH4} hold, then for given $C^7$-functions $(\bar{w}_-(r), \bar{q}_-(r))$, $\overline{U}_+ (r)$ satisfies the equation \eqref{RHS+} if and only if $\displaystyle\frac{1}{\overline{M}_-^2(r)}$ satisfies the following equation
\begin{align}\label{eq:m}
\partial_r \Big(\frac{1}{\overline{M}_-^2(r)}\Big)
=\frac{\gamma-1}{4\bar{q}_-^2(r)}\frac{\bar{w}_-^2(r)}{r}\Big( \frac{(\gamma +1)^2}{\gamma-1 + \frac{2}{\overline{M}_-^2(r)}} - \big(\gamma-1 + \frac{2}{\overline{M}_-^2(r)}\big)\Big).
\end{align}
\end{lem}
Thus, for given $C^7$-functions $\bar{w}_-(r)$ and $\bar{q}_-(r)$, \eqref{eq:m} is a closed system for $\overline{M}_-(r)$. Combining with \eqref{RHS-} and \eqref{BCRH1}-\eqref{BCRH4}, one can determine $C^7$-states $\bar{p}_-(r)$, $\bar{\rho}_-(r)$ and $\overline{U}_+ (r)$. It is the following lemma. 
\begin{lem}\label{lem:2}
For given $C^7$-functions $(\bar{w}_-(r), \bar{q}_-(r))$ with $\bar{w}_-(0)=0$ and $\bar{q}_-(r)>0$ and for given boundary values $\overline{M}_-(0)$ and $\bar{p}_-(0)$ on the axis with $\overline{M}_-(0)>1$ and $\bar{p}_-(0)> 0$, there exist unique $C^7$-functions $(\bar{p}_-(r), \bar{s}_-(r))$ in the supersonic domain
$\Omega_-^{\bar{z}_s}\defs
\{ (z, r)\in \mathbb{R}^2 : 0 < z < \bar{z}_s,\, 0 < r < r_0\}$,
such that \eqref{RHS-} and \eqref{eq:m} hold. Moreover, $\overline{U}_+ (r)$ given by
 \eqref{BCRH1}-\eqref{BCRH4} satisfies equation \eqref{RHS+}.
\end{lem}

An inverse version of Lemma \ref{lem:2} can also be established by showing the following two lemmas.
\begin{lem}\label{BGlemma3}
If \eqref{RHS-} and \eqref{BCRH1}-\eqref{BCRH4} hold, then for given $C^7$-functions $(\bar{p}_-(r), \bar{s}_-(r))$, $\overline{U}_+ (r)$ satisfies the equation \eqref{RHS+} if and only if $\displaystyle\frac{1}{\overline{q}_-^2(r)}$ satisfies the following equation
   \begin{align}\label{barq-eqB}
    \partial_r\Big( \frac{1}{\bar{q}_-^2}\Big)
  =&  \frac{(\gamma-1)^2}{4} \frac{ \bar{w}_-^2}{r} \frac{\bar{\rho}_- }{\gamma \bar{p}_-}\Big\{-\Big( 1+\frac{2\gamma \bar{p}_-}{(\gamma-1)\bar{\rho}_-}\frac{1}{ \bar{q}_-^2} \Big) + \frac{(\gamma+1)^2}{(\gamma-1)^2\big(1+
  \frac{2\gamma \bar{p}_-}{(\gamma-1)\bar{\rho}_-}\frac{1}{ \bar{q}_-^2}\big)}\notag\\
  &\qquad \qquad\qquad \qquad  - \frac{4}{(\gamma-1)^2}\frac{r}{\bar{w}_-^2}\partial_r \Big(\frac{\gamma \bar{p}_-}{\bar{\rho}_- }\Big) \Big\}\frac{1}{ \bar{q}_-^2}.
   \end{align}
\end{lem}

For given $C^7$-functions $(\bar{p}_-(r),\bar{s}_-(r))$, by \eqref{RHS-}, one can obtain $\bar{w}_-(r)$. Then \eqref{barq-eqB} is a closed system for $\displaystyle\frac{1}{\overline{q}_-^2(r)}$. Therefore one can obtain $\overline{U}_-(r)$ and then $\overline{U}_+(r)$ by showing the following lemma.
\begin{lem}\label{BGlemma4}
For given $C^7$-functions $(\bar{p}_-(r), \bar{s}_-(r))$ with $\bar{p}_-(0)>0$ and for given boundary value $\bar{q}_-(0)>0$ on the axis with $\overline{M}_-(0)>1$, there exist unique $C^7$-functions $(\bar{w}_-(r), \bar{q}_-(r))$ in $\Omega_-^{\bar{z}_s}$
such that \eqref{RHS-} and \eqref{barq-eqB} hold. Moreover, $\overline{U}_+ (r)$ given by \eqref{BCRH1}-\eqref{BCRH4} satisfies equation \eqref{RHS+}.
\end{lem}
The proof of Lemmas \ref{lem:1}-\ref{BGlemma4} will be given in section 3. Moreover, Theorem \ref{BGthm} follows easily from Lemmas \ref{lem:1}-\ref{BGlemma4}.

\subsection{Step 2: Existence of small perturbed shock solutions with prescribed receiver pressure.}
Based on the special shock solutions established in \emph{Step 1}, the following steps are to study the problem {\bf {$\llbracket \textit{FBP}\rrbracket $}} by perturbing the special shock solutions.
More precisely, in problem {\bf {$\llbracket \textit{FBP}\rrbracket $}}, we take $U_-(r)$ in \eqref{U-E0r} and $p_+(r)$ in \eqref{p+r} as
\begin{align}
 U_-(r) \defs& \overline{U}_- (r) + \sigma U_{en}(r),\label{UenE}\\
 p_+(r)\defs& \bar{p}_+(r)+ \sigma p_{ex}(r),\label{penreq}
\end{align}
where $\sigma>0$ is a constant and $U_{en}(r)\defs  (p_{en}(r), 0, w_{en}(r), q_{en}(r), s_{en}(r))^\top$. In addition, we require that for $ r\in (0, r_0)$
\begin{align}\label{requireBG}
\partial_r \Big(\bar{p}_-(r)+\sigma p_{en}(r)
\Big)
\neq\frac{1}{r}\Big(
\frac{\bar{p}_-(r)+\sigma p_{en}(r)}{A}\Big)^{\frac{1}{\gamma}}e^{-\frac{\bar{s}_-(r) +\sigma s_{en}(r)}{\gamma c_v}}
\big( \bar{w}_-(r) + \sigma  w_{en}(r)\big)^2, 
\end{align}
which means that $\overline{U}_-(r) + \sigma U_{en}(r)$ is not a special supersonic solution because it does not satisfy \eqref{RHS-}.

By Lemma \ref{lem:1} - Lemma \ref{BGlemma4}, we see that for given $(\bar{p}_-(r), \bar{s}_-(r))$, there exists a unique $(\bar{w}_-(r),\bar{q}_-(r))$, and vice versa.
So if $(p_{en}(r), s_{en}(r))\neq (0,0)$, by Lemma \ref{BGlemma4}, one can obtain a unique $\big(w_-(0,r), q_-(0,r)\big)$ that satisfies \eqref{CE1}-\eqref{CE5}. Based on the explicit formula of the solutions given in the proof of Lemma \ref{BGlemma4} in section 3, $\big(w_-(0,r), q_-(0,r)\big)$ is also a small perturbation of $(\bar{w}_-(r),\bar{q}_-(r))$ when $\sigma$ is small.
Therefore, without loss of generality, in this paper, we consider
\begin{align}\label{Uenwq}
 U_{en}(r)=(0,0,w_{en}(r),q_{en}(r), 0)^\top.
\end{align}
Then requirement \eqref{requireBG} is equivalent to
\begin{equation}\label{assumew0r0=1}
w_{en}(r)\neq0 \qquad\mbox{for some }r\in(0,r_0).
\end{equation}
In addition, we impose the following conditions related to boundary conditions and compatibility conditions
\begin{align}
  &\bar{w}_-(0) =0,\quad  \bar{w}_-(r_0)=0,\label{assumew0r0=0}\\
  &\partial_r \bar{q}_-(0) =0, \quad \partial_r\bar{q}_-(r_0) =0,\label{assumebarq}\\
&  \partial_r^k \bar{w}_-(0)=0, \, k=2,4,6, \quad \partial_r^j \bar{q}_-(0)=0,\, j=3,5,7,\label{assumption23=0}\\
&\partial_r^k w_{en} (0) =w_{en}(r_0) =0, \qquad k=0,2,4,6,\\
 &\partial_r^j q_{en}(0) = \partial_r q_{en}(r_0) =0,\qquad j=1,3,5,7, \label{assumewenqen}\\
 &\partial_r p_{ex} (0) = \partial_r p_{ex} (r_0) = 0.\label{assumepex}
 \end{align}
We will consider the transonic shock solutions of problem {\bf {$\llbracket \textit{FBP}\rrbracket $}} with conditions \eqref{UenE}-\eqref{penreq} and \eqref{Uenwq}-\eqref{assumepex}.

\subsubsection{Existence of the quasi-supersonic solution.}
In dealing with the free boundary problem, the first step is to establish the unique existence of the quasi-supersonic solution ahead of the shock front.
\begin{thm}\label{mainthmU-}
Assume that $(w_{en}, q_{en})\in C^7([0,r_0])^2$, \eqref{assumew0r0=1}-\eqref{assumewenqen} hold, and the equations \eqref{CE1}-\eqref{CE5} and boundary conditions \eqref{U-E0r}-\eqref{vwaxis} satisfy the compatibility conditions up to the seventh order, then there exists a sufficiently small positive constant $\sigma_L$ depending on $\overline{U}_-$ and $L$, such that for any $0<\sigma< \sigma_L$, there exists a unique solution $U_-\in C^3(\overline{\Omega})$ to the equations \eqref{CE1}-\eqref{CE5} with the initial-boundary value conditions \eqref{U-E0r}-\eqref{vwaxis}. Moreover, the following estimate holds
 \begin{align}\label{eq837}
\|U_- -\overline{U}_- \|_{C^3(\overline{\Omega})} \leq  C_- \sigma\Big(\|w_{en}\|_{H^7(\Gamma_{en})}   + \|q_{en}\|_{H^7(\Gamma_{en})} \Big),
\end{align}
where the positive constant $C_-$ only depends on $\overline{U}_-$, $L$, $r_0$ and $\gamma$.
\end{thm}
\begin{rem}
Due to the boundary conditions on the shock for the quasi-subsonic solution behind the shock and to determine the shock location for approximate solutions later, we need at least $C^{3}$-regularity of solution $U_-$ in the quasi-supersonic domain. We study the hyperbolic system by energy estimates in dimension five to remove the symmetry axis singularity. So $H^7$-regularity of $(w_{en}, q_{en})$ is required. See section 4 for details of the proof.
\end{rem}

\subsubsection{Approximate shock position, and the quasi-subsonic solution behind the free shock front.}
One of the main difficulties is to determine the position of the shock front, while the shock position for the special shock solution is arbitrary.
In order to overcome it, 
we will introduce a specific free boundary problem by linearizing the Euler system and boundary conditions around the special solutions and ignoring the higher order error terms in section 6.
Most importantly, the shock location can be uniquely determined in the specific free boundary problem via the solvability condition, under certain sufficient conditions.
This solution and its shock front, which we call the approximate shock solution and the approximate shock front, will be served as the first order approximation of the solutions and the shock front of the nonlinear problem for the perturbed data.
Then in section 6, a nonlinear iteration scheme is designed around this approximate shock solution with a higher order error of the perturbation, which leads us to the existence of the perturbed shock solution with the shock front being close to the approximate shock front.
Actually, to determine the approximate shock front position $\dot{z}_*$, 
in section 6.1, we introduce
\begin{align*}
 {I}_1({z}) \defs&  \int_0^{r_0}  r\bar{p}_+^{\frac{1}{\gamma}}(r) ( 1- \kappa_1(r)) \frac{\overline{M}_-^2(r)-1}{\bar{\rho}_-(r)\bar{q}_-^2(r)}\dot{p}_- ({z}, r)\dif r 
 +  \int_0^{r_0} \int_0^{{z}} r\bar{p}_+^{\frac{1}{\gamma}}(r)\kappa_3(r) \dot{\theta}_-(\tau,r)\dif \tau \dif r,\\
I_2\defs&  \int_0^{r_0} \frac{\overline{M}_+^2(r)-1}{\bar{\rho}_+(r)\bar{q}_+^2(r)}  r\bar{p}_+^{\frac{1}{\gamma}}(r) p_{ex}(r)\dif r- \int_0^{r_0}  r\bar{p}_+^{\frac{1}{\gamma}}(r)\kappa_2(r) q_{en}(r)\dif r.
\end{align*}
Here $\dot{p}_-$ and $\dot{\theta}_-$ are the solution of the linearized Euler system for quasi-supersonic flow (see \eqref{CDEI-}-\eqref{CDEV-}), and $\kappa_i(r)$ for $i=1,2,3$ are the quantities only depending on the background solutions $\overline{U}_\pm(r)$ (see \eqref{kappa1r}-\eqref{kappa3r}). Then as discussed in section 6.1, the solvability condition \eqref{solvability0} is equivalent to
\begin{equation}\label{eq0661}
I_1(\dot{z}_*) = \sigma I_2.
\end{equation}
Actually, $\dot{z}_* \in (0,L)$ is determined by the following lemma.
\begin{lem}\label{lemma:approxz*}
Assume that $(w_{en}, q_{en})\in C^7([0,r_0])^2$, $ p_{ex}(r)\in C^{2,\alpha}([0,r_0])$, \eqref{assumew0r0=1}-\eqref{assumepex} hold, the equations \eqref{CE1}-\eqref{CE5} and boundary conditions \eqref{U-E0r}-\eqref{vwaxis} satisfy the compatibility conditions up to the seventh order, and
\begin{align}
I_3\defs &\int_0^{r_0}\frac{i_1(r)}{\bar{q}_-^2(r)}\frac{\bar{w}_-(r)}{r} w_{en}(r)\dif r >0,\label{LetW0E>0-1}
\end{align}
where
\begin{align*}
i_1(r) \defs & \bar{p}_+^{\frac{1}{\gamma}}(r) \Big( (1- \kappa_1(r))\big( \frac{ \bar{w}_+^2(r)}{\bar{c}_+^2(r)}- \frac{\bar{w}_-^2 (r)}{\bar{c}_-^2(r)} \big) + r\kappa_3(r) - r\partial_r \kappa_1(r) \Big).
\end{align*}
Then there exist a positive constant $C_{(\overline{U}_-)}$ depending on $\overline{U}_-$, and a constant $L_*$ satisfying
\begin{align}\label{letinterzzr-1}
0 < L_*< \frac{4I_3}{C_{(\overline{U}_-)}( \|w_{en}\|_{H^7(\Gamma_{en})} + \|q_{en}\|_{H^7(\Gamma_{en})} )\int_{0}^{r_0} |i_1(r)|\dif r},
\end{align}
such that if
\begin{align}\label{I2interval}
0< I_2 \leq L_*^2 \Big( I_3 - \frac{ C_{(\overline{U}_-)}( \|w_{en}\|_{H^7(\Gamma_{en})} + \|q_{en}\|_{H^7(\Gamma_{en})} )}{6} L_*\int_{0}^{r_0} |i_1(r)| \dif r \Big),
\end{align}
then there exists a unique $\dot{z}_* \in (0, L_*)$ such that \eqref{eq0661} holds.
\end{lem}
\begin{rem}
It is easy to see that condition \eqref{LetW0E>0-1} fails if $\bar{w}_-(r)=0$, that is the case when small swirl solution is a small perturbation of constant background solutions. Moreover, conditions \eqref{letinterzzr-1} and \eqref{I2interval} will be easier to hold if the non-zero swirl $\bar{w}_-$ of the special shock solution is larger. Therefore, non-zero swirl velocity plays a fundamental role in determining the shock location.
\end{rem}

Based on the obtained approximate shock front position $\dot{z}_*$, we introduce the following coordinates transformation to fix the free boundary $\Gamma_s$
  $$\pounds_{\dot{z}_*}^{-1} : (z,\,r) = \big(L + \frac{L - \varphi(\tilde{r})}{L - \dot{z}_*}(\tilde{z} - L),\,
\tilde{r}\big).$$
In the new coordinates $(\tilde{z},\tilde{r})$, shock front $\Gamma_s$ becomes  $\dot{\Gamma}_s=\{(\dot{z}_*,\tilde{r});\,\tilde{r}\in[0,r_0]\}$, which is fixed.
We define the norms of the function $u$ as below:
\begin{align*}
\|u \|_{C^{2,\alpha}(\Gamma_s)}\defs & \| u \circ \pounds_{\dot{z}_*}^{-1}\|_{C^{2,\alpha}(\dot{\Gamma}_s)},\\
\|u \|_{C^{k,\alpha}({\Omega}_+)}\defs & \| u \circ \pounds_{\dot{z}_*}^{-1}\|_{C^{k,\alpha}(\dot{\Omega}_+)},\, k=1,2.
\end{align*}
Then we will show the main theorem of this paper.

\begin{thm}\label{mainthmU}
Under the assumptions given in Lemma \ref{lemma:approxz*}, there exist a positive constant $r_*$ and a sufficiently small constant $\sigma_0 > 0$, only depending on $\overline{U}_-(r)$ and $\gamma$, such that for $0<r_0 \leq r_*$ and $ 0<\sigma\leq\sigma_0 $, 
there exists a shock solution $(U_-, U_+;\ \varphi(r)) $ to the free boundary value problem {\bf {$\llbracket \textit{FBP}\rrbracket $}}. Moreover,
 \begin{align}
 &\|U_- -\overline{U}_- \|_{C^3(\overline{\Omega}_-)} \leq  C_- \sigma\Big(\|w_{en}\|_{H^7(\Gamma_{en})}   + \|q_{en}\|_{H^7(\Gamma_{en})} \Big),
\end{align}
\begin{align}
 &\|p_+ -\bar{p}_+\|_{C^{2,\alpha}({\Omega}_+)} +\|\theta_+ \|_{C^{2,\alpha}({\Omega}_+)} +\|w_+ - \bar{w}_+ \|_{C^{1,\alpha}({\Omega}_+)}+\|q_+ -\bar{q}_+\|_{C^{2,\alpha}({\Omega}_+)} \notag\\
 &+\|s_+ -\bar{s}_+\|_{C^{2,\alpha}({\Omega}_+)} +\|rw_+ -r\bar{w}_+\|_{C^{2,\alpha}({\Omega}_+)} +\| \varphi'\|_{C^{2,\alpha}(\Gamma_s)}\notag\\
   \leq&  C_+ \sigma \Big( \|w_{en}\|_{H^7(\Gamma_{en})}   + \|q_{en}\|_{H^7(\Gamma_{en})}  + \|p_{ex}\|_{C^{2,\alpha}(\Gamma_{ex})}  \Big),
\end{align}
where the positive constant $C_-$ depends on $\overline{U}_-$, $L$, $r_0$ and $\gamma$, and the positive constant $C_+$ depends on $\overline{U}_{\pm}$, $L$, $r_0$, $\gamma$ and $\alpha$. In addition, the shock location satisfies
\begin{align}
  |\varphi(r_0) -\dot{z}_*| \leq C \sigma,
\end{align}
where the positive constant $C$ depends on $\overline{U}_{\pm}$, $L$, $r_0$, $\gamma$ and $\alpha$.

\end{thm}

\section{Special steady normal shock solutions with non-zero swirl}

In this section, we first prove Lemma \ref{lem:1} - Lemma \ref{BGlemma4}. Then we will state properties of such special shock solutions.

\subsection{Existence of steady special shock solutions}
In this subsection, we will prove Lemma \ref{lem:1} - Lemma \ref{BGlemma4} one by one, and then Theorem \ref{BGthm} follows.

\begin{proof}[Proof of Lemma \ref{lem:1}]
By \eqref{BCRH1}-\eqref{BCRH3}, we have
  \begin{align}
\bar{p}_+=& \Big(\frac{2\gamma}{\gamma+1}\frac{\bar{\rho}_-\bar{q}_-^2}{\gamma \bar{p}_-}-\frac{\gamma-1}{\gamma+1}\Big)\bar{p}_- = \frac{2}{\gamma+1}\bar{\rho}_-\bar{q}_-^2 -\frac{\gamma-1}{\gamma+1}\bar{p}_-,\label{furp+}\\
\bar{\rho}_+ =& \frac{\bar{\rho}_- \bar{q}_-}{{\mu}^2\Big(\bar{q}_- + \frac{2}{\gamma-1}\frac{\bar{c}_{-}^2}{\bar{q}_-}\Big)}  = \frac{(\gamma+1)\bar{\rho}_- }{(\gamma-1)\Big(1 + \frac{2}{\gamma-1}\frac{1}{\overline{M}_-^2}\Big)}.\label{furrho+}
  \end{align}
By applying \eqref{furp+}, \eqref{furrho+} and \eqref{BCRH4}, it follows from \eqref{RHS+} that
\begin{align}\label{RHS+become}
  \frac{2}{\gamma+1} \partial_r (\bar{\rho}_-\bar{q}_-^2) -\frac{\gamma-1}{\gamma+1} \partial_r \bar{p}_-= \frac{\bar{w}_-^2}{r}  \frac{(\gamma+1)\bar{\rho}_- }{(\gamma-1)\Big(1 + \frac{2}{\gamma-1}\frac{1}{\overline{M}_-^2}\Big)}.
\end{align}
Substituting \eqref{RHS-} into \eqref{RHS+become}, one has
\begin{align}\label{rhoqeq}
  \partial_r (\bar{\rho}_-\bar{q}_-^2) =& \frac{\gamma+1} {2}\Big( \frac{\bar{w}_-^2}{r}  \frac{(\gamma+1)\bar{\rho}_- }{(\gamma-1)\big(1 + \frac{2}{\gamma-1}\frac{1}{\overline{M}_-^2}\big)} + \frac{\gamma-1}{\gamma+1} \frac{1}{r} \bar{\rho}_- \bar{w}_-^2 \Big)\notag\\
  =&\frac{\bar{w}_-^2}{2r} \frac{\bar{\rho}_-^2\bar{q}_-^2}{\gamma \bar{p}_-} \Big( \frac{(\gamma+1)^2}{2+(\gamma-1)\overline{M}_-^2} +\frac{\gamma-1}{\overline{M}_-^2} \Big).
\end{align}
Applying \eqref{rhoqeq} and \eqref{RHS-}, one has
\begin{align}\label{BSM-1}
   \partial_r \Big(\frac{1}{\overline{M}_-^2}\Big) =& \partial_r \Big( \frac{\gamma \bar{p}_-}{\bar{\rho}_- \bar{q}_-^2} \Big) =  \frac{\gamma}{\bar{\rho}_-\bar{q}_-^2}\partial_r \bar{p}_- - \frac{\gamma \bar{p}_-}{\bar{\rho}_-^2 \bar{q}_-^4}\partial_r (\bar{\rho}_-\bar{q}_-^2)\notag\\
   =&  \frac{\gamma}{\bar{\rho}_-\bar{q}_-^2}\frac{1}{r} \bar{\rho}_- \bar{w}_-^2 - \frac{\gamma \bar{p}_-}{\bar{\rho}_-^2 \bar{q}_-^4}\frac{1}{2r} \frac{\bar{\rho}_-^2\bar{q}_-^2}{\gamma \bar{p}_-}\bar{w}_-^2 \Big( \frac{(\gamma+1)^2}{2+(\gamma-1)\overline{M}_-^2} +\frac{\gamma-1}{\overline{M}_-^2} \Big)\notag\\
   =&\frac{1}{\bar{q}_-^2}\frac{ \bar{w}_-^2}{r} \Big(\gamma -  \frac{1}{2} \Big( \frac{(\gamma+1)^2}{2+(\gamma-1)\overline{M}_-^2} +\frac{\gamma-1}{\overline{M}_-^2} \Big)   \Big).
 \end{align}
 Further calculations yield \eqref{eq:m} easily.

On the other hand, if \eqref{eq:m} holds, by reversing the argument from \eqref{eq:m} to \eqref{BSM-1}, we know that \eqref{BSM-1} holds. 
Then applying \eqref{BCRH4}, \eqref{furp+} and \eqref{furrho+}, one has
\begin{align}
    \partial_r \bar{p}_+
    =&  \frac{2}{\gamma+1} \partial_r (\bar{\rho}_-\bar{q}_-^2) - \frac{\gamma-1}{\gamma+1}\partial_r \bar{p}_- ,\label{rleft1}\\
     \frac{1}{r} \bar{\rho}_+\bar{w}_+^2 =& \frac{ \bar{w}_-^2}{r} \frac{(\gamma+1)\bar{\rho}_- }{(\gamma-1)\Big(1 + \frac{2}{\gamma-1}\frac{1}{\overline{M}_-^2}\Big)}.\label{rright1}
 \end{align}
By \eqref{BSM-1}, one has
 \begin{align}
   \partial_r (\bar{\rho}_-\bar{q}_-^2) = \frac{\bar{\rho}_-\bar{q}_-^2}{\bar{p}_-} \partial_r \bar{p}_- - \frac{\bar{\rho}_-^2 \bar{q}_-^4}{\gamma \bar{p}_-} \partial_r \Big(\frac{1}{\overline{M}_-^2}\Big).
 \end{align}
Then applying \eqref{RHS-}, \eqref{eq:m}, \eqref{rleft1} and \eqref{rright1}, it is easy to check that \eqref{RHS+} holds.
\end{proof}

\begin{proof}[Proof of Lemma \ref{lem:2}]
Let $ t(r)\defs  \frac{1}{\overline{M}_-^2(r)}$, then \eqref{eq:m} becomes
    \begin{align}\label{rem1}
\partial_r\Big( \frac{t}{\gamma-1} \Big)
    =\frac{\gamma-1}{4\bar{q}_-^2}\frac{\bar{w}_-^2}{r}\Big(   -\big(1+\frac{2t}{\gamma-1} \big)+\frac{1}{\mu^4}\frac{1}{1+\frac{2}{\gamma-1}t} \Big).
\end{align}
Let
  \begin{align}\label{reT}
    T(r)\defs 1+\frac{2}{\gamma-1}t(r).
  \end{align}
  Then \eqref{rem1} yields that
  \begin{align}
    \partial_r T = (\gamma-1)\frac{1}{r} \frac{\bar{w}_-^2}{2\bar{q}_-^2}\Big(   -T+\frac{1}{\mu^4T} \Big) =(\gamma-1)\frac{1}{r} \frac{\bar{w}_-^2}{2\bar{q}_-^2} \frac{1 - \mu^4T^2}{\mu^4T},
  \end{align}
which implies that
  \begin{align}
    \partial_r \Big(\ln (1 - \mu^4T^2) \Big) = -(\gamma-1)\frac{1}{r}\frac{\bar{w}_-^2}{\bar{q}_-^2}.
  \end{align}
Therefore, there exist a unique $T(r)$ of the form:
  \begin{align}
    T(r) =\frac{1}{\mu^2}\sqrt{ 1 - \big( 1 - \mu^4T^2(0) \big)e^{ -(\gamma-1)\int_0^r\frac{1}{\tau}\frac{\bar{w}_-^2(\tau)}{\bar{q}_-^2(\tau)}\dif \tau}},
  \end{align}
where $T(0) = 1+\frac{2}{\gamma-1}\frac{1}{\overline{M}_-^2(0)}\in (1, \frac{\gamma+1}{\gamma-1})$. It follows from \eqref{reT} that
   \begin{align}\label{ret}
    t(r) = \frac{\gamma+1}{2} \Big(\sqrt{ 1 - \big( 1 - \mu^4T^2(0) \big)e^{ -(\gamma-1)\int_0^r\frac{1}{\tau}\frac{\bar{w}_-^2(\tau)}{\bar{q}_-^2(\tau)}
    \dif \tau}}-\mu^2\Big).
  \end{align}
Moreover, the existence is global for any $r>0$, because it follows from $t(0) = \frac{1}{\bar{M}_-^2 (0)} \in (0,1)$ that
  \begin{align}
    1 - \mu^4T^2(0) =&  1- \frac{(\gamma-1)^2}{(\gamma+1)^2}\Big( 1+\frac{2}{\gamma-1}t(0)  \Big)^2\notag\\
    =& 1- \frac{(\gamma-1+2t(0))^2}{(\gamma+1)^2}\in(0,1).
  \end{align}
In addition, \eqref{ret} implies that for any $r\in (0,+\infty)$, $t(r)$ is monotonely increasing with respect to $r$. Hence, we have for any $r\in (0,+\infty)$
 \begin{align}
0< t(0)\leq t(r) < \frac{\gamma+1}{2}(1 - \mu^2) =1.
 \end{align}
Once $t(r)$ is obtained, then employing \eqref{RHS-}, one has
  \begin{align}
    \partial_r \bar{p}_-(r) =  \frac{1}{r}\frac{\gamma \bar{p}_-(r)}{t(r) \bar{q}_-^2(r)} \bar{w}_-^2(r).
  \end{align}
Then 
\begin{align}\label{barp-}
\bar{p}_-(r) = \bar{p}_-(0) \exp\Big(\gamma\int_0^r  \frac{\bar{w}_-^2(\tau)}{\tau t(\tau)\bar{q}_-^2(\tau)} \dif \tau  \Big),
\end{align}
where we use the conditions that $\bar{w}_-(0)=0$ and $\bar{w}_-$ is smooth.
Then
\begin{align}\label{barrho--}
\bar{\rho}_-(r) = \frac{\gamma \bar{p}_-(r)}{t(r) \bar{q}_-^2(r)}=\frac{\gamma \bar{p}_-(0)}{t(r) \bar{q}_-^2(r)}\exp\Big(\gamma\int_0^r  \frac{\bar{w}_-^2(\tau)}{\tau t(\tau)\bar{q}_-^2(\tau)} \dif \tau  \Big).
 \end{align}
By $\bar{p}_- = A e^{\frac{\bar{s}_-}{c_v}}\bar{\rho}_-^{\gamma}$, we have $\bar{s}_- = {c_v}\ln\big(  \frac{\bar{p}_-}{A\bar{\rho}_-^{\gamma}}\big)$.

Finally, $C^7$-regularity of the solutions follows easily from the solution expressions.
\end{proof}

\begin{proof}[Proof of Lemma \ref{BGlemma3}] Based on Lemma \ref{lem:1}, it is sufficient to show that \eqref{eq:m} is equivalent to \eqref{barq-eqB}. By a similar argument as done in the proof of Lemma \ref{lem:1}, we can see that both \eqref{eq:m} and \eqref{barq-eqB} are equivalent to the following equation
  \begin{align}
   \partial_r \Big(\frac{\gamma \bar{p}_-}{\bar{\rho}_- }\frac{1}{ \bar{q}_-^2}\Big) =\frac{\gamma-1}{4}\frac{\bar{w}_-^2}{r}\Big( \frac{(\gamma +1)^2}{\gamma-1 + \frac{2\gamma \bar{p}_-}{\bar{\rho}_-\bar{q}_-^2}} - \big(\gamma-1 +\frac{2\gamma \bar{p}_-}{\bar{\rho}_- \bar{q}_-^2}\big)  \Big)\frac{1}{\bar{q}_-^2}.
  \end{align}
Since the computation is straightward and long, we omit it for the shortness.
\end{proof}

\begin{proof}[Proof of Lemma \ref{BGlemma4}]
Since quantities $\bar{p}_-(r)$ and $\bar{s}_-(r)$ are given, by the state equation $\bar{p}_-(r) = A e^{\frac{\bar{s}(r)}{c_v}}\bar{\rho}_-^{\gamma}$, one has
  \begin{align}\label{rhovalue}
    \bar{\rho}_-(r) = \Big( \frac{\bar{p}_-(r)}{A} e^{-\frac{\bar{s}_-(r)}{c_v}} \Big)^{\frac{1}{\gamma}}.
  \end{align}
Then it follows from \eqref{RHS-} that
  \begin{align}\label{wvalue}
    \bar{w}_-^2(r) =  \frac{r\partial_r \bar{p}_-(r)}{\bar{\rho}_-(r)} = \frac{  r \partial_r \bar{p}_-(r)}{\Big( \frac{\bar{p}_-(r)}{A} e^{-\frac{\bar{s}_-(r)}{c_v}} \Big)^{\frac{1}{\gamma}} }.
  \end{align}
Let $Q(r)\defs \frac{1}{ \bar{q}_-^2(r)}$ and $ G(r)\defs \frac{\bar{\rho}_-(r)}{\gamma \bar{p}_-(r)}$, then \eqref{barq-eqB} yields that
\begin{align}\label{QReq}
 \partial_r \left(\frac{2Q}{(\gamma-1)G}\right) = \frac{(\gamma-1)\bar{w}_-^2}{2r} \Big\{-\Big( 1+\frac{2Q }{(\gamma-1)G} \Big) + \frac{(\gamma+1)^2}{(\gamma-1)^2\big(1+\frac{2Q}{(\gamma-1)G}\big)}\Big\}.
\end{align}
Let
\[
Y(r) \defs 1+\frac{2Q(r)}{(\gamma-1)G(r)}.
\]
  Then 
  \begin{align}
 \partial_r Y
 =\frac{\gamma-1}{2}\frac{ \bar{w}_-^2}{r}\frac{1 -\mu^4 Y^2 }{\mu^4 Y}.
\end{align}
Direct calculations yield that
\begin{align}
    Y(r) =\frac{1}{\mu^2}\sqrt{ 1 - \big( 1 - \mu^4Y^2(0) \big)e^{  -(\gamma-1)\int_0^r\frac{\bar{w}_-^2(\tau)}{\tau}\dif \tau}},
  \end{align}
 where
 \begin{align}
   Y(0) =  1+\frac{2Q(0)}{(\gamma-1)G(0)} = 1+ \frac{2\gamma A^{\frac{1}{\gamma}}(\bar{p}_-(0))^{\frac{\gamma-1}{\gamma}} e^{\frac{\bar{s}_-(0)}{\gamma c_v}}}{(\gamma-1)\bar{q}_-^2(0)} \in \Big(1, \frac{\gamma+1}{\gamma-1}\Big).
 \end{align}
Hence $Y(r)$ exists uniquely and globally. In addition, $Y(r)$ is monotonically increasing with respect to $r$ such that
$$
1<Y(0)\leq Y(r)<\frac{1}{\mu^2}=\frac{\gamma+1}{\gamma-1}.
$$
  Thus,
 \begin{align}
   Q(r) =& \frac{\gamma-1}{2}(Y(r)-1)G(r)\notag\\
    =& \frac{\gamma+1}{2}\frac{\bar{\rho}_-(r)}{\gamma \bar{p}_-(r)} \Big( \sqrt{ 1 - \big( 1 - \mu^4Y^2(0) \big)e^{  -(\gamma-1)\int_0^r\frac{\bar{w}_-^2(\tau)}{\tau}\dif \tau}} - \mu^2 \Big).
 \end{align}
By \eqref{rhovalue} and \eqref{wvalue}, we can obtain $\bar{q}_-^2(r) = \frac{1}{Q(r)}$.

Finally, the $C^7$-regularity of solutions follows easily from the solution expressions.
\end{proof}

\begin{proof}[Proof of Theorem \ref{BGthm}]
The existence of the special shock solutions is obtained by Lemma \ref{lem:2} and Lemma \ref{BGlemma4}. Note that the special solutions do not depend on $z$, so the shock location $\bar{\varphi}=\bar{z}_s$ can be arbitrary.
\end{proof}

\subsection{Properties of the special shock solutions}
The special transonic shock solution $(\overline{U}_- (r), \overline{U}_+ (r))$ obtained in Theorem \ref{BGthm} actually satisfies the following properties.
\begin{lem}\label{barvalue}
Assume that \eqref{assumew0r0=0} and \eqref{assumebarq} hold, the special transonic shock solution $(\overline{U}_- (r), \overline{U}_+ (r))$ obtained in Theorem \ref{BGthm} satisfies
  \begin{align}
  &\partial_r (\bar{p}_\pm)(0)=\partial_r(\bar{s}_\pm)(0) =\partial_r (\overline{M}_\pm^2 )(0) = 0,\label{pro1}\\
&\partial_r (\bar{p}_\pm)(r_0)=\partial_r(\bar{s}_\pm)(r_0)= \partial_r (\overline{M}_\pm^2)(r_0)=0.\label{pro2}
     \end{align}
Moreover, if \eqref{assumption23=0} holds, then
\begin{align}\label{partialr3=0}
 \partial_r^j (\bar{p}_\pm, \bar{\rho}_{\pm}, \bar{s}_\pm)(0)=0, \quad \partial_r^j \bar{q}_+(0)=0,\quad \partial_r^j ({\overline{M}_\pm^2}) (0) = 0\qquad\mbox{where }j=3,5,7.
\end{align}
\end{lem}

\begin{proof}
Applying \eqref{BCRH4} and \eqref{assumew0r0=0}, we have $\bar{w}_{\pm}(0)=\bar{w}_{\pm}(r_0) = 0$.
Then it follows from \eqref{RHS-} and \eqref{RHS+} that
  \begin{align}\label{barpvalue}
    \partial_r \bar{p}_\pm(0) =\partial_r \bar{p}_\pm(r_0)=0.
  \end{align}
At the same time, it follows from \eqref{eq:m} and the smoothness of $\bar{w}_{-}$ that
\begin{align}\label{r1M0r0}
  \partial_r \Big(\frac{1}{\overline{M}_-^2} \Big)(0) =\partial_r \Big(\frac{1}{\overline{M}_-^2} \Big)(r_0)=0.
\end{align}
So
\begin{align}\label{1rM-==0}
\partial_r \overline{M}_-^2 (0) =\partial_r \overline{M}_-^2(r_0)=0.
\end{align}

By \eqref{BCRH2},
  \begin{align}\label{rq+}
    \bar{q}_+
    = \mu^2\bar{q}_- \Big(1 + \frac{2}{\gamma-1}\frac{1}{\overline{M}_-^2}\Big).
  \end{align}
So it follows from \eqref{assumebarq}, \eqref{r1M0r0} and \eqref{rq+} that
  \begin{align}\label{q+0r0}
     \partial_r \bar{q}_+(0) =\partial_r \bar{q}_+(r_0)=0.
  \end{align}
Moreover, by \eqref{assumebarq}, \eqref{assumew0r0=0} and \eqref{rhoqeq}, we have
\begin{align}\label{barrho-}
  \partial_r\bar{\rho}_-(0) = \partial_r\bar{\rho}_-(r_0) = 0.
\end{align}
Then by \eqref{furrho+} 
\begin{align}\label{rho0}
    \partial_r \bar{\rho}_+(0) = \partial_r \bar{\rho}_+(r_0) =0.
  \end{align}
 So it follows from \eqref{barpvalue}, \eqref{q+0r0} and \eqref{rho0} that
  \begin{align}\label{r1M0r0+}
  \partial_r \Big(\frac{1}{\overline{M}_+^2} \Big)(0) =\partial_r \Big(\frac{1}{\overline{M}_+^2} \Big)(r_0)=0.
\end{align}
So
\begin{align}\label{1rM+==0}
  \partial_r \overline{M}_+^2 (0) =\partial_r \overline{M}_+^2(r_0)=0.
\end{align}
Finally, by the thermal relation $p= A e^{\frac{s}{c_v}}\rho^{\gamma}$, we have
  \begin{align}\label{barsvalue}
    \partial_r s  = c_v \frac{ \partial_r p}{p} - \gamma c_v \frac{\partial_r \rho}{\rho}.
  \end{align}
Therefore, $\partial_r \bar{s}_\pm (0) =\partial_r \bar{s}_\pm (r_0) =0$.

Now, we are going to prove \eqref{partialr3=0} for $j=3$.
By applying $\bar{w}_-(0)=0$ and $\partial_r^2 \bar{w}_-(0)=0$, we have
\begin{align}\label{rrw2r=}
\lim\limits_{r\rightarrow 0} \partial_r^2(\frac{\bar{w}_-^2}{r}) =& \lim\limits_{r\rightarrow 0}2\big( \frac{(\partial_r\bar{w}_-)^2+\bar{w}_-\partial_r^2\bar{w}_-}{r} - \frac{2\bar{w}_- \partial_r \bar{w}_-}{r^2} + \frac{\bar{w}_-^2}{r^3}   \big)\notag\\
=& 2 \partial_r \bar{w}_-(0) \partial_r^2 \bar{w}_-(0) + \frac23 \bar{w}_-(0) \partial_r^3 \bar{w}_-(0)\notag\\
=& 0.
\end{align}
By \eqref{RHS-}, direct calculations yield that
\begin{align}
  \partial_r^3 \bar{p}_-(r) = \frac{\bar{w}_-^2}{r} \partial_r^2 \bar{\rho}_-(r) + 2 \partial_r(\frac{\bar{w}_-^2}{r}) \partial_r \bar{\rho}_- + \bar{\rho}_- \partial_r^2 (\frac{\bar{w}_-^2}{r}).
\end{align}
Therefore, 
\begin{align}\label{p-rrr0=0}
  \partial_r^3 \bar{p}_-(0)=0.
\end{align}
Similarly, by \eqref{RHS+} and $\bar{w}_+(0) = \bar{w}_-(0)=0$, $\partial_r^2 \bar{w}_+(0) = \partial_r^2 \bar{w}_-(0)=0$ and \eqref{rho0}, we have
\begin{align}\label{p+rrr0=0}
  \partial_r^3 \bar{p}_+(0)=0.
\end{align}
Recalling the definition $t(r)= \frac{1}{\overline{M}_-^2(r)}$ and by \eqref{eq:m}, we have
\begin{align}\label{eq:mt}
\partial_r t
=\frac{\gamma-1}{4\bar{q}_-^2}\frac{\bar{w}_-^2}{r}\Big( \frac{(\gamma +1)^2}{\gamma-1 + 2t} - \big(\gamma-1 + 2t\big)\Big).
\end{align}
So
\begin{align}\label{rrrt0=0}
  \partial_r^3 t =& \frac{\gamma-1}{4} \frac{1}{\bar{q}_-^2}\Big( \frac{(\gamma +1)^2}{\gamma-1 + 2t} - \big(\gamma-1 + 2t\big)\Big)\partial_r^2(\frac{\bar{w}_-^2}{r})\notag\\
  & + \frac{\gamma-1}{4}\partial_r\Big(\frac{1}{\bar{q}_-^2}\big( \frac{(\gamma +1)^2}{\gamma-1 + 2t} - \big(\gamma-1 + 2t\big)\big)\Big)\partial_r(\frac{\bar{w}_-^2}{r})\notag\\
  &+ \frac{\gamma-1}{4}\frac{\bar{w}_-^2}{r} \partial_r^2\Big(\frac{1}{\bar{q}_-^2}\big( \frac{(\gamma +1)^2}{\gamma-1 + 2t} - \big(\gamma-1 + 2t\big)\big)\Big)\notag\\
  & + \frac{\gamma-1}{4}\partial_r(\frac{\bar{w}_-^2}{r})\partial_r\Big(\frac{1}{\bar{q}_-^2}\big( \frac{(\gamma +1)^2}{\gamma-1 + 2t} - \big(\gamma-1 + 2t\big)\big)\Big).
\end{align}
It follows from $\bar{w}_-(0)=0$, $\partial_r \bar{q}_-(0)=0$, \eqref{r1M0r0}, \eqref{rrw2r=}, and \eqref{rrrt0=0} that $\partial_r^3 t(0)=0$. That is
\begin{align}\label{r3M-0=0}
  \partial_r^3 \Big(\frac{1}{\overline{M}_-^2}\Big)(0)=0.
\end{align}
Hence
\begin{align}\label{r3M-0==0}
 \partial_r^3 {\overline{M}_-^2}(0)=0.
\end{align}
 By $\bar{\rho}_- = \frac{\gamma \bar{p}_- \overline{M}_-^2}{\bar{q}_-^2}$, it follows from \eqref{pro1}, \eqref{assumption23=0}, \eqref{p-rrr0=0}, \eqref{r3M-0==0}, and \eqref{assumption23=0} that
 \begin{align}\label{rrrrho0==0}
   \partial_r^3 \bar{\rho}_-(0) =0.
 \end{align}
 By \eqref{pro1}, \eqref{barsvalue}, \eqref{p-rrr0=0}, and \eqref{rrrrho0==0}, we have
 \begin{align}
   \partial_r^3 \bar{s}_- (0)=0.
 \end{align}
By \eqref{BCRH2}, further calculations yield that
\begin{align}\label{r3q+0=0}
  \partial_r^3 \bar{q}_+ = \mu^2\Big((1+\frac{2t}{\gamma-1})\partial_r^3 \bar{q}_- + \frac{2\bar{q}_-}{\gamma-1}\partial_r^3 t + \frac{6}{\gamma-1}\partial_r(\partial_r \bar{q}_- \partial_r t)    \Big).
\end{align}
So it follows from  \eqref{pro1}, \eqref{assumption23=0}, \eqref{r3M-0=0}, and \eqref{r3q+0=0} that
\begin{align}\label{rrrq+0==0}
  \partial_r^3 \bar{q}_+(0) =  0.
\end{align}
Similarly, we have
  \begin{align}\label{rrrrho+0==0}
    \partial_r^3 \bar{\rho}_+(0) = \partial_r^3 \overline{M}_+^2(0)= \partial_r^3 \bar{s}_+(0)=0.
  \end{align}

Taking forth derivatives on the both sides of \eqref{RHS-} with respect to the variable $r$, we have
\begin{align}\label{r5p-=0}
  \partial_r^5 \bar{p}_-(r) = \partial_r^4\big(\frac{ \bar{w}_-^2(r)}{r} \bar{\rho}_-(r) \big).
\end{align}
By $\bar{w}_-(0)=0$ in \eqref{assumew0r0=0}, \eqref{rrw2r=}, \eqref{rrrrho0==0}, and $\partial_r \bar{\rho}_-(0)=0$ in \eqref{barrho-}, it follows from \eqref{r5p-=0} that
\begin{align}\label{r5p-==0}
   \partial_r^5 \bar{p}_-(r)(0) = \bar{\rho}_-(0) \partial_r^4\big(\frac{ \bar{w}_-^2(r)}{r}  \big).
\end{align}
By $\partial_r^k \bar{w}_-(0)=0$, $k=0,2,4$ in \eqref{assumption23=0}, we have
\begin{align}\label{r4w-=0}
  \lim\limits_{r\rightarrow 0} \partial_r^4(\frac{\bar{w}_-^2}{r}) =& \frac15 \partial_r^5 \bar{w}_-^2 (0)\notag\\
  =& \frac25 \big( \bar{w}_-(0)\partial_r^5 \bar{w}_-(0) + 5 \partial_r\bar{w}_-(0) \partial_r^4 \bar{w}_-(0) + 10 \partial_r^2 \bar{w}_-(0)\partial_r^3 \bar{w}_-(0) \big)\notag\\
  =& 0.
\end{align}
So 
\begin{align}\label{r^5p-=0}
   \partial_r^5 \bar{p}_-(0) = 0.
\end{align}
Then by $\bar{w}_+ (r) = \bar{w}_-(r)$ and \eqref{RHS+}, a similar calculation yields that
\begin{align}\label{r^5p+=0}
  \partial_r^5 \bar{p}_+(0) = 0.
\end{align}
Taking forth derivatives on the both sides of \eqref{RHS-} with respect to the variable $r$, we have
\begin{align}\label{4eq:mt}
\partial_r^5 t
=\partial_r^4 \Big(\frac{\gamma-1}{4\bar{q}_-^2}\frac{\bar{w}_-^2}{r}\big( \frac{(\gamma +1)^2}{\gamma-1 + 2t} - \big(\gamma-1 + 2t\big)\big) \Big).
\end{align}
By $\bar{w}_-(0)=0$ in \eqref{assumew0r0=0}, $\partial_r^j \bar{q}_-(0) =0$ for $j=1,3$ in \eqref{assumption23=0}, \eqref{r1M0r0+}, \eqref{rrw2r=}, \eqref{r3M-0=0}, and \eqref{r4w-=0}, it follows from \eqref{4eq:mt} that $\partial_r^5 t(0)=0$, that is
\begin{align}\label{5rM-=0}
  \partial_r^5 \Big(\frac{1}{\overline{M}_-^2}\Big)(0)=0.
\end{align}
So it follows from \eqref{1rM-==0}, \eqref{r3M-0==0} and \eqref{5rM-=0} that
\begin{align}\label{5rM2-==0}
   \partial_r^5 \overline{M}_-^2 (0)=0.
\end{align}
Following a similar calculation as done for \eqref{rrrrho0==0}-\eqref{rrrrho+0==0}, by \eqref{assumption23=0}, \eqref{pro2}, \eqref{partialr3=0} for $j=3$, \eqref{r^5p-=0}, \eqref{r^5p+=0}, and \eqref{5rM2-==0},  we have
\begin{align}\label{j5prhos=0}
  \partial_r^5 (\bar{p}_+, \bar{\rho}_{\pm}, \bar{s}_\pm)(0)=0, \quad \partial_r^5 \bar{q}_+(0)=0.
\end{align}
Thus, we have proven \eqref{partialr3=0} for $j=5$.

Similarly, we can prove \eqref{partialr3=0} for $j=7$ from \eqref{assumption23=0}. Because the argument is similar and long, we omit it for shortness.
\end{proof}

\section{Quasi-supersonic solutions in the cylinder}

In this section, we will consider the existence of quasi-supersonic solutions in the whole domain $\Omega$ by proving Theorem \ref{mainthmU-}.
Before doing that, the compatibility conditions up to the seventh order for the initial-boundary value problem \eqref{CE1}-\eqref{CE5} and \eqref{U-E0r}-\eqref{vwaxis} with \eqref{Uenwq} and \eqref{assumew0r0=0}-\eqref{assumepex} will be given.

To simplify the notations, we will drop the subscript ``$ - $'' to write $U_-$ as $U$ in this section.
First, it is easy to see that the zero order and first order compatibility conditions have already been given in \eqref{assumew0r0=0}-\eqref{assumepex}.
Then we will derive the higher order compatibility conditions by induction. The equations \eqref{CE1}-\eqref{CE5} can be rewritten in the following matrix form:
\begin{align}\label{ABCeq}
  A(U) U_z + B(U) U_r + C(U)=0.
\end{align}
Taking derivatives on the both sides of the equation \eqref{ABCeq} with respect to the variables $z$ and $r$ respectively, we have in $\Omega$
\begin{align}
&A(U) U_{zz}+ B(U) U_{zr}  + D_{U} A(U) U_z U_z + D_{U} B(U) U_z U_r + D_{U} C(U) U_z =0,\label{ABCeqz}\\
& A(U) U_{zr} + B(U) U_{rr} + D_{U} A(U) U_r U_z + D_{U} B(U) U_r U_r+ D_U C(U) U_r =0.\label{ABCeqr}
\end{align}
Take twice tangential derivatives on the boundary conditions \eqref{U-E0r}-\eqref{vwaxis} with \eqref{Uenwq}, then
\begin{align}
U_{rr}=\partial_{rr}\overline{U}_-+\sigma\partial_{rr}U_{en}\qquad&\mbox{on }\Gamma_{en}\\
\partial_{zz}\theta=0\qquad&\mbox{on }\Gamma_w\cup\Gamma_a\\
\partial_{zz}w=0\qquad&\mbox{on }\Gamma_a.\label{4.6x}
\end{align}
Solving equations \eqref{ABCeqz}-\eqref{4.6x} with equation \eqref{ABCeqr},  boundary conditions \eqref{U-E0r}-\eqref{vwaxis} and the zeroth and first order compatibility conditions at point $(z,r)=(0,0)$ or $(z,r)=(0,r_0)$, we can obtain the second order compatibility conditions. Then by taking one more derivatives for each time, we can obtain the compatibility conditions up to the seventh order.

Before giving the proof of Theorem \ref{mainthmU-}, we would like to give the following lemma.
\begin{lem}\label{lem5.4x} For any fixed integer $n\geq2$, and for $\mathbf{x}=(x_1,\cdots,x_n,x_{n+1})$, set $r(\mathbf{x})\defs\sqrt{\sum_{i=1}^nx_i^2}$ and $B_R\defs \{\mathbf{x}\in\mathbb{R}^{n+1}:\,r(\mathbf{x})\leq R\}$.Then for $m=1,2$, $\Phi(\mathbf{x})\defs\varphi(r(\mathbf{x}),x_{n+1})\in C^{m,\alpha}(B_R)$ if and only if $\varphi(r,x_{n+1})\in C^{m,\alpha}([0,R]\times\mathbb{R})$ and $\varphi_r(0,x_{n+1})=0$. For $m=3,4$, $\Phi(\mathbf{x})\defs\varphi(r(\mathbf{x}),x_{n+1})\in C^{m,\alpha}(B_R)$ if and only if $\varphi(r,x_{n+1})\in C^{m,\alpha}([0,R]\times\mathbb{R})$ and $\varphi_r^j(0,x_{n+1})=0$ for $j=1,3$. In addition, there exists a constant $C_m$ depending only on $\alpha$ and $R$, such that
$C_m\|\Phi\|_{C^{m,\alpha}(B_R)}\leq \|\varphi\|_{C^{m,\alpha}([0,R]\times\mathbb{R})}\leq\|\Phi\|_{C^{m,\alpha}(B_R)}$. Moreover, if $\varphi_r^j(0,x_{n+1})=0$ for $j=1,3,5$, and if $\varphi(r,x_{n+1})\in C^{7}([0,R]\times\mathbb{R})$, then $\Phi(\mathbf{x})\defs\varphi(r(\mathbf{x}),x_{n+1})\in H^{7}(B_R)$ and satisfies
$\|\Phi\|_{H^{7}(B_R)}\leq C\|\varphi\|_{H^{7}([0,R]\times\mathbb{R})}$.
\end{lem}
\begin{proof}
The case $m=1,2$ is from lemma in \cite[Lemma 2.1]{HY2007}, which only considers the two-dimensional case and the proof can be generalized to higher dimensions easily. The proof for case $m=3,4$ can be easily done by combining the proof of the Lemma in \cite[Lemma 2.1]{HY2007} and the computation for the case of $H^7$,  so we omit their proofs for the shortness.

For the case of $H^7$, by \cite{LyonsZumbrun}, for any given multi-indices $\beta=(\beta_1, \beta_2,\cdots,\beta_n)$ with $\beta_i\geq0$ for $i=1,\cdots,n,n+1$,
 \begin{align}\label{xiformular}
\partial_{\mathbf{x}}^{\beta}  \partial_{x_{n+1}}^{\beta_{n+1}} \Phi = \sum_{k=0}^{[\frac{|\beta|}{2}]}\frac{1}{2^{k} k!} (\triangle_{\mathbf{x}}^{k} \mathbf{x}^{\beta}) \cdot (D^{|\beta|-k}\partial_{x_{n+1}}^{\beta_{n+1}}\varphi)(r,x_{n+1}),
 \end{align}
where $|\beta| = \beta_1+\beta_2+\cdots+\beta_n$, $\mathbf{x}^{\beta} = x_1^{\beta_1} x_2^{\beta_2}\cdots x_n^{\beta_n}$, $\triangle_{\mathbf{x}}$ denotes the Laplace operator and
  \begin{align}\label{Dvardefs}
   (D \varphi)(r,x_{n+1})\defs \frac{\partial_r\varphi(r,x_{n+1})}{r}.
  \end{align}
Without loss of generality, let us only consider the case that $|\beta|=7$ and $\beta_3=\beta_4=\cdots=\beta_{n+1}=0$, that is, $\beta=(\beta_1,\beta_2,0,\cdots,0)$ and $\beta_1+\beta_2=7$.
It follows from \eqref{xiformular} that for $i,j=1,2$ and $i\neq j$,
 \begin{align}
   \partial_{x_i}^7 \Phi =& 105x_iD^4\varphi + 105 x_i^3 D^5\varphi +  21 x_i^5 D^6\varphi + x_i^7 D^7\varphi \label{xl7=}\\
   \partial_{x_i}\partial_{x_j}^6\Phi
   =& 15 x_i D^4\varphi + 45 x_ix_j^2D^5\varphi + 15 x_i x_j^4 D^6\varphi +   x_i x_j^6 D^7\varphi\\
   \partial_{x_i}^2\partial_{x_j}^5\Phi
   =& 15 x_j D^4\varphi + 5 \big( 3x_i^2 x_j + 2x_j^3 \big) D^5\varphi
    + \big(x_j^5+ 10 x_i^2 x_j^3  \big)D^6\varphi +  x_i^2 x_j^5 D^7\varphi\\
   \partial_{x_i}^3\partial_{x_j}^4\Phi
   =& 9 x_i D^4\varphi + 18 x_i x_j^2 D^5\varphi + 3\big( x_i x_j^4 + x_i^3 x_j^2 \big)D^6\varphi+   x_i^3 x_j^4 D^7\varphi.\label{xl34=}
 \end{align}
It follows from \eqref{Dvardefs} and direct computation that $D\varphi=\frac{\varphi_r}{r}$, $D^2\varphi=\frac{\varphi_{rr}}{r^2}-\frac{\varphi_r}{r^3}$, $D^3\varphi=\frac{\varphi_{rrr}}{r^3}-\frac{3\varphi_{rr}}{r^4}+\frac{3\varphi_r}{r^5}$, and
\begin{align*}
D^4\varphi=&\frac{\partial_r^4\varphi}{r^4}-\frac{6\partial_r^3\varphi}{r^5}+\frac{15\partial_r^2\varphi}{r^6}-\frac{15\varphi_r}{r^7}\\
D^5\varphi=&\frac{\partial_r^5\varphi}{r^5}-\frac{10\partial_r^4\varphi}{r^6}+\frac{45\partial_r^3\varphi}{r^7}-\frac{105\partial_r^2\varphi}{r^8}+\frac{105\varphi_r}{r^9}\\
D^6\varphi=&\frac{\partial_r^6\varphi}{r^6}-\frac{15\partial_r^5\varphi}{r^7}+\frac{105\partial_r^4\varphi}{r^8}-\frac{420\partial_r^3\varphi}{r^9}+\frac{945\partial_r^2\varphi}{r^{10}}-\frac{945\varphi_r}{r^{11}}\\
D^7\varphi=&\frac{\partial_r^7\varphi}{r^7}-\frac{21\partial_r^6\varphi}{r^8}+\frac{210\partial_r^5\varphi}{r^9}-\frac{1260\partial_r^4\varphi}{r^{10}}+\frac{4725\partial_r^3\varphi}{r^{11}}-\frac{10395\partial_r^2\varphi}{r^{12}}+\frac{10395\varphi_r}{r^{13}}
\end{align*}
Note that $\varphi\in C^7$ and $\varphi_r(0)=\partial_{r}^3\varphi(0)=\partial_{r}^5\varphi(0)=0$, by L'H$\hat{o}$pital's rule,
\begin{align}
&\lim \limits_{r \rightarrow 0} r^7D^7\varphi = -\frac{5}{16} \partial_r^7 \varphi(0),\quad \lim \limits_{r \rightarrow 0} r^5D^6\varphi = \frac{1}{16} \partial_r^7 \varphi(0),\label{D76r0=}\\
&\lim \limits_{r \rightarrow 0} r^3D^5\varphi = -\frac{1}{48}\partial_r^7 \varphi(0),\quad \lim \limits_{r \rightarrow 0} rD^4\varphi = \frac{1}{48}\partial_r^7 \varphi(0).\label{D54r0=}
\end{align}

So we know for $|\beta|=7$, $\partial_{\mathbf{x}}^{\beta}\Phi$ are bounded functions.
Therefore, if $\varphi(r, x_{n+1})\in C^7([0,R]\times \mathbb{R})$ and $\partial_r^k\varphi(0, x_{n+1})=0, k=1,3,5$, then $\Phi(\mathbf{x})\defs \varphi(r(\mathbf{x}), x_{n+1})\in H^7(B_R)$.
\end{proof}

We are now ready to prove Theorem \ref{mainthmU-}.
\begin{proof}[Proof of Theorem \ref{mainthmU-}]
If the solutions are smooth,  equations \eqref{CE1}-\eqref{CE5} can be rewritten as follows.
\begin{align}
  &\partial_z (r\rho u) + \partial_r (r \rho v)=0,\label{Supeq1}\\
  & uu_z + vu_r + \frac{1}{\rho} p_z=0,\label{Supeq2}\\
  & uw_z + vw_r + \frac{vw}{r} =0,\label{Supeq3}\\
  & uB_z + vB_r=0,\label{Supeq4}\\
  & us_z + vs_r=0.\label{Supeq5}
\end{align}
It follows from \eqref{Supeq1} that there exists a function $ \psi$ such that
\begin{align*}
  \nabla \psi = (\psi_z, \psi_r) = (- r\rho v, r\rho u),
\end{align*}
so
\begin{align}\label{uvpsi}
  v= -\frac{\psi_z}{r \rho}, \quad u = \frac{\psi_r}{r \rho}.
\end{align}
Then for any $r\in (0,r_0)$, one has
\begin{align}\label{4.13x}
  \psi(0,r) = \int_0^r \tau \rho(0,\tau)u(0,\tau)\dif \tau,
\end{align}
where it follows from \eqref{UenE} and \eqref{Uenwq} that
\begin{equation}\label{4.14x}
(p(0,r),\theta(0,r),w(0,r),q(0,r),s(0,r))=(\bar{p}_-,0,\bar{w}_-+\sigma w_{en},\bar{q}_-+\sigma q_{en},\bar{s}_-)
\end{equation}
and then
\begin{equation}\label{4.15x}
\rho(0,r)=\left(\frac{p(0,r)}{Ae^{\frac{s(0,r)}{c_v}}}\right)^{\frac{1}{\gamma}}\qquad v(0,r)=0\qquad u(0,r)=q(0,r).
\end{equation}

Next, the transport equations \eqref{Supeq3}-\eqref{Supeq5} become 
\begin{align}\label{wbspsi}
 \psi_r W_z - \psi_z W_r =0,\,\,  \psi_r B_z - \psi_z B_r =0,\,\,\psi_r s_z - \psi_z s_r =0,
\end{align}
where $W\defs rw$. For their initial conditions, we have
\begin{align*}
W(0,r)=& W_{EN}(r)\defs r w(0,r)\\
s(0,r)=&s_{EN}(r)\defs \bar{s}_-(r),\\
B(0,r)=& B_{EN}(r) \defs \frac12 \Big( (q(0,r))^2 + (w(0,r))^2   \Big)
   + \frac{\gamma p(0,r)}{(\gamma-1)\rho(0,r)},
\end{align*}
where the intial states are given by \eqref{4.14x} and \eqref{4.15x}.

It follows from \eqref{4.13x} and the properties that $\bar{\rho}\geq C>0$ and $\bar{q}\geq C>0$ that $\psi$ is monotonely increasing with respect to $r\in[0,r_0]$. So it is one-to-one.
Let $\zeta \defs \psi^{-1}$. Obviously, $\zeta$ is a one to one mapping from $[0, \psi(r_0)]$ to $[0,r_0]$ too.
Thus, \eqref{wbspsi} yields that for $\psi\in[0,\psi(r_0)]$, 
\begin{align}
W(\psi) = W_{EN}(\zeta(\psi)),\quad   B(\psi) = B_{EN}(\zeta(\psi)), \quad s(\psi) = s_{EN}(\zeta(\psi)).
\end{align}

Finally, let us consider equation \eqref{Supeq2}. Let
\begin{align}\label{Frhoeq}
F\big(\rho, B, s, \frac{W}{r}, \frac{|\nabla \psi|^2}{r^2}\big)\defs \frac12 \Big(\frac{|\nabla \psi|^2}{r^2} + \rho^2(\frac{W}{r})^2 \Big) +\frac{\gamma A e^{\frac{s}{c_v}} \rho^{\gamma+1}}{\gamma-1} -  \rho^2 B.
\end{align}
Bernoulli's law implies that $ F\big(\rho, B, s,\frac{W}{r}, \frac{|\nabla \psi|^2}{r^2}\big) =0$. Taking the derivative on the both sides of \eqref{Frhoeq} with respect to $\rho$, we have
\begin{align*}
F_{\rho}
=& \rho ((\frac{W}{r})^2-2B) + \frac{\gamma(\gamma+1) A e^{\frac{s}{c_v}}}{\gamma-1}\rho^{\gamma}\\
= &\frac{\gamma(\gamma+1) A e^{\frac{s}{c_v}}}{\gamma-1}\rho^{\gamma} - \frac{1}{\rho}\frac{|\nabla \psi|^2}{r^2}- \frac{2\gamma A e^{\frac{s}{c_v}}}{\gamma-1}\rho^{\gamma}\notag\\
   = & {\rho}\big( c^2 -(u^2 +v^2)  \big)<0.
\end{align*}
Thus, by the inverse function theorem, function $\rho$ can be represented as a function of $(B, s, \frac{W}{r}, \frac{|\nabla \psi|^2}{r^2})$ such that $ F\big(\rho, B, s,\frac{W}{r}, \frac{|\nabla \psi|^2}{r^2}\big) =0$. That is,
\begin{align}\label{rhobsw}
  \rho = \rho(B, s, \frac{W}{r}, \frac{|\nabla \psi|^2}{r^2}).
\end{align}
Notice that \eqref{Supeq2} is equivalent to
\begin{align}\label{Supeq2re}
v(u_r -   v_z) = -\frac12 \partial_z(u^2 + v^2)- \frac{1}{\rho} p_z.
  \end{align}
Let
\begin{align}\label{psivar}
   \psi = r^2 \varphi.
\end{align}
Applying \eqref{uvpsi}, direct calculations yield that
\begin{align}\label{urvz}
 u_r- v_z =  r \varphi_{rr} + r \varphi_{zz} + 3\varphi_r - \frac{1}{\rho}(2\varphi+ r\varphi_r)\rho_r - \frac{1}{\rho}r \varphi_z \rho_z .
\end{align}
Then let us calculate $\rho_z$ and $\rho_r$. Let $\chi\defs \frac{|\nabla \psi|^2}{r^2}$. It follows from \eqref{Frhoeq} that
\begin{equation}
F_B=-\rho^2\qquad F_s=\frac{\gamma A e^{\frac{s}{c_v}}\rho^{\gamma+1}}{(\gamma-1)c_v}=\frac{\gamma p\rho}{(\gamma-1)c_v}\qquad F_{\frac{W}{r}}=\rho^2\frac{W}{r}\qquad F_{\chi}=\frac12.
\end{equation}
Then by the Bernoulli's law
\begin{align*}
F_{\rho}\rho_z=&F_BB_z+F_ss_z+F_{\frac{W}{r}}(\frac{W}{r})_z+F_{\chi}\chi_z\\
 =&(F_BB'(\psi)+F_ss'(\psi)+F_{\frac{W}{r}}\frac{W'(\psi)}{r})\psi_z+\frac{1}{2r^2} \partial_z\Big( \psi_z^2 + \psi_r^2\Big) \\
=&\frac{1}{r^2}\Big( \psi_z \psi_{zz} + \psi_r \psi_{zr}\Big)+\rho^2\left(-B'+\frac{\gamma p}{(\gamma-1)c_v\rho}s'+\frac{WW'}{r^2}\right)\psi_z\\
F_{\rho}\rho_r=&F_BB_r+F_ss_r+F_{\frac{W}{r}}(\frac{W}{r})_r+F_{\chi}\chi_r\\
=&(F_BB'+F_ss'+F_{\frac{W}{r}}\frac{W'}{r})\psi_r-\frac{WF_{\frac{W}{r}}}{r^2}+ \partial_r\Big(\frac{\psi_z^2 + \psi_r^2}{2r^2}\Big)\\
=&\frac{1}{r^2}\Big( \psi_z \psi_{zr} + \psi_r \psi_{rr}\Big) - \frac{1}{r^3}\big( \psi_z^2 + \psi_r^2\big)\\
&+\rho^2\left(-B'+\frac{\gamma p}{(\gamma-1)c_v\rho}s'+\frac{WW'}{r^2}\right)\psi_r-\frac{\rho^2W^2}{r^3}.
\end{align*}
Note that $u^2+v^2=\frac{\chi}{\rho^2}$, $p=Ae^{\frac{s}{c_v}}\rho^{\gamma}$ and $F_{\rho}=\rho(c^2-(u^2+v^2))$. Then it follows from equations \eqref{Supeq2re} and \eqref{urvz} and a straightforward and long computation that $\varphi$ satisfies equation
\begin{align}\label{result4}
  &\big(1 - \frac{r^2 \varphi_z^2}{\rho^2 (q^2 - c^2)}\big) \varphi_{zz} +\big(1 - \frac{(2\varphi + r\varphi_r)^2}{\rho^2 (q^2 - c^2)}\big) \big(\varphi_{rr} +\frac{3}{r}\varphi_r\big)\notag\\
  & - \frac{2(2\varphi+  r\varphi_r)}{\rho^2 (q^2 - c^2)}r \varphi_z\varphi_{zr} - \frac{3(2\varphi+  r\varphi_r)}{\rho^2 (q^2 - c^2)} \varphi_z^2\notag\\
  =&F(\varphi, D\varphi)\notag\\
\defs&(r \varphi_r + 2\varphi) \Big\{ \frac{\rho}{u}B'(\psi)-\frac{\rho}{u} \frac{w}{r}W'(\psi)- \frac{\rho^{\gamma}}{\gamma-1}\frac{A}{c_v} e^{\frac{s}{c_v}}\frac{1}{u} s'(\psi)\notag\\
  &\qquad  + \frac{1}{\rho} (r \varphi_r + 2\varphi)
  \Big( \frac{\partial\rho}{\partial B}B'(\psi)+ \frac{\partial\rho}{\partial s}s'(\psi) + \frac{w}{r}\frac{\rho W'(\psi)}{q^2 -c^2}\Big) -\frac{1}{q^2 -c^2}\frac{w^2}{r^2}\Big\}\notag\\
  & + \frac{1}{\rho} \varphi_z \Big( \frac{\partial\rho}{\partial B}B'(\psi)r^2\varphi_z + \frac{\partial\rho}{\partial s}s'(\psi)r^2\varphi_z + \frac{\partial\rho}{\partial w}W'(\psi)r\varphi_z\Big).
\end{align}

In order to overcome the difficulty from the singularity along the symmetry axis $\{r=0\}$ in equation \eqref{result4}, we introduce $\mathbf{r} = (r_1, r_2, r_3, r_4)$ and $\sum\limits_{i=1}^4 r_i^2 = r^2$. The domain $\Omega$ becomes
$\mathcal{D}_- \defs \{(z,\mathbf{r}): z\in (0, L), \mathbf{r}\in \mathbb{R}^4, |\mathbf{r}|< r_0 \}$.
Let
$$
\phi(z, r_1, r_2, r_3, r_4) = \varphi (z, \sqrt{r_1^2 + r_2^2 +r_3^2 +r_4^2}).
$$
Then equation \eqref{result4} becomes 
\begin{align}\label{5Dwaveeq}
&\big(1 - \frac{r^2 \varphi_z^2}{\rho^2 (q^2 - c^2)}\big) \phi_{zz} +\big(1 - \frac{(2\varphi + r\varphi_r)^2}{\rho^2 (q^2 - c^2)}\big)\sum_{i=1}^4 \phi_{r_ir_i}\notag\\
  & - \frac{2(2\varphi+  r\varphi_r)}{\rho^2 (q^2 - c^2)} \phi_z\sum_{i=1}^4 r_i(\phi_z)_{r_i}- \frac{3(2\varphi+  r\varphi_r)}{\rho^2 (q^2 - c^2)} (\phi_z)^2 =  F(\phi, D\phi).
  \end{align}
Next, at the background solution $\overline{U}_-$, it follows from \eqref{uvpsi} that
\begin{align}
  \overline{\psi}(r) \defs \psi(0,r) \equiv \int_0^r \tau \bar{\rho}_-(\tau) \bar{q}_-(\tau)\dif \tau.
\end{align}
Then by definition \eqref{psivar}, we have
\begin{align}\label{4.27x}
 \overline{\varphi}(r) = \frac{1}{r^2}\int_0^r \tau \bar{\rho}_-(\tau) \bar{q}_-(\tau)\dif \tau.
\end{align}
Let
\begin{align}\label{4.28x}
  \widehat{\varphi}(z,r)\defs \varphi(z,r) - \overline{\varphi}(z,r),
\end{align}
and
$$\widehat{\phi}(z, r_1, r_2, r_3, r_4) =  \widehat{\varphi} (z, \sqrt{r_1^2 + r_2^2 +r_3^2 +r_4^2}).
$$
Then it follows from \eqref{5Dwaveeq} that $\widehat{\phi}$ satisfies equation
\begin{align}\label{5Dwaveeqhat}
&\big(1 - \frac{r^2 \varphi_z^2}{\rho^2 (q^2 - c^2)}\big) \widehat{\phi}_{zz} +\big(1 - \frac{(2\varphi + r\varphi_r)^2}{\rho^2 (q^2 - c^2)}\big)\sum_{i=1}^4 \widehat{\phi}_{r_ir_i}\notag\\
  & - \frac{2(2\varphi+  r\varphi_r)}{\rho^2 (q^2 - c^2)} \widehat{\phi}_z\sum_{i=1}^4 r_i(\widehat{\phi}_z)_{r_i}- \frac{3(2\varphi+  r\varphi_r)}{\rho^2 (q^2 - c^2)} (\widehat{\phi}_z)^2\notag\\
   =&  F(\phi, D\phi) - \big(1 - \frac{(2\varphi + r\varphi_r)^2}{\rho^2 (q^2 - c^2)}\big)\sum_{i=1}^4 \overline{\phi}_{r_ir_i},
  \end{align}
where $\overline{\phi}(r_1, r_2, r_3, r_4) = \overline{\varphi} (\sqrt{r_1^2 + r_2^2 +r_3^2 +r_4^2})$.
Notice that
\begin{align*}
&\big(1 - \frac{r^2 \varphi_z^2}{\rho^2 (q^2 - c^2)}\big) = \big(1 - \frac{\psi_z^2}{r^2 \rho^2 (q^2 - c^2)}\big) = \big(1 - \frac{v^2}{q^2 - c^2}\big) = \frac{u^2 -c^2}{q^2 - c^2},\\
&\big(1 - \frac{(2\varphi + r\varphi_r)^2}{\rho^2 (q^2 - c^2)}\big) = \big(1 - \frac{u^2}{q^2 - c^2}\big) = \frac{v^2-c^2}{ q^2 - c^2}.
\end{align*}
Then equation \eqref{5Dwaveeqhat} can be rewritten as
\begin{align*}
&(u^2-c^2)\widehat{\phi}_{zz} -(c^2-v^2)\sum_{i=1}^4 \widehat{\phi}_{r_ir_i}
  -\frac{2u}{\rho} \widehat{\phi}_z\sum_{i=1}^4 r_i(\widehat{\phi}_z)_{r_i}- \frac{3u}{\rho} (\widehat{\phi}_z)^2\\
= &(q^2-c^2)F +(c^2-v^2)\sum_{i=1}^4 \overline{\phi}_{r_ir_i}.
  \end{align*}

Recall that for the background solution, $\bar{v} =0$ and $\bar{u} =\bar{q} >\bar{c}$. So for sufficiently small constant $\sigma>0$, $c^2-v^2>0$ and $u^2-c^2>0$.
Thus, \eqref{result4} is a five dimensional nonlinear wave equation.
By \eqref{UenE}, \eqref{Uenwq}, \eqref{4.13x}-\eqref{4.15x}, \eqref{psivar} and \eqref{4.27x}-\eqref{4.28x}, the initial-boundary conditions for $\widehat{\phi}$ are
\begin{align}
   \widehat{\phi}(0, r_1, r_2, r_3, r_4) =& \frac{A^{-\frac{1}{\gamma}}}{r^2}\int_0^{r} \tau \bar{p}_-^{\frac{1}{\gamma}}(\tau)e^{-\frac{\bar{s}_-(\tau)}{\gamma c_v}} \sigma q_{en}(\tau)\dif \tau\label{bry5D1}\\
  \partial_z  \widehat{\phi}(0, r_1, r_2, r_3, r_4) =& 0\label{4.31x}\\
  \widehat{\phi}(z, r_0) =& 0,
  \end{align}
where $r=\sqrt{r_1^2 + r_2^2 +r_3^2 +r_4^2}$.

Let $\overline{H}(r) \defs \bar{p}_-^{\frac{1}{\gamma}}(r)e^{-\frac{\bar{s}_-(r)}{\gamma c_v}} q_{en}(r)$, then \eqref{bry5D1} becomes
  \begin{align}
     \widehat{\phi}(0, r_1, r_2, r_3, r_4) = \sigma\frac{A^{-\frac{1}{\gamma}}}{r^2}\int_0^{r} \tau \overline{H}(\tau)\dif \tau.
  \end{align}
Now we will show that
\begin{align}
\frac{1}{r^2}\int_0^{r} \tau \overline{H}(\tau)\dif \tau\in H^7(B_{r_0}).
\end{align}
First, by $\partial_r^j( \bar{p}_-, \bar{s}_-,q_{en}) (0)=0$, $j=1,3,5$, we have
\begin{align}
\partial_r^j  \overline{H}(0) =0, \quad j=1,3,5.
\end{align}
Thus
\begin{align}
&\lim\limits_{r\rightarrow0} \partial_r \big(  \frac{1}{r^2}\int_0^{r} \tau \overline{H}(\tau)\dif \tau \big)= \lim\limits_{r\rightarrow0}\big(\frac{r \overline{H}(r) }{r^2} - 2\frac{\int_0^{r} \tau \overline{H}(\tau)\dif \tau }{r^3}\big) = \frac{ \partial_r\overline{H}(0)}{3}=0,\label{H'0=0}\\
&\lim\limits_{r\rightarrow 0} \partial_r^3 \big(  \frac{1}{r^2}\int_0^{r} \tau \overline{H}(\tau)\dif \tau \big)=\lim\limits_{r\rightarrow 0} \big( \frac{\overline{H}''(r)}{r} - 4 \frac{\overline{H}'(r)}{r^2}  +12 \frac{\overline{H}(r)}{r^3} -24 \frac{\int_0^{r} \tau \overline{H}(\tau)\dif \tau }{r^5}\big)\notag\\
=& \frac15 \partial_r^3\overline{H}(0) =0,\\
&\lim\limits_{r\rightarrow 0} \partial_r^5 \big(  \frac{1}{r^2}\int_0^{r} \tau \overline{H}(\tau)\dif \tau \big)= \frac17 \partial_r^5\overline{H}(0) =0.\label{H5'0=0}
\end{align}
By \eqref{H'0=0}-\eqref{H5'0=0} and Lemma  \ref{lem5.4x}, the initial data $\widehat{\phi}(0, r_1, r_2, r_3, r_4)\in H^7(B_{r_0})$.

By doing the standard energy estimate for the linear wave equation and the Picard iteration for the local existence (see Theorem 6.4.11 in \cite[Page 113]{Lars}), we have for any given $L>0$, if $\sigma>0$ is sufficiently small, there is a unique solution $\widehat{\phi}\in \cap_{j=0}^7 C^j\Big( [0,L];H^{7-j}(\mathcal{D}_-)  \Big)$ such that
\[
\sum_{j=0}^7 \|\widehat{\phi}\|_{C^j( [0,L];H^{7-j}(\mathcal{D}_-))} \leq C \sigma (\|w_{en}\|_{H^7([0,r_0]))}+ \|q_{en}\|_{H^7([0,r_0])}),
\]
 where the positive constant $C$ depends on $\overline{U}_-$, $L$ and $r_0$.
So it follows from the embedding theorem that
$\widehat{\phi}\in C^4(\overline{\mathcal{D}_-})$ and
 \begin{align}\label{widehatphieq}
   \|\widehat{\phi}\|_{C^4(\overline{\mathcal{D}_-})} \leq  C \sigma (\|w_{en}\|_{H^7([0,r_0]))}+ \|q_{en}\|_{H^7([0,r_0])}).
 \end{align}

Moreover, the linearized problem of \eqref{5Dwaveeqhat} is rotationally invariant. The same are the domain and the initial-boundary conditions. So by the uniqueness, we know the solutions of the linearized problem of \eqref{5Dwaveeqhat} are actually functions of $(z,r)$, such that $\widehat{\phi}(z, r_1, r_2, r_3, r_4) =  \widehat{\varphi} (z, \sqrt{r_1^2 + r_2^2 +r_3^2 +r_4^2})$. By the iteration argument and the uniqueness of the nonlinear problem \eqref{5Dwaveeqhat}-\eqref{4.31x}, we know that the solutions of the nonlinear problem \eqref{5Dwaveeqhat}-\eqref{4.31x} are also functions of $(z,r)$.
Next, we will verify the following identities hold
\begin{align}\label{r13hatphi0=0}
  \partial_r \widehat{\varphi}(z,0)=0, \quad \partial_r^3\widehat{\varphi}(z,0)=0.
\end{align}
By \eqref{4.28x}, we only need to prove that
\begin{align}
&\partial_r \overline{\varphi}(0) =0, \quad \partial_r^3 \overline{\varphi}(0) =0,\label{overlinephir13=0}\\
& \partial_r \varphi (z,0) =0, \quad \partial_r^3 \varphi(z,0) =0.\label{r13varphi0=0}
\end{align}
By $\partial_r^j(\bar{\rho}_-, \bar{q}_-)(0)=0$, $j=1,3$, it is easy to see that \eqref{overlinephir13=0} holds.

Next, it follows from \eqref{uvpsi} and $\psi = r^2 \varphi$ that
\begin{align}
r \rho u  = \psi_r = 2r \varphi + r^2 \varphi_r.
\end{align}
Then 
\begin{align}
3 \varphi_r + r \varphi_{rr} = \partial_r (\rho u)\qquad\mbox{and}\qquad 
5\partial_{r}^3 \varphi+ r  \partial_r^4 \varphi= \partial_r^3  (\rho u).\label{varphirrr=}
\end{align}
They yield that 
\begin{align}
3 \partial_r\varphi(z,0) = \partial_r (\rho u)(z,0)\qquad\mbox{and}\qquad 
5\partial_{r}^3 \varphi(z,0) = \partial_r^3  (\rho u)(z,0).\label{varphirrr=0}
\end{align}
So in order to show \eqref{r13hatphi0=0},
it suffices to prove that
\begin{align}\label{prover13=0}
  \partial_r (\rho u)(z,0) =0 \qquad\mbox{and}\qquad \partial_r^3  (\rho u)(z,0) =0.
\end{align}
Taking derivative with respect to the variable $r$ on the both sides of \eqref{Supeq5} and applying $v(z,0)=0$ and $\partial_r s(0,0)=0$, we have
\begin{align}\label{rsz0=0}
  \partial_r s(z,0)=0.
\end{align}
Similarly, by $v(z,0)=0$ and $\partial_r (p, q, s)(0,0)=0$, \eqref{Supeq4} yields that
\begin{align}\label{rBz0=0}
  \partial_r B(z,0)=0.
\end{align}
\eqref{CE1} and \eqref{CE3} yield that
 \begin{align}\label{CE13=}
   \rho u v_z + \rho v v_r + p_r -\frac{1}{r}\rho w^2=0.
 \end{align}
By $v(z,0)=0$ and $w(z,0)=0$, \eqref{CE13=} yields that
\begin{align}\label{rpz0=0}
  \partial_r p(z,0)=0.
\end{align}
\eqref{rsz0=0}, \eqref{rBz0=0} and \eqref{rpz0=0} imply that
\begin{align}\label{rrhoqz0=0}
  \partial_r \rho (z,0)=0 \qquad\mbox{and}\qquad \partial_r q (z,0)=0.
\end{align}
Taking derivative with respect to the variable $r$ on the both sides of \eqref{Supeq1}, we have
\begin{align}\label{rsuper1=}
&\partial_r u \partial_z \rho + u \partial_z(\partial_r \rho) + v \partial_r^2 \rho + \partial_r v \partial_r \rho + \partial_r \rho (\partial_z u + \partial_r u)\notag\\
& +\rho (\partial_z (\partial_r u) + \partial_r^2 v) + \partial_r \rho \frac{v}{r} + \rho \partial_r (\frac{v}{r})=0.
\end{align}
By \eqref{rrhoqz0=0} and $v(z,0)=0$, then  
\begin{align}\label{r2vz0=0}
  \partial_r^2 v(z,0)=0.
\end{align}
Taking twice derivatives with respect to the variable $r$ on the both sides of \eqref{Supeq3} and applying $\partial_r^2\bar{w}(0)=0$, $\partial_r^2 w_{en}(0,0)=0$, $w(z,0)=0$ and \eqref{r2vz0=0}, we have
\begin{align}\label{r2wz0=0}
  \partial_r^2 w(z,0)=0.
\end{align}
Taking twice derivative with respect to the variable $r$ on the both sides of \eqref{CE13=} and employing \eqref{rrhoqz0=0}, \eqref{r2vz0=0}, \eqref{r2wz0=0}, $v(z,0)=0$, $w(z,0)=0$, we have
\begin{align}
  \partial_r^3 p(z,0)=0.
\end{align}
Then
\begin{align}\label{r3rhoqz0=0}
  \partial_r^3 \rho(z,0)=0, \quad    \partial_r^3 q(z,0)=0.
\end{align}
By \eqref{rrhoqz0=0}, \eqref{r2vz0=0}, \eqref{r3rhoqz0=0} and $v(z,0)=0$, we obtain \eqref{prover13=0}. Therefore, \eqref{r13hatphi0=0} is verified, and then it follows from Lemma \ref{lem5.4x} that
 \begin{align}
   \| \widehat{\varphi}\|_{C^4(\overline{\Omega})} \leq C \sigma (\|w_{en}\|_{H^7([0,r_0]))}+ \|q_{en}\|_{H^7([0,r_0])}).
 \end{align}
Hence $U_-\in C^3(\overline{\Omega})$ and \eqref{eq837} holds.
\end{proof}

\section{Linearized problem in the quasi-subsonic domain}
In this section, we will reformulate the nonlinear free boundary value problem in the quasi-subsonic domain behind the shock front into a fixed nonlinear boundary value problem. Then we will linearize the problem and establish the well-posedness theorems of it by introducing two auxiliary problems for the two elliptic equations and then solving the remaining three transport equations. The well-posedness theorems of the linearized problem play an important role in solving the nonlinear free boundary value problem in the next section.

\subsection{Reformulation of the nonlinear free boundary value problem in the quasi-subsonic domain behind the shock front}
By characteristic decomposition, the equations \eqref{CE1}-\eqref{CE5} can be rewritten as follows:
\begin{align}
& \frac{1-M^2\cos^2\theta}{\rho q^2}\partial_z p - \frac{M^2\cos\theta \sin\theta}{\rho q^2}\partial_r p - \partial_r \theta - \frac{\cos\theta \sin\theta}{r} =0,\label{CDE1}\\
&\rho q^2\partial_z \theta  - \frac{\rho q^2M^2\cos\theta \sin\theta}{1-M^2\cos^2\theta}\partial_r \theta +\frac{1-M^2 }{1-M^2\cos^2\theta}\partial_r p -\frac{\rho q^2\sin^2\theta}{r(1-M^2\cos^2\theta)} - \frac{\rho w^2}{r}=0,\label{CDE2}\\
&\partial_z (rw) + \tan\theta \cdot \partial_r (rw) =0,\label{CDE3}\\
&  \partial_z {B} + \tan\theta\cdot  \partial_r {B}  =0, \label{CDE4}\\
&\partial_z s + \tan\theta \cdot \partial_r s =0,\label{CDE5}
\end{align}
where the Bernoulli's function
${B} \defs \frac12 (q^2 + w^2) +\frac{\gamma p}{(\gamma-1)\rho}$.
Next, the R.-H. conditions \eqref{CRH1}-\eqref{CRH5} can be rewritten as
\begin{align}
G_1(U_+, U_-) \defs& [\rho u][p+\rho v^2]-[\rho u v ] [\rho v]= 0,\label{IG1}\\
G_2(U_+, U_-) \defs& [p+\rho u^2][p+\rho v^2]-([\rho u v ])^2=0,\label{IG2}\\
G_3(U_+, U_-)\defs& [ w] =0,\label{IG3}\\
G_4(U_+, U_-) \defs& [\frac12 q^2 +i ]=0,\label{IG4}\\
G_5(U_+, U_-;\varphi')\defs& [\rho uv]-\varphi'[p+\rho v^2]=0.\label{IG5}
\end{align}
In order to fix and flatten the shock front $\Gamma_s$, which is a free boundary, let us introduce the coordinates transformation
\begin{align}\label{Kfixed}
 \pounds_K: (\tilde{z},\tilde{r} ) = \Big(L + \displaystyle\frac{L - K}{L - {\varphi}(r)}(z - L) ,\,r\Big),
\end{align}
where constant $K\in(0,L)$ is the approximating shock location, which is uniquely determined and will be given in the next section. We expect $|\varphi(r)-K|$ to be small, so $\pounds_K$ is invertible and its inverse is
\begin{align}
\pounds_{K}^{-1} : (z,\,r) = \big(L + \frac{L - \varphi(\tilde{r})}{L - K}(\tilde{z} - L),\,
\tilde{r}\big).
\end{align}
Then the domain $\Omega_+$ becomes
\begin{align*}
\Omega_+^K\defs\{ (z,r)\in \mathbb{R}^2 : K < z < L, 0 < r < r_0\}.
\end{align*}
For the notational simplicity, we drop " {$\widetilde{}$} " in the sequel argument.

By tedious calculations, we can linearize equations \eqref{CDE1}-\eqref{CDE5} for $U_+=\overline{U}_+ + \delta U$ and $\varphi' = \delta\varphi'$ as follows.
\begin{align}
& \partial_z \Big(\frac{1-\overline{M}_+^2}{\bar{\rho}_+\bar{q}_+^2} r \bar{p}_+^{\frac{1}{\gamma}} \delta{p}  \Big) -\partial_r \Big( r \bar{p}_+^{\frac{1}{\gamma}}\delta{\theta}
\Big )  = r \bar{p}_+^{\frac{1}{\gamma}} {f}_1(\delta U, \varphi),\label{eqf1}\\
&\partial_z \Big(\bar{\rho}_+\bar{q}_+^2  \bar{p}_+^{-\frac{1}{\gamma}}\delta{\theta} \Big) +\partial_r \Big(  \bar{p}_+^{-\frac{1}{\gamma}}\delta{p} \Big)
-\frac{2\bar{\rho}_+ \bar{w}_+\bar{p}_+^{-\frac{1}{\gamma}}}{r}
\delta{w} + \frac{\bar{\rho}_+ \bar{w}_+^2\bar{p}_+^{-\frac{1}{\gamma}}}{r\gamma c_v}
\delta{s}=\bar{p}_+^{-\frac{1}{\gamma}}{f}_2^{\sharp}(\delta U, \varphi),\label{eqf2}\\
& \big( \partial_z +  H(z,r)\partial_r\big) (r\delta w)+ \partial_r( r\bar{w}_+ )\cdot \delta \theta ={f}_3(\delta U, \varphi),\label{eqf3}\\
&  \big( \partial_z +  H(z,r)\partial_r\big) \delta B + \partial_r\bar{B}_+ \cdot \delta \theta =f_4(\delta U, \varphi),\label{eqf4}\\
& \big( \partial_z +  H(z,r)\partial_r\big) \delta s + \partial_r\bar{s}_+ \cdot \delta \theta={f}_5(\delta U, \varphi),\label{eqf5}
\end{align}
where
\begin{align}
 H(z,r)\defs&\frac{(L- \varphi)\tan\delta \theta}{L - K + (z - L)\delta \varphi'\tan\delta \theta},\\
{f}_1(\delta U,\varphi)
\defs& \Big( \frac{1-\overline{M}_+^2}{\bar{\rho}_+\bar{q}_+^2} - \frac{1-M^2\cos^2\theta}{\rho q^2}\Big) \partial_z \delta{p} + \frac{M^2\cos\theta \sin\theta}{\rho q^2}\partial_r \delta p\notag\\
& +\frac{\cos\theta \sin\theta - \delta \theta}{r}
+ \frac{\bar{w}_+^2}{r} \Big(
\frac{\bar{\rho}_+ M^2\cos\theta \sin\theta}{\rho q^2}-\frac{\overline{M}_+^2 }{\bar{q}_+^2}\delta{\theta}
\Big)\notag\\
 &+ \frac{(z-L)\delta \varphi'}{L-\varphi}\Big(\frac{M^2\cos\theta \sin\theta}{\rho q^2}\partial_z \delta p + \partial_z \delta \theta   \Big)\notag\\
&-\frac{\varphi -K}{L-\varphi}\frac{1-M^2\cos^2\theta}{\rho q^2}\partial_z \delta p,\label{deff1}\\
{f}_2^{\sharp}(\delta U,\varphi)
\defs&-\big(  \rho q^2 - \bar{\rho}_+\bar{q}_+^2  \big)\partial_z \delta \theta + \frac{M^2\sin^2\theta }{1-M^2\cos^2\theta}\Big(\partial_r\delta p+
\frac{1}{r}\bar{\rho}_+ \bar{w}_+^2\Big)\notag\\
&+ \frac{\rho q^2M^2\cos\theta \sin\theta}{1-M^2\cos^2\theta}\partial_r \delta \theta + \frac{1}{r}\frac{\rho q^2\sin^2\theta}{1-M^2\cos^2\theta}\notag\\
&+\frac{1}{r}\Big(\rho w^2 -\bar{\rho}_+ \bar{w}_+^2- \frac{\overline{M}_+^2 \bar{w}_+^2}{\bar{q}_+^2}\delta{p}- 2\bar{\rho}_+ \bar{w}_+
\delta{w} + \frac{1}{\gamma c_v}\bar{\rho}_+ \bar{w}_+^2\delta{s}\Big)\notag\\
& + \frac{(z-L)\delta \varphi'}{L-\varphi}\Big(\frac{\rho q^2 M^2\cos\theta \sin\theta}{1-M^2\cos^2\theta}\partial_z \delta\theta - \frac{1 - M^2}{1-M^2\cos^2\theta}\partial_z \delta p  \Big)\notag\\
&- \frac{\varphi -K}{L-\varphi}\rho q^2\partial_z \delta\theta,\label{itemf2}\\
{f}_3(\delta U, \varphi)
\defs&  - \partial_r (r\bar{w}_+) \cdot \big( H(z,r) -\delta \theta \big),\\
f_4(\delta U,\varphi)\defs& -\partial_r\bar{B}_+ \cdot \big(H(z,r) -\delta \theta \big),\\
{f}_5(\delta U, \delta \varphi', \delta z_{*})
\defs&  -\partial_r\bar{s}_+ \cdot \big(H(z,r) -\delta \theta \big).\label{itemf5}
\end{align}

The boundary conditions in the new coordinates are
\begin{enumerate}
	\item On the wall of the nozzle $\Gamma_{w}\cap\overline{\Omega_+^K}$,
\begin{align}\label{GAMMA4theta}
\delta \theta = 0;
\end{align}
\item On the symmetry axis $\Gamma_{a}\cap\overline{\Omega_+^K}$,
\begin{align}\label{GAMMA2}
  \delta \theta = 0,\quad \delta w =0;
\end{align}

\item On the exit of the nozzle $\Gamma_{ex}$,
\begin{equation}\label{eq867}
  \delta p = \sigma p_{ex}(r),\quad \text{on}\quad\Gamma_{ex};
\end{equation}

\item On the shock front $\Gamma_K\defs\{z= K, r\in(0,r_0)\}$,
\begin{align}
&\alpha_j^+\cdot \delta U = G_j^{\sharp}(\delta U,\delta U_-, \varphi),\quad {j = 1,2,3,4},\label{eq251RH}\\
&\alpha_5^+ \cdot \delta U - [\bar{p}]\delta {\varphi}' = G_5^{\sharp} (\delta U,\delta U_-, \varphi),\label{eq11000}
\end{align}
where ${\mathbf{\alpha}}_j^{+} = {{\nabla_{{U}_+}}}G_j(\overline{U}_+, \overline{U}_-)$, ${\mathbf{\alpha}}_5^{+} = {{\nabla_{{U}_+}}}G_5(\overline{U}_+, \overline{U}_-;0)$, and
\begin{align}
&G_j^{\sharp}(\delta U,\delta U_-, \varphi): = \alpha_j^+\cdot \delta U - G_j(U, U_-(\varphi', \varphi(r_0))),\label{5.29xw}\\
&G_5^{\sharp} (\delta U,\delta U_-, \varphi): = \alpha_5^+ \cdot \delta U - [\bar{p}]\delta \varphi' - G_5(U, U_-(\varphi', \varphi(r_0));\varphi')\label{xxl}.
\end{align}
\end{enumerate}

By straightforward computation, we have
\begin{align}
& {\mathbf{\alpha}}_1^+ = \Big(\frac{\bar{q}_+}{\bar{c}_+^2},\, 0,\,0,\, \bar{\rho}_+,\, -\frac{\bar{\rho}_+\bar{q}_+}{\gamma c_v} \Big)^\top,\label{alpha1}\\
& {\mathbf{\alpha}}_2^+  = \Big(1+\overline{M}_+^2 ,\,0,\, 0,\, 2\bar{\rho}_+ \bar{q}_+, \, -\frac{\bar{\rho}_+\bar{q}_+^2}{\gamma c_v}  \Big)^\top,\quad {\mathbf{\alpha}}_3^+  = \Big(0,\,0,\,1,\,0 ,\,0 \Big)^\top, \label{alpha3}\\
& {\mathbf{\alpha}}_4^+  = \Big(\frac{1}{\bar{\rho}_+},\,0,\,0,\,\bar{q}_+,\, \frac{\bar{p}_+}{(\gamma -1)c_v \bar{\rho}_+}\Big)^\top,\quad {\mathbf{\alpha}}_5^+  = \Big(0,\, \bar{\rho}_+ \bar{q}_+^2,\, 0,\,0,\, 0\Big)^\top.\label{alpha5}
\end{align}
Thus, the boundary conditions \eqref{eq251RH} can be rewritten as
\begin{equation*}
A_s (\delta {p},\delta {w},\delta {q}, \delta {s})^\top =\mathbf{G}\defs  (G_1^{\sharp}, G_2^{\sharp}, G_3^{\sharp}, G_4^{\sharp})^\top,
\end{equation*}
where
\[
A_s:= \begin{pmatrix} \displaystyle\frac{\bar{q}_+}{\bar{c}_+^2}&0 &\bar{\rho}_+ & -\displaystyle\frac{\bar{\rho}_+\bar{q}_+}{\gamma c_v}\\
1+\overline{M}_+^2 &0& 2\bar{\rho}_+ \bar{q}_+ & -\displaystyle\frac{\bar{\rho}_+\bar{q}_+^2}{\gamma c_v} \\
0&1&0&0 \\
\displaystyle\frac{1}{\bar{\rho}_+}& 0&\bar{q}_+ &\displaystyle\frac{\bar{p}_+}{(\gamma -1)c_v \bar{\rho}_+}
\end{pmatrix}.
\]	
It is easy to see that
\begin{align}\label{AMRS}
 \det A_s = \displaystyle\frac{\bar{p}_+}{(\gamma -1)c_v} (1 - \overline{M}_+^2)\neq 0,\quad \text{as}\quad \overline{M}_+\neq 1,
\end{align}
so we have the following boundary conditions on $\Gamma_s$
\begin{equation}\label{5.35xw}
\big( \delta {p},\, \delta {w},\, \delta {q},\,\delta {s}\big)=\big( g_1,\,g_2,\, g_3,\, g_4  \big)\defs A_s^{-1} \mathbf{G}.
\end{equation}
Moreover, it follows from \eqref{eq11000} that
\begin{align}\label{varphi}
\delta {\varphi}' = g_5\defs \left(\displaystyle\frac{  \alpha_5^+ \cdot \delta U- G_5^{\sharp}(\delta U,\delta U_-, \varphi)}{[\bar{p}]}\right).
\end{align}

Finally, because the elliptic equation \eqref{eqf2} is coupled with the hyperbolic equations \eqref{eqf3} and \eqref{eqf5}, we need to further reformulate the elliptic equation \eqref{eqf2}.
Consider the Cauchy problems of equations \eqref{eqf3} and \eqref{eqf5}:
\begin{eqnarray}
&&\begin{cases}\label{eqw}
 \big( \partial_z +  H(z,r)\partial_r\big) (r\delta w)+ \partial_r( r\bar{w}_+ )\cdot \delta \theta ={f}_3(\delta U,  \varphi)&\text{in}\quad {\Omega}_+^K,\\[3pt]
\delta w(K, r)=g_2(K,r);
\end{cases}\\
&&\begin{cases}\label{eqs}
\big( \partial_z +  H(z,r)\partial_r\big) \delta s + \partial_r\bar{s}_+ \delta \theta={f}_5(\delta U,\varphi)& \text{in}\quad {\Omega}_+^K,\\[3pt]
\delta s(K, r)=g_4(K,r).
\end{cases}
\end{eqnarray}

For any fixed point $(z,r)$, let $\tau\mapsto R(\tau;z,{r})$ be the characteristic curve passing through $(z,{r})$, which satisfies
\begin{align*}
\displaystyle\frac{\dif R(\tau;z,{r})}{\dif \tau} =& H(\tau, R(\tau;z,{r})),\quad \tau\in [K, z],\\
R (z;z,{r})=&{r}.
\end{align*}
Then
\begin{equation}\label{eqdeltaw}
\begin{array}{rl}
& r\delta w(z,r) \\[2mm]
= & R(K;z,r)g_2(R(K;z,r))
+ \int_{K}^{z} {f}_3(\tau, R(\tau;z,{r}) )\dif \tau\\[2mm]
& - \int_{K}^{z}\Big( R(\tau;z,{r})\partial_r \bar{w}_+(R(\tau;z,{r})) + \bar{w}_+(R(\tau;z,{r})) \Big)\delta \theta (\tau, R(\tau;z,{r}) )\dif \tau
\end{array}
\end{equation}
and
\begin{equation}\label{eqdeltas}
\begin{array}{rl}
\delta s(z,r) = &g_4(R(K;z,r)) + \int_{K}^{z} {f}_5(\tau, R(\tau;z,{r}) )\dif \tau\\[2mm]
&- \int_{K}^{z} \partial_r\bar{s}_+(R(\tau;z,{r})) \delta \theta(\tau, R(\tau;z,{r}) )\dif \tau.
\end{array}
\end{equation}
Substituting \eqref{eqdeltaw} and \eqref{eqdeltas} into \eqref{eqf2}, we have
\begin{align}\label{eqf22}
&\partial_z \Big(\bar{\rho}_+\bar{q}_+^2  \bar{p}_+^{-\frac{1}{\gamma}}\delta{\theta} \Big) +\partial_r \Big(  \bar{p}_+^{-\frac{1}{\gamma}}\delta{p} \Big)\notag\\
&- \frac{\bar{\rho}_+\bar{w}_+}{r}\Big(\frac{ \partial_r\bar{s}_+}{\gamma c_v}\bar{w}_+ - 2\big(\partial_r \bar{w}_+ + \frac{\bar{w}_+}{r}\big) \Big)\bar{p}_+^{-\frac{1}{\gamma}}\int_{K}^{z} \delta{\theta}(\tau,r)\dif \tau = \bar{p}_+^{-\frac{1}{\gamma}}f_2(\delta U, \delta \varphi),
\end{align}
where
\begin{align}
& f_2(\delta U, \varphi)\notag\\
&\defs {f}_2^{\sharp}(\delta U, \varphi)
 + \frac{2\bar{\rho}_+ \bar{w}_+}{r^2}\Big\{ R(K;z,r)g_2(R(K;z,r))
+ \int_{K}^{z} f_3(\tau, R(\tau;z,{r}) )\dif \tau\notag\\
& -  \int_{K}^{z} \Big( R(\tau;z,{r})\partial_r \bar{w}_+(R(\tau;z,{r})) + \bar{w}_+(R(\tau;z,{r})) \Big)\delta \theta (\tau, R(\tau;z,{r}) )  \dif \tau \notag\\
&+\Big(r \partial_r \bar{w}_+(r) + \bar{w}_+(r)\Big)\int_{K}^{z}\delta \theta (\tau, r ) \dif \tau  \Big\}\notag\\
&-\frac{1}{\gamma c_v}\frac{\bar{\rho}_+ \bar{w}_+^2}{r}
\Big\{ g_4(R(K;z,r) ) + \int_{K}^{z} f_5(\tau, R(\tau;z,{r}) )\dif \tau\notag\\
&- \int_{K}^{z} \partial_r\bar{s}_+( R(\tau;z,{r}) ) \delta \theta(\tau, R(\tau;z,{r}) )\dif \tau
 + \partial_r\bar{s}_+(r )\int_{K}^{z}  \delta \theta(\tau, r)\dif \tau  \Big\}.\label{f2itera}
\end{align}

\begin{rem}\label{rem:5.1x}
There is a nonlocal term $\int_{K}^{z} \delta{\theta}(\tau,r)\dif \tau$ in \eqref{eqf22}. This is quite different from \cite{FG1,FB63,LiXY2010,ParkRyu_2019arxiv,Yong,HPark,WS94} and shows the spectacular influence of non-zero swirl. It introduces stronger coupling in the Euler equations with integral-type nonlocal terms.
\end{rem}

\subsection{Linearized boundary value problems for the two elliptic equations}
In this and the next subsections, we will consider the linearised boundary value problem by assuming that $f_i$ for $i=1,\cdots,5$ and the values of $(p,w,q,s)$ on the shock front are given in some suitable H\"{o}lder norm spaces. Then we will establish the unique existence of the solutions to the linearized equations \eqref{eqf1}, \eqref{eqf3}-\eqref{eqf5} and \eqref{eqf22}, with boundary conditions \eqref{GAMMA4theta}-\eqref{xxl}. 

First, let us consider the boundary value problem for the two elliptic equations \eqref{eqf1} and \eqref{eqf22} in $\Omega_+^K$
\begin{align}
& \partial_z \Big(\frac{1-\overline{M}_+^2}{\bar{\rho}_+\bar{q}_+^2} r \bar{p}_+^{\frac{1}{\gamma}} \delta{p}  \Big) -\partial_r \Big( r \bar{p}_+^{\frac{1}{\gamma}}\delta{\theta}
\Big )  = r \bar{p}_+^{\frac{1}{\gamma}} f_1(z,r),\label{F1eq1}\\
&\partial_z \Big(\bar{\rho}_+\bar{q}_+^2  \bar{p}_+^{-\frac{1}{\gamma}}\delta{\theta} \Big) +\partial_r \Big(  \bar{p}_+^{-\frac{1}{\gamma}}\delta{p} \Big)\notag\\
&- \frac{\bar{\rho}_+\bar{w}_+}{r}\Big(\frac{ \partial_r\bar{s}_+}{\gamma c_v}\bar{w}_+ - 2\big(\partial_r \bar{w}_+ + \frac{\bar{w}_+}{r}\big) \Big)\bar{p}_+^{-\frac{1}{\gamma}}\int_{K}^{z} \delta{\theta}(\tau,r)\dif \tau = \bar{p}_+^{-\frac{1}{\gamma}} f_2(z,r),\label{F2eq2}
\end{align}
where the boundary conditions are
\begin{align}
  \delta{\theta} =0\qquad &\mbox{on }\Gamma_a\cup\Gamma_w\label{thetaK}\\
  \delta{p}=\sigma p_{ex} (r)\qquad&\mbox{on }\Gamma_{ex}\label{pLK}\\
  \delta{p}= g_1 \qquad&\mbox{on } \Gamma_K.\label{pK}
\end{align}
For the given functions $f_1$, $f_2$, $g_1$ and $p_{ex}$, we assume that
\begin{align}
&f_1\in C^{1,\alpha}(\overline{\Omega_+^K}),\quad\partial_r f_1(z,0)=0,\quad \partial_r f_1(z,r_0)=0,\label{F1F20r0}\\
&f_2\in C^{1,\alpha}(\overline{\Omega_+^K}),\quad f_2(z,0) = 0, \quad f_2(z,r_0)=0,\label{bc:F2}\\
& g_1\in C^{2,\alpha}([0,r_0]),\quad p_{ex}\in C^{2,\alpha}([0,r_0]),\label{assumeregularity}\\
&\partial_r g_1(K,0)=0,\quad  \partial_r g_1(K,r_0)=0,\quad \partial_r p_{ex}(L,0)=0,\quad  \partial_r p_{ex}(L,r_0)=0.\label{partialrg1pe}
\end{align}

\begin{rem}
the regularity assumed for $f_i$, $i=1,2$ in \eqref{F1F20r0} is different from the one for $f_i$, $i=3,4,5$ in \eqref{con:regu} due to the symmetry axis singularity of the swirl velocity $w$.
\end{rem}

In this section, we will show the following theorem on the unique existence of solutions for the bondary value problem \eqref{F1eq1}-\eqref{pK}.
\begin{thm}\label{BigThm}
Suppose that conditions \eqref{F1F20r0}-\eqref{partialrg1pe} hold. There exists a positive constant $r_*$, only depending on $\overline{U}_-(r)$ and $\gamma$, such that for $0<r_0 \leq r_*$, there exists a unique solution $(\delta p, \delta \theta)$ for the boundary value problem \eqref{F1eq1}-\eqref{pK}, if and only if
\begin{equation}\label{thmsolvabilityeq}
 \int_0^{r_0} \frac{1-\overline{M}_+^2}{\bar{\rho}_+\bar{q}_+^2}r \bar{p}_+^{\frac{1}{\gamma}} \Big(\sigma p_{ex} (r) - {g}_1(K, r)\Big)\dif r = \int \int_{\Omega_+^K} r \bar{p}_+^{\frac{1}{\gamma}} f_1(z,r)\dif z \dif r.
\end{equation}
Moreover, the solution $(\delta p, \delta \theta)$ satisfies the following estimate:
\begin{align}\label{eq083}
&\| \delta p \|_{C^{2,\alpha}(\overline{\Omega_+^K})} + \| \delta \theta \|_{C^{2,\alpha}(\overline{\Omega_+^K})}\notag\\
\leq& C\Big(  \sum_{i=1}^2 \| f_i \|_{C^{1,\alpha}(\overline{\Omega_+^K})} + \| g_1 \|_{C^{2,\alpha}([0,r_0])} + \| \sigma p_{ex} \|_{C^{2,\alpha}([0,r_0])}\Big),
\end{align}
where the constant $C$ only depends on $\overline{U}_{\pm}$, $K$, $L$, $\gamma$, $r_0$ and $\alpha$.
\end{thm}

To prove Theorem \ref{BigThm}, we need to handle the inhomogeneous terms in the linearised equations \eqref{F1eq1}-\eqref{F2eq2}. To make it, we introduce two auxilliary problems by letting $(\delta p, \delta \theta)=(\delta p_1 +\delta p_2, \delta \theta_1 + \delta \theta_2)$, where $(\delta p_1,\delta\theta_1)$ and $(\delta p_2,\delta\theta_2)$ satisfy the following two linear boundary value problems, respectively. 

{\bf {$\llbracket \textit{Problem I}\rrbracket $}}: $(\delta p_1,\delta\theta_1)$ satisfies the following linear boundary value problem in $\Omega_+^K$
\begin{align}
  &\partial_z \Big(\frac{1-\overline{M}_+^2}{\bar{\rho}_+\bar{q}_+^2} r \bar{p}_+^{\frac{1}{\gamma}} \delta{p}_1  \Big) -\partial_r \Big( r \bar{p}_+^{\frac{1}{\gamma}}\delta{\theta}_1
\Big )  = r \bar{p}_+^{\frac{1}{\gamma}} f_1(z,r),\label{Problem1eq1}\\
  &\partial_z \Big(\bar{\rho}_+\bar{q}_+^2  \bar{p}_+^{-\frac{1}{\gamma}}\delta{\theta}_1 \Big) +\partial_r \Big(  \bar{p}_+^{-\frac{1}{\gamma}}\delta{p}_1 \Big)=0,\label{Problem1eq2}
  \end{align}
where the boundary conditions are
\begin{align}
  &\delta{\theta}_1 =0\qquad \mbox{on }\Gamma_a\cup\Gamma_w\label{Problem1bry1}\\
  &\delta{p}_1 = \sigma p_{ex} \qquad\mbox{on }\Gamma_{ex}\label{Problem1bry2}\\
  &\delta{p}_1= g_1 \qquad\mbox{on } \Gamma_K.\label{Problem1bry3}
\end{align}

{\bf {$\llbracket \textit{Problem II}\rrbracket $}}: $(\delta p_2,\delta\theta_2)$ satisfies the following linear boundary value problem in $\Omega_+^K$
\begin{align}
  &\partial_z \Big(\frac{1-\overline{M}_+^2}{\bar{\rho}_+\bar{q}_+^2} r \bar{p}_+^{\frac{1}{\gamma}} \delta{p}_2  \Big) -\partial_r \Big( r \bar{p}_+^{\frac{1}{\gamma}}\delta{\theta}_2\Big ) =0,\label{Problem2eq1}\\
  &\partial_z \Big(\bar{\rho}_+\bar{q}_+^2  \bar{p}_+^{-\frac{1}{\gamma}}\delta{\theta}_2 \Big) +\partial_r \Big(  \bar{p}_+^{-\frac{1}{\gamma}}\delta{p}_2 \Big)\notag\\
&- \frac{\bar{\rho}_+\bar{w}_+}{r}\Big(\frac{ \partial_r\bar{s}_+}{\gamma c_v}\bar{w}_+ - 2\big(\partial_r \bar{w}_+ + \frac{\bar{w}_+}{r}\big) \Big)\bar{p}_+^{-\frac{1}{\gamma}}\int_{K}^{z} \delta{\theta}_2(\tau,r)\dif \tau\notag\\
=& \bar{p}_+^{-\frac{1}{\gamma}} f_2(z,r) + \frac{\bar{\rho}_+\bar{w}_+}{r} \Big(\frac{\partial_r\bar{s}_+}{\gamma c_v}\bar{w}_+ - 2\big(\partial_r \bar{w}_+ + \frac{\bar{w}_+}{r}\big) \Big)\bar{p}_+^{-\frac{1}{\gamma}}\int_{K}^{z} \delta{\theta}_1(\tau,r)\dif \tau,\label{Problem2eq2}
  \end{align}
where the boundary conditions are
\begin{align}
  &\delta{\theta}_2 =0\qquad\mbox{on }\Gamma_a\cup\Gamma_w\label{Problem2bry1}\\
  &\delta{p}_2= 0\qquad\mbox{on } \Gamma_{ex}\label{Problem2bry2}\\
  &\delta{p}_2 = 0\qquad \mbox{on }\Gamma_K.\label{Problem2bry3}
\end{align}

Then Theorem \ref{BigThm} will be proven by establishing the unique existence of solutions of {\bf {$\llbracket \textit{Problem I}\rrbracket $}} and {\bf {$\llbracket \textit{Problem II}\rrbracket $}}. It is Lemma \ref{3Dphi}, Lemma \ref{theta2p2} and Lemma \ref{raiseRegularity}.

\subsubsection{Well-posedness of {\bf {$\llbracket \textit{Problem I}\rrbracket $}}}
By \eqref{Problem1eq2}, there exists a potential function $\Phi$ such that
\begin{align}\label{Phieq}
  \nabla \Phi = (\partial_z \Phi, \partial_r \Phi) = \Big( \bar{p}_+^{-\frac{1}{\gamma}} \delta{p}_1, - \bar{\rho}_+\bar{q}_+^2   \bar{p}_+^{-\frac{1}{\gamma}} \delta{\theta}_1\Big).
\end{align}
Substituting \eqref{Phieq} into \eqref{Problem1eq1}, we have
\begin{align}\label{phii}
 \partial_z \Big(\frac{1-\overline{M}_+^2}{\bar{\rho}_+\bar{q}_+^2}  \bar{p}_+^{\frac{2}{\gamma}}\partial_z \Phi \Big) + \partial_r \Big( \frac{\bar{p}_+^{\frac{2}{\gamma}}}{\bar{\rho}_+\bar{q}_+^2 } \partial_r \Phi\Big ) + \frac{1}{r} \frac{\bar{p}_+^{\frac{2}{\gamma}}}{\bar{\rho}_+\bar{q}_+^2 } \partial_r \Phi = \bar{p}_+^{\frac{1}{\gamma}} f_1,
\end{align}
where the boundary conditions become
\begin{align}
\partial_r \Phi=0\qquad& \mbox{on }\Gamma_a\cup\Gamma_w\label{IBC22}\\
\partial_z \Phi =\bar{p}_+^{-\frac{1}{\gamma}}g_1 (K, r)\qquad&\mbox{on }\Gamma_K\label{IBC12}\\
\partial_z \Phi = \sigma \bar{p}_+^{-\frac{1}{\gamma}}p_{ex}( r)\qquad &\mbox{on }\Gamma_{ex}.\label{IBC32}
\end{align}
Obviously, to study {\bf {$\llbracket \textit{Problem I}\rrbracket $}}, it is sufficient to study boundary value problem \eqref{phii}-\eqref{IBC32}. To deal with the singularity at the symmetry axis $r=0$, let us consider the problem in three dimensions. That is, introduce
\begin{align*}
  \mathbf{x}= (x_1,x_2,x_3)\defs (z, r\cos \tau, r\sin\tau).
\end{align*}
  Define
 \begin{align*}
   \mathcal{P}\defs & \big\{(x_1, x_2, x_3)\in \mathbb{R}^3 : K< x_1 <L, \, 0\leq x_2^2 + x_3^2<r_0 \big\},\\
    \mathcal{P}_{K}\defs & \big\{(x_1, x_2, x_3)\in \mathbb{R}^3 : x_1 = K, \, 0\leq x_2^2 + x_3^2<r_0 \big\},\\
     \mathcal{P}_{ex}\defs & \big\{(x_1, x_2, x_3)\in \mathbb{R}^3 : x_1 =L, \, 0\leq x_2^2 + x_3^2<r_0 \big\},\\
     \mathcal{P}_{w}\defs & \big\{(x_1, x_2, x_3)\in \mathbb{R}^3 : K< x_1 <L, \,  x_2^2 + x_3^2=r_0 \big\}.
 \end{align*}
Let $\phi(x_1,x_2,x_3)\defs \Phi(x_1, \sqrt{x_2^2 + x_3^2})$.
Then problem \eqref{phii}-\eqref{IBC32} can be rewritten as the following {\bf {$\llbracket \textit{Problem I(3D)}\rrbracket $}}:
\begin{align}\label{phiequa}
\partial_{x_1} \Big(\frac{1-\overline{M}_+^2}{\bar{\rho}_+\bar{q}_+^2}  \bar{p}_+^{\frac{2}{\gamma}}\partial_{x_1} \phi \Big) + \sum_{i=2}^3\partial_{x_i}
\Big( \frac{\bar{p}_+^{\frac{2}{\gamma}}}{\bar{\rho}_+\bar{q}_+^2 } \partial_{x_i}\phi\Big ) =\bar{p}_+^{\frac{1}{\gamma}} f_1(x_1,r),\quad \text{in} \quad \mathcal{P},
\end{align}
with the boundary conditions:
\begin{align}
   &x_2 \partial_{x_2}\phi(x_1,x_2,x_3) +  x_3 \partial_{x_3}\phi(x_1,x_2,x_3) =0, &\quad &\text{on}&\quad \mathcal{P}_w,\label{pw3D}\\
   &\partial_{x_1}\phi(K,x_2,x_3) =\bar{p}_+^{-\frac{1}{\gamma}} g_1(K, r), &\quad &\text{on}&\quad \mathcal{P}_{K},\\
   &\partial_{x_1}\phi(L,x_2,x_3) = \sigma \bar{p}_+^{-\frac{1}{\gamma}}p_{ex}(r), &\quad &\text{on}&\quad \mathcal{P}_{ex}.\label{pex3D}
\end{align}
It follows from assumptions \eqref{F1F20r0}-\eqref{partialrg1pe} and Lemmas \ref{barvalue} and \ref{lem5.4x} that the coefficients in equation \eqref{phiequa} are $C^{2,\alpha}(\overline{\mathcal{P}})$ and
$$
\|f_1\|_{C^{1,\alpha}(\overline{\mathcal{P}})} +
\|g_1\|_{C^{2,\alpha}(\overline{\mathcal{P}_K})} + \|\sigma p_{ex}\|_{C^{2,\alpha}(\overline{\mathcal{P}_{ex}})} \leq C(\|f_1\|_{C^{1,\alpha}(\overline{\Omega_+^K})} + \|g_1\|_{C^{2,\alpha}([0,r_0])} + \|\sigma p_{ex}\|_{C^{2,\alpha}([0,r_0])}).
$$

For the boundary value problem \eqref{phiequa}-\eqref{pex3D}, we have the following Lemma.
\begin{lem}\label{3Dphi}
There exists a unique weak solution $\phi\in H^1(\mathcal{P})$ up to a constant for the {\bf {$\llbracket \textit{Problem I(3D)}\rrbracket $}} if and only if
\begin{align}\label{solvability}
\int_0^{r_0} \frac{1-\overline{M}_+^2}{\bar{\rho}_+\bar{q}_+^2}r \bar{p}_+^{\frac{1}{\gamma}} \Big( \sigma p_{ex}(r) - {g}_1(K, r)\Big)\dif r = \int_0^{r_0}\int_{K}^{L} r \bar{p}_+^{\frac{1}{\gamma}} f_1(z,r)\dif z \dif r.
\end{align}
Select the solution $\phi$ with the additional condition that $\int_{\mathcal{P}} \phi \dif \mathbf{x}=0$.
Then if the conditions \eqref{assumeregularity}, \eqref{F1F20r0} and \eqref{partialrg1pe} hold, any weak solution $\phi\in H^1(\mathcal{P})$ of the {\bf {$\llbracket \textit{Problem I(3D)}\rrbracket $}} has a higher regularity with $\phi\in C^{2,\alpha}(\overline{\mathcal{P}})$  and satisfies the following estimates
\begin{align}\label{phic3}
\|\phi\|_{C^{2,\alpha}(\overline{\mathcal{P}})} \leq C \Big(\|f_1\|_{C^{1,\alpha}(\overline{\mathcal{P}})} +
\|g_1\|_{C^{2,\alpha}(\overline{\mathcal{P}_K})} + \|\sigma p_{ex}\|_{C^{2,\alpha}(\overline{\mathcal{P}_{ex}})}\Big),
\end{align}
where the constant $C$ depends on $\overline{U}_\pm$, $\gamma$, $r_0$, $L$ and $\alpha$. Moreover,
\begin{equation}\label{5.75x}
\phi(x_1,x_2,x_3)\equiv\Phi(x_1, \sqrt{x_2^2 + x_3^2}),
\end{equation}
where $\Phi\in C^{2,\alpha}(\Omega_+^K)$ is a solution of {\bf {$\llbracket \textit{Problem I}\rrbracket $}}, and it holds that
\begin{align}\label{theta1p13D}
&\|\delta \theta_1\|_{C^{1,\alpha}(\overline{\Omega_+^K})} + \|\delta p_1\|_{C^{1,\alpha}(\overline{\Omega_+^K})} \notag\\
\leq & C \Big(\|f_1\|_{C^{1,\alpha}(\overline{\Omega_+^K})} +
    \|g_1\|_{C^{2,\alpha}([0,r_0])} + \|\sigma p_{ex}\|_{C^{2,\alpha}([0,r_0])}\Big),
\end{align}
where the constant $C$ only depends on $\overline{U}_\pm$, $K$, $L$, $\gamma$, $r_0$ and $\alpha$.
\end{lem}

\begin{proof}
This proof is divided into two steps.

\emph{Step 1.} In this step, we will show the existence of weak solution  $\phi\in H^1(\mathcal{P})$.
Let
\begin{align*}
  G(x_1, x_2, x_3) = \bar{p}_+^{-\frac{1}{\gamma}}\Big(\frac{g_1(K,r) - \sigma p_{ex}( r)}{2(K-L)}(x_1-K)^2 + g_1(K,r)x_1\Big).
\end{align*}
Let
$$
\phi_* \defs \phi - G.
$$
Then it follows from \eqref{partialrg1pe} that $\phi_*$ satisfies the following boundary value problem
\begin{align}\label{phiequa*}
\partial_{x_1} \Big(\frac{1-\overline{M}_+^2}{\bar{\rho}_+\bar{q}_+^2}  \bar{p}_+^{\frac{2}{\gamma}}\partial_{x_1} \phi_* \Big) + \sum_{i=2}^3\partial_{x_i}
\Big( \frac{\bar{p}_+^{\frac{2}{\gamma}}}{\bar{\rho}_+\bar{q}_+^2 } \partial_{x_i}\phi_*\Big ) = F_*,\quad \text{in} \quad \mathcal{P},
\end{align}
where
\begin{align*}
F_* \defs \bar{p}_+^{\frac{1}{\gamma}} f_1(x_1,r) - \partial_{x_1} \Big(\frac{1-\overline{M}_+^2}{\bar{\rho}_+\bar{q}_+^2}  \bar{p}_+^{\frac{2}{\gamma}}\partial_{x_1} G \Big) - \sum_{i=2}^3\partial_{x_i}
\Big( \frac{\bar{p}_+^{\frac{2}{\gamma}}}{\bar{\rho}_+\bar{q}_+^2 }
\partial_{x_i} G\Big ),
\end{align*}
with boundary conditions
\begin{align}
   &x_2 \partial_{x_2}\phi_*(x_1,x_2,x_3) +  x_3 \partial_{x_3}\phi_*(x_1,x_2,x_3) =0, &\quad &\text{on}&\quad \mathcal{P}_w,\label{pw3D*}\\
   &\partial_{x_1}\phi_*(K,x_2,x_3) =0, &\quad &\text{on}&\quad \mathcal{P}_{K},\\
   &\partial_{x_1}\phi_*(L,x_2,x_3) = 0, &\quad &\text{on}&\quad \mathcal{P}_{ex}.\label{pex3D*}
\end{align}

$\phi_*$ is a weak solution of the problem \eqref{phiequa*}-\eqref{pex3D*} if and only if for any test function $\zeta\in H^1(\mathcal{P})$,
\begin{align}\label{zetaphi*}
-\int_{\mathcal{P}} \Big(\frac{1-\overline{M}_+^2}{\bar{\rho}_+\bar{q}_+^2}\bar{p}_+^{\frac{2}{\gamma}} \partial_{x_1} \phi \partial_{x_1}\zeta  + \sum_{i=2}^3 \frac{\bar{p}_+^{\frac{2}{\gamma}} }{\bar{\rho}_+\bar{q}_+^2 } \partial_{x_i}\phi\partial_{x_i}\zeta \Big)\dif \mathbf{x} =  \int_{\mathcal{P}} F_* \zeta \dif \mathbf{x}.
\end{align}
By choosing $\zeta =1$ in \eqref{zetaphi*}, we have that $\int_{\mathcal{P}} F_* \dif \mathbf{x}=0$, that is
\begin{align}\label{solvability*}
 \int_{\mathcal{P}} \bar{p}_+^{\frac{1}{\gamma}} F_1(x_1,r) \dif \mathbf{x} =& \int_{\mathcal{P}} \Big(\partial_{x_1} \big(\frac{1-\overline{M}_+^2}{\bar{\rho}_+\bar{q}_+^2}  \bar{p}_+^{\frac{2}{\gamma}}\partial_{x_1} G \big) + \sum_{i=2}^3\partial_{x_i}
\big( \frac{\bar{p}_+^{\frac{2}{\gamma}}}{\bar{\rho}_+\bar{q}_+^2 } \partial_{x_i} G\big )\Big)\dif \mathbf{x}\notag\\
=& \int_{\mathcal{P}_{ex}} \frac{1-\overline{M}_+^2}{\bar{\rho}_+\bar{q}_+^2} \bar{p}_+^{\frac{2}{\gamma}}\partial_{x_1} G  \dif S
-\int_{\mathcal{P}_{K}} \frac{1-\overline{M}_+^2}{\bar{\rho}_+\bar{q}_+^2} \bar{p}_+^{\frac{2}{\gamma}}\partial_{x_1} G  \dif S.
\end{align}
It is condition \eqref{solvability}.

Next, we will prove that if \eqref{solvability*} holds (equivalently \eqref{solvability} holds), there exists a weak solution $\phi_*$ to the problem \eqref{phiequa*}-\eqref{pex3D*}. Define
\begin{align*}
\mathcal{S} \defs \Big\{\phi_* \in H^1 (\mathcal{P})\Big| \int_{\mathcal{P}} \phi_* \dif \mathbf{x}=0\Big\}.
\end{align*}
Let
\begin{align*}
   \mathcal{B}[\phi, \zeta] \defs -\int_\mathcal{P}  \Big(\frac{1-\overline{M}_+^2}{\bar{\rho}_+\bar{q}_+^2}\bar{p}_+^{\frac{2}{\gamma}} \partial_{x_1} \phi \partial_{x_1}\zeta  + \sum_{i=2}^3 \frac{\bar{p}_+^{\frac{2}{\gamma}} }{\bar{\rho}_+\bar{q}_+^2 } \partial_{x_i}\phi\partial_{x_i}\zeta \Big)\dif \mathbf{x}.
  \end{align*}
It is easy to see that for any $\phi\in\mathcal{S}$ and $\zeta\in H^1(\mathcal{P})$
\begin{align*}
  | \mathcal{B}[\phi, \zeta]|\leq  3 \Big\|  \frac{\bar{p}_+^{\frac{2}{\gamma}}}{\bar{\rho}_+\bar{q}_+^2 } \Big\|_{L^\infty([0,r_0])} \int_\mathcal{P} |D \phi| |D \zeta|\dif \mathbf{x}
  \leq C\|\phi\|_{H^1(\mathcal{P})} \|\zeta\|_{H^1(\mathcal{P})},
\end{align*}
where the positive constant $C$ only depends on $\overline{U}_+$.

On the other hand, by the Poincar\'{e} inequality, we have for any $\phi\in\mathcal{S}$
\begin{align*}
\|\phi\|_{L^2(\mathcal{P})}^2 = \int_{\mathcal{P}} |\phi - \int_{\mathcal{P}} \phi \dif \mathbf{x}|^2 \dif \mathbf{x} \leq C\int_{\mathcal{P}}|D\phi|^2 \dif \mathbf{x}.
\end{align*}
Thus for any $\phi\in\mathcal{S}$
\begin{align*}
  |\mathcal{B} [\phi, \phi]| \geq C \|\phi\|_{H^1(\mathcal{P})}^2.
\end{align*}
By applying Lax-Milgram theorem, for all $F_*\in L^2$ with $\int_{\mathcal{P}} F_* \dif \mathbf{x}=0$, \emph{i.e.}, \eqref{solvability*} holds, there exists a unique $\phi_* \in \mathcal{S}$ such that for any $\zeta\in H^1(\mathcal{P})$
\begin{align*}
(F_*, \zeta)_{L^2(\mathcal{P})} = &(F_*, \int_{\mathcal{P}} \zeta \dif \mathbf{x})_{L^2(\mathcal{P})} + (F_*, \zeta - \int_{\mathcal{P}} \zeta \dif \mathbf{x})_{L^2(\mathcal{P})} \notag\\
= &B[\phi_*, \zeta - \int_{\mathcal{P}} \zeta \dif \mathbf{x}] =B[\phi_*, \zeta].
\end{align*}
Therefore, $\phi_* \in \mathcal{S} \subset H^1(\mathcal{P})$ is a weak solution. Notice that $(\phi_* + \text{constant.})$ is also a weak solution and the solution is unique in $\mathcal{S}$. So solution $\phi_* \in H^1(\mathcal{P})$ is unique up to a constant.
Moreover, choose the solution $\phi$ with the additional condition that $\int_{\mathcal{P}} \phi \dif \mathbf{x}=0$, then the following estimate holds
\begin{align}\label{H13D}
  \|\phi\|_{H^1(\overline{\mathcal{P}})} \leq C \Big(\|f_1\|_{C^{1,\alpha}(\overline{\mathcal{P}})} +\|g_1\|_{C^{2,\alpha}(\overline{\mathcal{P}_K})} + \|\sigma p_{ex}\|_{C^{2,\alpha}(\overline{\mathcal{P}_{ex}})} \Big),
\end{align}
where the constant ${C}$ only depends on $\overline{U}_\pm$, $\gamma$, $r_0$ and $\alpha$.

\emph{Step 2.} In this step, we will prove \eqref{phic3}-\eqref{theta1p13D}. Notice that we can multiply corresponding coefficient functions in \eqref{phiequa*} to make the boundary value problem \eqref{phiequa*}-\eqref{pex3D*} conormal, such that we can apply the results in Chapter 5 in \cite{GM13}. First, by \cite[Theorem 5.36 and Theorem 5.45]{GM13} with $u$ and $-u$ to have
\begin{align}\label{H1L6}
    \|\phi\|_{L^{\infty}(\overline{\mathcal{P}})} \leq C \Big(\|f_1\|_{C^{1,\alpha}(\overline{\mathcal{P}})} +\|g_1\|_{C^{2,\alpha}(\overline{\mathcal{P}_{K}})} + \|\sigma p_{ex}\|_{C^{2,\alpha}(\overline{\mathcal{P}_{ex}})} +  \|\phi\|_{{H^1}(\overline{\mathcal{P}})}   \Big).
  \end{align}
Then we can apply \cite[Theorem 5.36 and Theorem 5.45]{GM13} for $\phi$ locally near the boundary and apply the standard De Giorgi type interior H\"{o}lder estimates (see \cite[Theorem 8.24]{DNS})  to have the local H\"{o}lder estimates
\begin{align}\label{0alpha}
    \|\phi\|_{C^{0,\alpha}(\overline{\mathcal{P}})} \leq C \Big(\|f_1\|_{C^{1,\alpha}(\overline{\mathcal{P}})} +\|g_1\|_{C^{2,\alpha}(\overline{\mathcal{P}_{K}})} + \|\sigma p_{ex}\|_{C^{2,\alpha}(\overline{\mathcal{P}_{ex}})} +  \|\phi\|_{{H^1}(\overline{\mathcal{P}})}   \Big).
  \end{align}
Now we can apply \cite[Theorem 4.6]{GM13} to have
  \begin{align}\label{phic0}
    \|\phi\|_{C^{1,\alpha}(\overline{\mathcal{P}})} \leq C \Big(\|f_1\|_{C^{1,\alpha}(\overline{\mathcal{P}})} +\|g_1\|_{C^{2,\alpha}(\overline{\mathcal{P}_{K}})} + \|\sigma p_{ex}\|_{C^{2,\alpha}(\overline{\mathcal{P}_{ex}})} +  \|\phi\|_{{H^1}(\overline{\mathcal{P}})}   \Big).
  \end{align}

Next, we will improve the regularity of the solution $\phi$ to establish \eqref{phic3}. Because the interior estimates are standard, we only estimate $\phi$ when point $\mathbf{x}$ is near the points $(K,r_0)$ or $(L,r_0)$. Without loss of the generality and for the shortness, we only consider the point  near $(K, r_0)$, since the other cases can be treated similarly.

By boundary condition $\partial_r\Phi(x_1,r_0)=0$ on the nozzle wall, we can extend the solution across the nozzle wall evenly by defining
\begin{align*}
\widetilde{\Phi}(x_1,r) \defs \begin{cases}
\Phi(x_1,r),\quad 0<r<r_0,\\
 \Phi (x_1, 2r_0-r),\quad r\geq r_0.
\end{cases}
\end{align*}
Let
$$\widetilde{\phi}(x_1,x_2,x_3)\defs \widetilde{\Phi}(x_1, \sqrt{x_2^2 + x_3^2}) = \widetilde{\Phi}(x_1,r).$$
Similarly, by assumptions $\partial_r f_1(z,r_0)=0$ and $\partial_r g_1(z,r_0)=0$, we also define
\begin{align*}
\widetilde{f}_1(x_1,r) \defs \begin{cases}
f_1(x_1,r),\quad 0<r<r_0,\\
f_1(x_1, 2r_0-r),\quad r\geq r_0.
\end{cases}
\widetilde{g}_1(K,r) \defs \begin{cases}
g_1(K,r),\quad 0<r<r_0,\\
g_1(K, 2r_0-r),\quad r\geq r_0.
\end{cases}
\end{align*}
Let $\mathcal{P}_\tau \cap \mathcal{P}$ be a small neighbourhood near point $(K, r_0)$. Let the extended domain be
\begin{align*}
  \widetilde{\mathcal{P}_\tau}\defs \big(  \mathcal{P}_\tau \cap \mathcal{P} \big)  \cup \big\{(z,2r_0-r): (z,r)\in  \big(  \mathcal{P}_\tau \cap \mathcal{P} \big)   \big\}.
\end{align*}
Then it follows from Lemma \ref{barvalue} that the extended functions satisfy the following boundary value problem
\begin{align*}
  \partial_{x_1} \Big(\frac{1-\bar{M}_+^2}{\bar{\rho}_+\bar{q}_+^2}  \bar{p}_+^{\frac{2}{\gamma}}\partial_{x_1}\widetilde{\phi} \Big) + \sum_{i=2}^3\partial_{x_i}
\Big( \frac{\bar{p}_+^{\frac{2}{\gamma}}}{\bar{\rho}_+\bar{q}_+^2 } \partial_{x_i}\widetilde{\phi}\Big ) =  \bar{p}_+^{\frac{1}{\gamma}}\widetilde{f_1}\qquad &\text{in}\quad \widetilde{\mathcal{P}_\tau},\\
 \bar{p}_+^{\frac{1}{\gamma}}\partial_{x_1}\widetilde\phi(K,x_2,x_3) =\widetilde{g}_1(K, r), \qquad &\text{on}\quad  x_1 = K.
\end{align*}
Then for any smooth domain $\widetilde{\mathcal{P}_\tau}^{\sharp} \subset\widetilde{\mathcal{P}_\tau}$, we easily have
\begin{align}\label{Psharp}
  \|\widetilde{\phi}\|_{C^{2,\alpha}(\widetilde{\mathcal{P}_\tau}^{\sharp})} \leq C \Big(\|\widetilde{f_1} \|_{C^{1,\alpha}(\overline{\widetilde{\mathcal{P}_\tau}})} + \|\widetilde{g}_1\|_{C^{2,\alpha}(x_1 = K)} + \|\widetilde{\phi}\|_{C^{1,\alpha}(\partial \widetilde{\mathcal{P}_\tau} \backslash \{x_1 = K\})} \Big).
\end{align}
So it follows from \eqref{phic0} and \eqref{Psharp} that
\begin{align}\label{phic2}
 \|\widetilde{\phi}\|_{C^{2,\alpha}(\mathcal{P}_\tau^{\sharp})} \leq    \|\widetilde{\phi}\|_{C^{2,\alpha}(\widetilde{\mathcal{P}_\tau}^{\sharp})} \leq C \Big( \|\widetilde{f_1} \|_{C^{1,\alpha}(\overline{\widetilde{\mathcal{P}_\tau}})} + \|\widetilde{g}_1\|_{C^{2,\alpha}(x_1 = K)} +  \|\phi\|_{{H^1}(\overline{\mathcal{P}})} \Big),
\end{align}
where $\mathcal{P}_\tau^{\sharp}=\widetilde{\mathcal{P}_\tau}^{\sharp}\cap\overline{\mathcal{P}}$.
Now by applying the standard identity decomposition technique to combining the local estimates,  similar to\eqref{phic2}, near all the four corner points $(K,0)$ ,$(K,r_0)$, $(L,0)$ and $(L,r_0)$, we can obtain \eqref{phic3} from \eqref{H13D}.

By the uniqueness of the solutions of problem \eqref{phiequa}-\eqref{pex3D} by selecting the solutions to satisfy $\int_{\mathcal{P}}\phi\dif \mathbf{x}=0$, it follows from the fact that \eqref{phiequa}-\eqref{pex3D} is rotational invariant, that is, \eqref{5.75x} holds. It follows from Lemma \ref{lem5.4x} and the regularity of $\phi$ that boundary conditions \eqref{IBC22} hold. Moreover, \eqref{phic3} yields that
 \begin{align}\label{Phic3}
\|\Phi\|_{C^{2,\alpha}(\overline{\Omega_+^K})} \leq C \Big(\|f_1\|_{C^{1,\alpha}(\overline{\Omega_+^K})} +
\|g_1\|_{C^{2,\alpha}([0,r_0])} + \|\sigma p_{ex}\|_{C^{2,\alpha}([0,r_0])}\Big).
\end{align}
Finally, \eqref{theta1p13D} follows from \eqref{Phieq} and \eqref{Phic3}.

\end{proof}

\subsubsection{Well-posedness of {\bf {$\llbracket \textit{Problem II}\rrbracket $}}}
Now let us consider {\bf {$\llbracket \textit{Problem II}\rrbracket $}}.
By \eqref{Problem2eq1}, there exists a potential function $\widehat{\Psi}$ such that
\begin{align*}
  \nabla \widehat{\Psi} = (\partial_z \widehat{\Psi}, \partial_r \widehat{\Psi}) =& \Big( r \bar{p}_+^{\frac{1}{\gamma}} \delta{\theta}_2, \frac{1-\overline{M}_+^2}{\bar{\rho}_+\bar{q}_+^2}  r\bar{p}_+^{\frac{1}{\gamma}}\delta{p}_2\Big),
\end{align*}
which implies that
\begin{align}\label{thetaphat}
  \delta{\theta}_2 = \frac{\bar{p}_+^{-\frac{1}{\gamma}} }{r }\partial_z \widehat{\Psi}, \quad
  \delta{p}_2 = \frac{\bar{\rho}_+\bar{q}_+^2}{1-\overline{M}_+^2}
  \frac{\bar{p}_+^{-\frac{1}{\gamma}}}{r }\partial_r \widehat{\Psi}.
\end{align}
Then boundary conditions \eqref{Problem2bry1}-\eqref{Problem2bry3} become the tangential derivative of $\widehat{\Psi}$ along all the boundaries is zero. So $\widehat{\Psi}$ is a constant along the boundary.
Note that $\widehat{\Psi}$ is unique up to a constant, so by substracting a possible constant, the boundary conditions \eqref{Problem2bry1}-\eqref{Problem2bry3} become
\begin{align}
  \widehat{\Psi} =0 \qquad \text{on}\quad \partial\Omega_+^K.
\end{align}
Let $\Psi= \frac{\widehat{\Psi}}{r^2}$, then
\begin{align}\label{PsiBry0}
  \Psi = 0\qquad \text{on}\quad \partial\Omega_+^K\backslash\Gamma_a.
\end{align}

On $\Gamma_a$, it follows from Lemma \ref{barvalue}, equation \eqref{Problem2eq2} and boundary conditions \eqref{Problem1bry1} and \eqref{Problem2bry1} that $\partial_r(\delta p_2)(z,0)=0$. Note that
\[
\Psi(z,0)=\lim_{r\rightarrow0}\frac{\widehat{\Psi}}{r^2}=\lim_{r\rightarrow0}\frac{\widehat{\Psi}_r}{2r}= \frac{(1-\overline{M}_+^2)\bar{p}_+^{\frac{1}{\gamma}}}{2\bar{\rho}_+\bar{q}_+^2}(0) \delta p_2(z,0).
\]
Then
\begin{align*}
 \Psi_r(z,0) =& \lim\limits_{r\rightarrow 0} \frac{\Psi(z,r) - \Psi(z,0)}{r}\notag\\
   = &\lim\limits_{r\rightarrow 0} \frac{\widehat{\Psi}(z,r) -  \frac12 \frac{1-\overline{M}_+^2(0)}{\bar{\rho}_+(0)\bar{q}_+^2(0)} \bar{p}_+^{\frac{1}{\gamma}}(0) \delta{p}_2(z,0) r^2}{r^3}\notag\\
  =&\lim\limits_{r\rightarrow 0} \frac{\widehat{\Psi}_r(z,r) - \frac{1-\overline{M}_+^2(0)}{\bar{\rho}_+(0)\bar{q}_+^2(0)} \bar{p}_+^{\frac{1}{\gamma}}(0) \delta{p}_2(z,0) r}{3r^2}\notag\\
  =& \lim\limits_{r\rightarrow 0} \frac{\frac{1-\overline{M}_+^2(r)}{\bar{\rho}_+(r)\bar{q}_+^2(r)} \bar{p}_+^{\frac{1}{\gamma}}(r) \delta{p}_2(z,r) - \frac{1-\overline{M}_+^2(0)}{\bar{\rho}_+(0)\bar{q}_+^2(0)} \bar{p}_+^{\frac{1}{\gamma}}(0) \delta{p}_2(z,0) }{3r}\notag\\
  =& 0.
\end{align*}

So we have the boundary condition
\begin{equation}\label{5.93x}
\Psi_r=0\qquad\mbox{on }\Gamma_a.
\end{equation}

Moreover, it follows from \eqref{thetaphat} that
\begin{align}\label{PEQ}
   \delta{\theta}_2 =r \bar{p}_+^{-\frac{1}{\gamma}} \partial_z \Psi, \quad
  \delta{p}_2 = \frac{\bar{\rho}_+\bar{q}_+^2}{1-\overline{M}_+^2} \bar{p}_+^{-\frac{1}{\gamma}}\Big( r \partial_r \Psi + 2\Psi \Big).
\end{align}
By plugging \eqref{PEQ} into equation \eqref{Problem2eq2}, we know $\Psi$ satisfies the following equation
\begin{align}\label{RRproblem2}
   &\partial_z \Big(\bar{\rho}_+\bar{q}_+^2 \bar{p}_+^{-\frac{2}{\gamma}}\partial_z {\Psi}\Big) +\frac{1}{r} \partial_r \Big(  \frac{\bar{\rho}_+\bar{q}_+^2}{1-\overline{M}_+^2}
   \bar{p}_+^{-\frac{2}{\gamma}}
   \big( r \partial_r {\Psi} + 2 {\Psi}\big)\Big)\notag\\
    &-\frac{\bar{\rho}_+\bar{w}_+}{r}\Big(\frac{\partial_r\bar{s}_+}{\gamma c_v}\bar{w}_+  - 2\big(\partial_r \bar{w}_+ + \frac{\bar{w}_+}{r}\big) \Big)\bar{p}_+^{-\frac{2}{\gamma}} {\Psi} = {F}^{\sharp}(z,r).
\end{align}
where
${F}^{\sharp}(z,r)\defs \frac{\bar{p}_+^{-\frac{1}{\gamma}} {F}_2(z,r)}{r}$
and
\begin{equation}\label{5.95x}
{F}_2(z,r)\defs
f_2(z,r) + \frac{\bar{\rho}_+\bar{w}_+}{r} \Big(\frac{\partial_r\bar{s}_+}{\gamma c_v}\bar{w}_+ - 2\big(\partial_r \bar{w}_+ + \frac{\bar{w}_+}{r}\big) \Big)\int_{K}^{z} \delta{\theta}_1(\tau,r)\dif \tau.
\end{equation}

By \eqref{BCRH4} and \eqref{assumew0r0=0}, we know that $\bar{w}_+(0) =\bar{w}_-(0) =0$. So
\[
\frac{\bar{w}_{\pm}(r)}{r}=\frac{\bar{w}_{\pm}(r)-\bar{w}_{\pm}(0)}{r}=\int_0^1\partial_r(\bar{w}_{\pm})(\tau r)d\tau.
\]
Then it follows from the fact that $\overline{U}_{\pm}(r)$ are $C^7$-functions that $\frac{\bar{w}_{+}}{r}\in C^6(\overline{\Omega_+^K})$.  By \eqref{bc:F2}, \eqref{theta1p13D} and \eqref{5.95x}, we know that ${F}_2(z,r) \in C^{1,\alpha}(\overline{\Omega_+^K})$ and
\begin{equation}\label{5.110x}
\|{F}_2\|_{C^{1,\alpha}(\overline{\Omega_+^K})}\leq C(\|f_2\|_{C^{1,\alpha}(\overline{\Omega_+^K})}+\|f_1\|_{C^{1,\alpha}(\overline{\Omega_+^K})} +
    \|g_1\|_{C^{2,\alpha}([0,r_0])} + \|\sigma p_{ex}\|_{C^{2,\alpha}([0,r_0])}).
\end{equation}
It follows from \eqref{bc:F2} and $\delta\theta_1(z,0)=0$ that $ {F}_2(z,0) =0$.
Then
\begin{align}\label{5.98x}
 {F}^{\sharp} =  \frac{\bar{p}_+^{-\frac{1}{\gamma}} {F}_2(z,r)}{r} =\bar{p}_+^{-\frac{1}{\gamma}}\frac{{F}_2(z,r) - {F}_2(z,0)}{r}.
\end{align}
So it follows from \eqref{5.110x} that
\begin{equation}\label{5.123xw}
\begin{array}{rl}
&\|F^{\sharp}\|_{C^{0,\alpha}(\Omega_+^K)}\\[2mm]
\leq & C\|\partial_rF_2\|_{C^{0,\alpha}(\Omega_+^K)}\\[2mm]
\leq &C(\|f_2\|_{C^{1,\alpha}(\overline{\Omega_+^K})}+\|f_1\|_{C^{1,\alpha}(\overline{\Omega_+^K})} +
    \|g_1\|_{C^{2,\alpha}([0,r_0])} + \|\sigma p_{ex}\|_{C^{2,\alpha}([0,r_0])}).
\end{array}
\end{equation}

Because some of the coefficients in equation \eqref{RRproblem2} tend to infinity as $r\rightarrow0$, we will consider the elliptic problem in five dimensions to remove the symmetry-axis singularity.
Let $\mathbf{r} = (r_1, r_2, r_3, r_4)$ and $\sum\limits_{i=1}^4 r_i^2 = r^2$.
Let
 \begin{align*}
  \mathcal{D} \defs \{(z,\mathbf{r}): z\in (K, L), \mathbf{r}\in \mathbb{R}^4, |\mathbf{r}|< r_0 \}.
\end{align*}
Let
$$
\mathbf{Z}\defs (z, \mathbf{r})\in \mathcal{D}.
$$
Then obviously ${F}^{\sharp}(z,r)\in C^{0,\alpha}(\overline{\mathcal{D}})$ and satisfies
\begin{align}\label{Fsharp=divF}
\|F^{\sharp}\|_{C^{0,\alpha}(\overline{\mathcal{D}})}\leq C(\|f_2\|_{C^{1,\alpha}(\overline{\Omega_+^K})}+\|f_1\|_{C^{1,\alpha}(\overline{\Omega_+^K})} +
    \|g_1\|_{C^{2,\alpha}([0,r_0])} + \|\sigma p_{ex}\|_{C^{2,\alpha}([0,r_0])}).
\end{align}
Let $\psi(z, r_1, r_2, r_3, r_4) = \Psi (z, \sqrt{r_1^2 + r_2^2 +r_3^2 +r_4^2})$,
then {\bf {$\llbracket \textit{Problem II}\rrbracket $}} can be rewritten as the following {\bf {$\llbracket \textit{Problem II(5D)}\rrbracket $}}.
Equation \eqref{RRproblem2} becomes
\begin{align}\label{5DP2}
\partial_z \Big(\bar{\rho}_+ \bar{q}_+^2 \bar{p}_+^{-\frac{2}{\gamma}}\partial_z \psi  \Big) + \sum_{i=1}^4 \partial_{r_i} \Big( \frac{\bar{\rho}_+ \bar{q}_+^2}{1-\overline{M}_+^2} \bar{p}_+^{-\frac{2}{\gamma}}\partial_{r_i} \psi\Big)+  F_0(r) \psi ={F}^{\sharp}(z, \mathbf{r})\qquad \text{in}\quad \mathcal{D},
\end{align}
where
\begin{align}\label{F0value}
  F_0(r) \defs \frac{2}{r}\partial_r\Big( \frac{\bar{\rho}_+\bar{q}_+^2}{1-\overline{M}_+^2} \bar{p}_+^{-\frac{2}{\gamma}}\Big) - \frac{\bar{\rho}_+\bar{w}_+}{r}\Big(\frac{ \partial_r\bar{s}_+}{\gamma c_v} \bar{w}_+ - 2\big(\partial_r \bar{w}_+ + \frac{\bar{w}_+}{r}\big)  \Big)\bar{p}_+^{-\frac{2}{\gamma}}.
\end{align}
Boundary condition \eqref{PsiBry0} becomes
\begin{align}\label{b5DP2}
  \psi =0 \qquad \text{on}\quad \partial\mathcal{D}.
\end{align}

In order to apply Lemma \ref{lem5.4x} to know that $F_0\in C^{2,\alpha}(\overline{\mathcal{D}})$, we need to show $\partial_r F_0(0)=0$. Let
\begin{align}\label{F0value}
  F_0(r) \defs  \overline{F}_1(r) +\overline{F}_2(r),
\end{align}
 where $ \overline{F}_1(r) \defs \frac{2}{r}\partial_r\Big( \frac{\bar{\rho}_+\bar{q}_+^2}{1-\overline{M}_+^2} \bar{p}_+^{-\frac{2}{\gamma}}\Big)$ and $\overline{F}_2(r)\defs -\frac{\bar{\rho}_+\bar{w}_+}{r}\Big(\frac{ \partial_r\bar{s}_+}{\gamma c_v} \bar{w}_+ - 2\big(\partial_r \bar{w}_+ + \frac{\bar{w}_+}{r}\big)  \Big)\bar{p}_+^{-\frac{2}{\gamma}}$.
Direct calculations yield that
\begin{align}\label{barF1r=}
  &\partial_r \overline{F}_1(r)\notag\\
=& 2 \Big(-\frac{2}{\gamma} \frac{\bar{\rho}_+\bar{q}_+^2\bar{p}_+^{-\frac{2}{\gamma}-1}}{1-\overline{M}_+^2} \partial_r (\frac{\partial_r \bar{p}_+}{r}) + \frac{\bar{q}_+^2 \bar{p}_+^{-\frac{2}{\gamma}}}{1-\overline{M}_+^2} \partial_r(\frac{\partial_r \bar{\rho}_+}{r}) + \frac{2\bar{\rho}_+\bar{q}_+ \bar{p}_+^{-\frac{2}{\gamma}}}{1-\overline{M}_+^2} \partial_r(\frac{\partial_r\bar{q}_+}{r})\notag\\
 &\qquad + \frac{\bar{\rho}_+ \bar{q}_+^2 \bar{p}_+^{-\frac{2}{\gamma}}}{(1-\overline{M}_+^2)^2}  \partial_r(\frac{\partial_r \overline{M}_+^2}{r})\Big)\notag\\
 & + 2\Big(-\frac{2}{\gamma} (\frac{\partial_r \bar{p}_+}{r}) \partial_r\big(\frac{\bar{\rho}_+\bar{q}_+^2\bar{p}_+^{-\frac{2}{\gamma}-1}}{1-\overline{M}_+^2} \big)  + (\frac{\partial_r \bar{\rho}_+}{r}) \partial_r \big(\frac{\bar{q}_+^2 \bar{p}_+^{-\frac{2}{\gamma}}}{1-\overline{M}_+^2} \big) + (\frac{\partial_r\bar{q}_+}{r})\partial_r\big( \frac{2\bar{\rho}_+\bar{q}_+ \bar{p}_+^{-\frac{2}{\gamma}}}{1-\overline{M}_+^2}\big) \notag\\
 &\qquad + (\frac{\partial_r \overline{M}_+^2}{r})\partial_r\big(\frac{\bar{\rho}_+ \bar{q}_+^2 \bar{p}_+^{-\frac{2}{\gamma}}}{(1-\overline{M}_+^2)^2} \big)  \Big)
\end{align}
and
\begin{align}\label{barF2r=}
  \partial_r \overline{F}_2(r)=& -\big(\frac{\bar{w}_+}{r}\partial_r\bar{\rho}_+ + \bar{\rho}_+ \partial_r(\frac{\bar{w}_+}{r})  \big)\Big(\frac{ \partial_r\bar{s}_+}{\gamma c_v} \bar{w}_+ - 2\big(\partial_r \bar{w}_+ + \frac{\bar{w}_+}{r}\big)  \Big)\bar{p}_+^{-\frac{2}{\gamma}}\notag\\
  &-\frac{\bar{\rho}_+\bar{w}_+}{r}\Big(\frac{ \partial_r^2 \bar{s}_+}{\gamma c_v} \bar{w}_+ + \frac{ \partial_r \bar{s}_+}{\gamma c_v} \partial_r\bar{w}_+ - 2\big(\partial_r^2 \bar{w}_+ + \partial_r(\frac{\bar{w}_+}{r})\big)  \Big)\bar{p}_+^{-\frac{2}{\gamma}}\notag\\
  & + \frac{2}{\gamma} \frac{\bar{\rho}_+\bar{w}_+}{r}\Big(\frac{ \partial_r\bar{s}_+}{\gamma c_v} \bar{w}_+ - 2\big(\partial_r \bar{w}_+ + \frac{\bar{w}_+}{r}\big)  \Big)\bar{p}_+^{-\frac{2}{\gamma} -1}\partial_r \bar{p}_+.
\end{align}
By applying $\bar{w}_+(0)=\bar{w}_-(0)= 0$ and Lemma \ref{barvalue}, we have
\begin{align}
 \partial_r \overline{F}_2(0)=0.
\end{align}
In \eqref{barF1r=},
\begin{align}
  \lim\limits_{r\rightarrow 0} \partial_r(\frac{\partial_r \bar{p}_+}{r}) = \lim\limits_{r\rightarrow 0} \frac{r \partial_r^2 \bar{p}_+ - \partial_r \bar{p}_+}{r^2} = \lim\limits_{r\rightarrow 0}\frac{r \partial_r^3 \bar{p}_+}{2r} = \frac12 \partial_r^3 \bar{p}_+(0).
\end{align}
Similarly,
\begin{align}
   \lim\limits_{r\rightarrow 0} \partial_r(\frac{\partial_r \bar{\rho}_+}{r}) =\frac12 \partial_r^3 \bar{\rho}_+(0), \,   \lim\limits_{r\rightarrow 0} \partial_r(\frac{\partial_r \bar{q}_+}{r}) =\frac12 \partial_r^3 \bar{q}_+(0), \,  \lim\limits_{r\rightarrow 0} \partial_r(\frac{\partial_r \overline{M}_+^2}{r}) =\frac12 \partial_r^3 \overline{M}_+^2(0).
\end{align}
Then it follows from Lemma \ref{barvalue} that
\begin{align}
 \partial_r \overline{F}_1(0)=0.
\end{align}
Thus
$ \partial_r F_0(0) =\partial_r \overline{F}_1(0) + \partial_r \overline{F}_2(0)=0$. Therefore,
\begin{equation}
F_0,\, \mbox{which only depends on the special shock solution, is in } C^{2,\alpha}(\overline{\mathcal{D}}).
\end{equation}

Now the well-posedness of problem \eqref{5DP2}-\eqref{b5DP2} and then the well-posedness of problem \eqref{PsiBry0} and \eqref{RRproblem2} will be established by proving the following lemma.
\begin{lem}\label{theta2p2}
There exists a positive constant $r_*$, only depending on $\overline{U}_-(r)$ and $\gamma$, such that for $0<r_0 < r_*$, there exists a unique weak solution $\psi\in H_0^1 (\overline{\mathcal{D}})$ to problem \eqref{5DP2}-\eqref{b5DP2}. Moreover, if \eqref{bc:F2} holds, the $H_0^1(\overline{\mathcal{D}})$ weak solution $\psi$ of the {\bf {$\llbracket \textit{Problem II(5D)}\rrbracket $}} is in $C^{2,\alpha}(\overline{\mathcal{D}})$ and satisfies 
\begin{align}\label{psic3}
\|\psi\|_{C^{2,\alpha}(\overline{\mathcal{D}})} \leq {C} \Big(\|f_2\|_{C^{1,\alpha}(\overline{\Omega_+^K})}  + \| \delta \theta_1 \|_{C^{1,\alpha}(\overline{\mathcal{D}})} \Big),
\end{align}
where the constant $C$ only depends on $\overline{U}_\pm$, $\gamma$, $r_0$, $L$ and $\alpha$. Moreover,
\begin{equation}\label{5.103x}
\psi(z,r_1,r_2,r_3,r_4)\equiv\Psi(z, \sqrt{r_1^2+r_2^2 + r_3^2+r_4^2}),
\end{equation}
where $\Psi\in C^{2,\alpha}(\Omega_+^K)$ is a solution of {\bf {$\llbracket \textit{Problem II}\rrbracket $}}, and it holds that
\begin{align}\label{p2theta2}
&\|\delta p_2\|_{C^{1,\alpha}(\overline{\Omega_+^K})}+  \|\delta \theta_2\|_{C^{1,\alpha}(\overline{\Omega_+^K})}\notag\\
\leq& C \Big(\|f_1\|_{C^{1,\alpha}(\overline{\Omega_+^K})} +\|f_2\|_{C^{1,\alpha}(\overline{\Omega_+^K})} +
\|g_1\|_{C^{2,\alpha}([0,r_0])} + \|\sigma p_{ex}\|_{C^{2,\alpha}([0,r_0])}\Big),
\end{align}
where the constant $C$ only depends on $\overline{U}_\pm$, $K$, $L$, $\gamma$, $r_0$ and $\alpha$.
\end{lem}

\begin{proof}
The proof is divided into four steps.

\emph{Step 1.} In this step, we will prove the uniqueness of the solution $\psi$. Assume that $\psi_1$ and $\psi_2$ are two solutions to problem \eqref{5DP2}-\eqref{b5DP2}. Let
\begin{align*}
\widetilde{\psi}\defs \psi_1 - \psi_2.
\end{align*}
Then $\widetilde{\psi}$ satisfies the following equation
\begin{align}\label{u5DP2}
  \partial_z \Big(\bar{\rho}_+ \bar{q}_+^2 \bar{p}_+^{-\frac{2}{\gamma}}\partial_z \widetilde{\psi}\Big) + \sum_{i=1}^4 \partial_{r_i} \Big( \frac{\bar{\rho}_+ \bar{q}_+^2}{1-\bar{M}_+^2}\bar{p}_+^{-\frac{2}{\gamma}}\partial_{r_i} \widetilde{\psi}\Big) + F_0(r) \widetilde{\psi}= 0\qquad \text{in}\quad \mathcal{D},
\end{align}
with the boundary condition
\begin{align}\label{ub5DP2}
  \widetilde{\psi} =0 \qquad \text{on}\quad \partial\mathcal{D}.
\end{align}
Multiplying $\widetilde{\psi}$ on the both sides of equation \eqref{u5DP2}, integrating over $\mathcal{D}$ and then integration by part, we have
\begin{align}\label{psi=0}
  &\int_{\mathcal{D}}{F}_0(r) \widetilde{\psi}^2 \dif \mathbf{Z}\notag\\
  =&\int_{\mathcal{D}}\bar{\rho}_+ \bar{q}_+^2\bar{p}_+^{-\frac{2}{\gamma}} |\partial_z\widetilde{\psi}|^2 \dif \mathbf{Z}+ \sum_{i=1}^4 \int_{\mathcal{D}} \frac{\bar{\rho}_+ \bar{q}_+^2}{1-\bar{M}_+^2}\bar{p}_+^{-\frac{2}{\gamma}}|\partial_{r_i} \widetilde{\psi}|^2 \dif \mathbf{Z}\notag\\
  \geq& \min\limits_{|\mathbf{r}|<r_0}\{\bar{\rho}_+ \bar{q}_+^2 \bar{p}_+^{-\frac{2}{\gamma}}\} \int_{\mathcal{D}}|D\widetilde{\psi}|^2 \dif \mathbf{Z}\notag\\
   \geq& \min\limits_{|\mathbf{r}|<r_0}\{\bar{\rho}_+ \bar{q}_+^2 \bar{p}_+^{-\frac{2}{\gamma}}\} \frac{1}{4r_0^2}\int_{\mathcal{D}}\widetilde{\psi}^2 \dif \mathbf{Z},
\end{align}
where the last inequality is by Poincar\'{e}'s inequality (see \cite[Proposition 3.10]{GM}).
So if we can show
\begin{align}\label{Ir*}
\max\limits_{|\mathbf{r}|<r_0}{F}_0 (r)  \leq \frac{1}{4r_0^2}\min\limits_{|\mathbf{r}|<r_0}\{  \bar{\rho}_+ \bar{q}_+^2 \bar{p}_+^{-\frac{2}{\gamma}} \},
\end{align}
\eqref{psi=0} yields that $\widetilde{\psi} \equiv0$. Then the uniqueness of the solutions $\psi$ follows.

Hence, we remains to show \eqref{Ir*} in this step. It follows from \eqref{RHS-}-\eqref{BCRH4} and \eqref{F0value} that
\begin{align*}
  {F}_0(r) =&\frac{2\bar{w}_-^2}{r^2} \frac{\bar{\rho}_+ \bar{p}_+^{-\frac{2}{\gamma}}}{\gamma -1 +2t}\notag\\
  &\times \Big\{\frac{\gamma -1 +2t}{1-t} \cdot \frac{(\gamma-1)(\gamma+t)+ (\gamma-2 )(\gamma+1)}{\gamma+1} \notag\\
 &\qquad +(\gamma-1)(\gamma+t) +(\gamma +1)+ (\gamma +1) \frac{r\partial_r \bar{w}_-}{\bar{w}_-}\notag\\
 &\qquad + \frac{\gamma+1}{2}\frac{\bar{w}_-^2}{\bar{q}_-^2}\Big((\gamma-1) - \frac{(\gamma-1)(\gamma+1)^2(1-t)^2}{(\gamma-1+2t)^2(2\gamma -(\gamma-1)t)}\Big)\Big\}\notag\\
<&  2\frac{\gamma +1}{\gamma -1}\frac{\bar{w}_-^2\bar{\rho}_+ \bar{p}_+^{-\frac{2}{\gamma}}}{r^2}\Big(\frac{2(\gamma -1)}{1-t} + \gamma + \frac{r\partial_r \bar{w}_-}{\bar{w}_-} + \frac{\gamma -1}{2} \frac{\bar{w}_-^2}{\bar{q}_-^2} \Big),
\end{align*}
where $ t(r) = \frac{1}{\overline{M}_-^2(r)}$.
Notice that
\begin{align*}
\bar{\rho}_+ \bar{q}_+^2 \bar{p}_+^{-\frac{2}{\gamma}} =   \frac{\gamma -1 +2t}{\gamma +1}  \Big(\frac{2\gamma \overline{M}_-^2 - \gamma +1}{\gamma+1}\bar{p}_-   \Big)^{-\frac{2}{\gamma}}\bar{\rho}_- \bar{q}_-^2.
\end{align*}
Obviously, there exists a constant $r_*>0$ such that
\begin{equation}
r_* \leq\frac{\sqrt{2}\mu^2 \min\limits_{|\mathbf{r}|<r_0}\left\{\sqrt{\bar{\rho}_-} \bar{q}_- \Big(\frac{2\gamma \overline{M}_-^2 - \gamma +1}{\gamma+1}\bar{p}_-   \Big)^{-\frac{1}{\gamma}}\right\} }{4\max\limits_{|\mathbf{r}|<r_0}\left\{ \frac{\bar{w}_- \sqrt{\bar{\rho}_-} \Big(\frac{2\gamma \overline{M}_-^2 - \gamma +1}{\gamma+1}\bar{p}_-   \Big)^{-\frac{1}{\gamma}}}{r} \cdot \sqrt{\frac{2(\gamma -1)}{1-t} + \gamma + \frac{r\partial_r \bar{w}_-}{\bar{w}_-} + \frac{\gamma -1}{2} \frac{\bar{w}_-^2}{\bar{q}_-^2}}  \right\} }>0.
\end{equation}
Then for any $0<r_0\leq r_*$, by a straightforward calculation, we know that \eqref{Ir*} holds.

\emph{Step 2.} In this step, we will prove the existence of the weak solution $\psi\in H_0^1(\mathcal{D})$ to {\bf {$\llbracket \textit{Problem II(5D)}\rrbracket $}}.
Let $\zeta\in H_0^1(\mathcal{D})$ be a test function. The bilinear form corresponding to problem {\bf {$\llbracket \textit{Problem II(5D)}\rrbracket $}} is as follows.
\begin{align*}
  \mathcal{B}[\psi, \zeta ]\defs \int_{\mathcal{D}}\bar{\rho}_+ \bar{q}_+^2\bar{p}_+^{-\frac{2}{\gamma}} \partial_z\psi \partial_z\zeta \dif \mathbf{Z}+ \sum_{i=1}^4 \int_{\mathcal{D}} \frac{\bar{\rho}_+ \bar{q}_+^2}{1-\bar{M}_+^2} \bar{p}_+^{-\frac{2}{\gamma}}\partial_{r_i} \psi \partial_{r_i}\zeta \dif \mathbf{Z}
   - \int_{\mathcal{D}}F_0(r) \psi \zeta  \dif \mathbf{Z}.
\end{align*}
It is easy to check that $ \mathcal{B}[\psi, \zeta]$ is bounded. By applying \eqref{psi=0} and \eqref{Ir*}, we obtain for any $0<r_0<r_*$,
\begin{align}\label{coercive5D}
   \mathcal{B}[\psi, \psi ]\geq \beta \|\psi\|_{H_0^1(\mathcal{D})}^2,
\end{align}
where the constant $\beta>0$ only depends on $\overline{U}_\pm $, $\gamma $, $r_0$, $L$ and $K$. By the Lax-Milgram theorem, 
there exists a unique solution $\psi\in H_0^1(\mathcal{D})$ to problem {\bf {$\llbracket \textit{Problem II(5D)}\rrbracket $}}.

Next, we will establish an $H^1$-estimate of the solution $\psi$. Note that $\psi\in H_0^1(\mathcal{D})$ is a weak solution of problem {\bf {$\llbracket \textit{Problem II(5D)}\rrbracket $}} if $\mathcal{B}[\psi, \zeta]= \int_{\mathcal{D}} {F}^{\sharp} \psi \dif \mathbf{Z}$ for any test function $\zeta\in H_0^1(\mathcal{D})$. Choosing $\zeta=\psi$, we have
\begin{align}\label{I5DP2}
 \int_{\mathcal{D}}\bar{\rho}_+ \bar{q}_+^2 \bar{p}_+^{-\frac{2}{\gamma}}|\partial_z\psi|^2 \dif \mathbf{Z}+ \sum_{i=1}^4 \int_{\mathcal{D}} \frac{\bar{\rho}_+ \bar{q}_+^2}{1-\bar{M}_+^2} \bar{p}_+^{-\frac{2}{\gamma}}|\partial_{r_i} \psi|^2 \dif \mathbf{Z}- \int_{\mathcal{D}} {F}_0 \psi^2 \dif \mathbf{Z}
  = \int_{\mathcal{D}}{F}^{\sharp} \psi \dif \mathbf{Z}.
\end{align}
By applying Cauchy's inequality with small constant $\epsilon>0$, \eqref{I5DP2} implies that
\begin{align*}
\frac{1}{4r_0^2}\min\limits_{|\mathbf{r}|<r_0}\{\bar{\rho}_+ \bar{q}_+^2\bar{p}_+^{-\frac{2}{\gamma}}\} \| \nabla\psi\|_{L^2(\mathcal{D})}^2 - \max\limits_{|\mathbf{r}|<r_0}{F}_0 (r) \| \psi\|_{L^2(\mathcal{D})}^2 
\leq\epsilon\| \nabla \psi\|_{L^2(\mathcal{D})}^2 + \frac{1}{4\epsilon} \|{F}^{\sharp} \|_{L^2(\mathcal{D})}^2.
\end{align*}
Then it follows from \eqref{Ir*} and the Poincar\'{e}'s inequality that 
\begin{equation}\label{PART1}
	\| \psi \|_{H_0^1(\mathcal{D})}\leq {C} \| {F}^{\sharp} \|_{L^2(\mathcal{D})},
	\end{equation}
where the constant $C$ only depends on $\overline{U}_\pm$, $\gamma$, $r_0$ and $L$.

By \eqref{5.123xw} we know that
\[
\| {F}^{\sharp} \|_{L^2(\mathcal{D})}
\leq C(\|f_2\|_{C^{1,\alpha}(\overline{\Omega_+^K})}+\|f_1\|_{C^{1,\alpha}(\overline{\Omega_+^K})} +
    \|g_1\|_{C^{2,\alpha}([0,r_0])} + \|\sigma p_{ex}\|_{C^{2,\alpha}([0,r_0])}).
\]
Therefore, by \eqref{PART1}, we have
\begin{align}\label{H15D}
 \|\psi\|_{H_0^1({\mathcal{D}})} \leq C(\|f_2\|_{C^{1,\alpha}(\overline{\Omega_+^K})}+\|f_1\|_{C^{1,\alpha}(\overline{\Omega_+^K})} +
    \|g_1\|_{C^{2,\alpha}([0,r_0])} + \|\sigma p_{ex}\|_{C^{2,\alpha}([0,r_0])}),
 \end{align}
where the constant $C$ only depends on $\overline{U}_\pm$, $\gamma$, $r_0$ , $L$ and $\alpha$.

\emph{Step 3.} In this step, we will prove estimate \eqref{psic3}. Let set  $\{(z,r)\}\defs\{(z,\mathbf{r})\in\overline{\mathcal{D}}:\,|\mathbf{r}|=r\}$.
It follows from the local boundary estimate \cite[Corollary 8.36]{DNS} and interior estimate \cite[Theorem 8.32]{DNS} with $g=0$, there exists $\alpha\in(0,1)$ such that for any $\Omega'\Subset\overline{\mathcal{D}}\backslash\{(K,r_0),\, (L,r_0)\}$,
\begin{align}\label{psiC1alphalocal}
 \|\psi\|_{C^{1,\alpha} (\Omega')} \leq {C} \|\psi\|_{C^0({\mathcal{D}})}+ \| {F}^{\sharp} \|_{L^{q}(\mathcal{D})}),
\end{align}
where $\alpha=1-n/q$ and constant $C$ depends on $\Omega'$ but does not depend on $\psi$.

It remains to consider the $C^{1,\alpha}$-estimate of solution $\psi$ near the points in the set $\{(K,r_0),\,(L,r_0)\}$.
Without loss of the generality, we only consider the regularity of $\psi$ near the set $\{(K, r_0)\}$.

Since $\psi(z,\mathbf{r})\big|_{|\mathbf{r}|=r_0}=0$, $f_2(z,r_0)=0$ and $\delta\theta_1(z,r_0)=0$, we define
\begin{align*}
\psi^\star(z,\mathbf{r}) \defs \begin{cases}
\psi(z,\mathbf{r})\quad& 0<|\mathbf{r}|\leq r_0\\
 -\psi (z, \mathbf{r}^\star)\quad& |\mathbf{r}|> r_0
\end{cases}
\qquad
{F}_2^\star(z,r) \defs
\begin{cases}
{F}_2(z,r)\quad& 0<r<r_0\\
 -{F}_2(z, 2r_0-r)\quad& r\geq r_0
\end{cases}
\end{align*}
where $\mathbf{r}^\star$ is the reflection point of $\mathbf{r}$ with respect to the nozzle wall $|\mathbf{r}|=r_0$, such that the segment $\mathbf{r}\mathbf{r}^\star$ is perpendicular to the nozzle wall and the distances from points $\mathbf{r}$ and $\mathbf{r}^\star$ to the nozzle wall are equal.

Let $\mathcal{D}_\tau$ be a small domain in $\mathcal{D}$ with $\{(K, r_0)\}\subset\overline{\mathcal{D}_\tau}$. Define
\begin{align*}
  \widetilde{\mathcal{D}_\tau}\defs  \mathcal{D}_\tau  \cup \big\{(z,\mathbf{r}^\star): (z,\mathbf{r})\in  \mathcal{D}_\tau    \big\}   \cup (\big\{(z,\mathbf{r}):\,|\mathbf{r}|=r_0\big\}\cap  \overline{\mathcal{D}_\tau}).
\end{align*}
It follows from Lemma \ref{barvalue} that we can extend $\bar{p}_+$, $\bar{s}_+$ and $\bar{M}^2_+$ evenly across the boundary $\{|\mathbf{r}|=r_0\}$, and then the same do for $\bar{\rho}_+$ and $\bar{q}_+$ from the thermal relations.
Then we easily know that ${\psi}^\star$ satisfies the following equation
\begin{align}\label{5DP2r}
 \partial_z \Big(\bar{\rho}_+ \bar{q}_+^2 \bar{p}_+^{-\frac{2}{\gamma}}\partial_z \psi^\star  \Big) + \sum_{i=1}^4 \partial_{r_i} \Big( \frac{\bar{\rho}_+ \bar{q}_+^2}{1-\overline{M}_+^2} \bar{p}_+^{-\frac{2}{\gamma}}\partial_{r_i} \psi^\star\Big) + {F}_0 (r) \psi^\star
  ={F}^{\sharp}\qquad \text{in}\quad  \widetilde{\mathcal{D}_\tau}
\end{align}
where
${F}^{\sharp}$ is defined by \eqref{5.98x} by replacing $F_2$ by $F_2^\star$.
Meanwhile, the boundary condition on $\Gamma\defs\{z=K\}\cap  \widetilde{\mathcal{D}_\tau}$ is
\begin{align}
  \psi^\star = 0 \qquad \text{on}\quad  \Gamma.
\end{align}
Then it follows from \cite[Corollary 8.36]{DNS} again that for any smooth domain $\widetilde{\mathcal{D}_\tau}^{\sharp} \Subset\Gamma\cup\widetilde{\mathcal{D}_\tau}$, we have
\begin{align}\label{psistar}
\|\psi^\star\|_{C^{1,\alpha}(\widetilde{\mathcal{D}_\tau}^{\sharp})} \leq C \Big(  \|{F}^{\sharp} \|_{C^{0}(\overline{\mathcal{D}})} + \|\psi^{\star}\|_{C^{0}(\widetilde{\mathcal{D}_\tau})} \Big).
\end{align}

Therefore, by \eqref{psiC1alphalocal} and \eqref{psistar}, and by the definition of $\psi^{\star}$ and $F_2^{\star}$, we have
\begin{align}\label{psiC1alphal}
   \| \psi \|_{C^{1,\alpha}(\overline{\mathcal{D}})}\leq {C} (\|{F}_2 \|_{C^{0,\alpha}(\overline{\mathcal{D}})}+\| \psi \|_{C^{0}(\overline{\mathcal{D}})}).
   \end{align}

Next, let us consider the $C^{2,\alpha}$-estimate of $\psi$. Take derivative on \eqref{5DP2} with respect to $z$ and denote
\begin{align}\label{5.120x}
\psi^z(z,\mathbf{r})\defs \partial_z \psi(z,\mathbf{r}),\,\,
 R_1(\mathbf{r})\defs \bar{\rho}_+ \bar{q}_+^2 \bar{p}_+^{-\frac{2}{\gamma}}, \,\,  R_2(\mathbf{r})\defs  \frac{\bar{\rho}_+ \bar{q}_+^2}{1-\overline{M}_+^2} \bar{p}_+^{-\frac{2}{\gamma}}.
 \end{align}
Because $\partial_zR_i=0$ and by \eqref{Fsharp=divF}, $\psi^z$ satisfies the following elliptic equation
\begin{align}\label{psizproblem}
 \partial_z\Big(R_1(\mathbf{r})\partial_{z} \psi^z\Big)+ \sum_{i=1}^4 \partial_{r_i}\Big(R_2(\mathbf{r}) \partial_{r_i} \psi^z\Big)
 = \partial_z \big( {F}^{\sharp}(z,r)\big) -  F_0(\mathbf{r}) \psi^z\qquad \text{in}\quad \mathcal{D},
\end{align}
where $r=|\mathbf{r}|$.
Next, by direct calculation, it follows from \eqref{5DP2} and \eqref{b5DP2} that $\psi^z$ satisfies the following Neumann-Dirichlet mixed boundary conditions
\begin{align}\label{brypsiz=}
  \psi^z(z,\mathbf{r})\big|_{|\mathbf{r}|=r_0}=0,\quad \partial_z \psi^z(K,\mathbf{r}) = \frac{{F}^{\sharp}(K, \mathbf{r})}{ R_1(\mathbf{r})},\quad \partial_z \psi^z(L,\mathbf{r}) = \frac{{F}^{\sharp}(L,\mathbf{r})}{ R_1(\mathbf{r})}.
\end{align}
Notice that by \eqref{bc:F2}, \eqref{Problem1bry1}, \eqref{5.95x} and \eqref{Fsharp=divF}, we have
\begin{align}\label{divFzr0=0}
  {F}^{\sharp}(z,r_0) = \frac{\bar{p}_+^{-\frac{1}{\gamma}}(r_0) {F}_2(z,r_0)}{r_0}
 =0 \qquad \text{for}\quad z\in[K,L].
\end{align}
Therefore, boundary conditions \eqref{brypsiz=} are compatible.
So near the set $\{(K,r_0),\,(L,r_0)\}$, we do an odd extension across boundary $|\mathbf{r}|=r_0$ for $\psi^z$, even extensions for $R_i$, $\partial_z \big({F}^{\sharp}\big)$ and $F_0$, and apply \cite[Theorem 5.54]{GM13} near the boundary $z=K$ or $z=L$ for the conormal boundary conditions and \cite[Theorem 5.21]{GM} interiorly, to obtain
\begin{align}\label{psizC1alpha}
   \| \partial_z\psi \|_{C^{1,\alpha}(\overline{\mathcal{D}})}\leq {C} (\|{F}_2 \|_{C^{1,\alpha}(\overline{\mathcal{D}})}+\|\psi \|_{C^{0}(\overline{\mathcal{D}})}).
   \end{align}

Then, we will prove that $\psi^{r_j}\defs \partial_{r_i}\psi\in C^{1,\alpha}(\mathcal{D}')$ for $j=1,2,3,4$ and for any subdomain $\mathcal{D}'\Subset(\overline{\mathcal{D}}\backslash\{|\mathbf{r}|=r_0\})$. Taking derivatives on both sides of the equation \eqref{5DP2} with respect to the variable $r_j$, we have $\psi^{r_j}$ satisfies equation
   \begin{align}
&\partial_z\Big(R_1(\mathbf{r})\partial_{z}\psi^{r_j}\Big) + \sum_{i=1}^4\partial_{r_i}\Big(R_2(\mathbf{r})\partial_{r_i}\psi^{r_j}\Big)\notag\\
=& \partial_{r_j} \big({F}^{\sharp}(z,r)\big) -\partial_z\Big(\partial_{r_j} R_1(\mathbf{r})\partial_{z}\psi\Big)  -\partial_{r_i}\Big(\partial_{r_j} R_2(\mathbf{r})\partial_{r_i}\psi\Big) - \partial_{r_j}\Big(F_0(\mathbf{r})\psi\Big) \qquad \text{in}\quad \mathcal{D}.
   \end{align}
For the boundary condition, it follows from \eqref{b5DP2} that
\begin{equation}
   \psi^{r_j}(K,\mathbf{r}) =  \psi^{r_j}(L,\mathbf{r}) =0.
\end{equation}
Hence, it follows from \cite[Corollary 8.36]{DNS} and \cite[Theorem 8.32]{DNS} again that for any $j=1,2,3,4$ and for any subdomain $\mathcal{D}'\Subset(\overline{\mathcal{D}}\backslash\{|\mathbf{r}|=r_0\})$
\begin{equation}\label{5.127x}
   \|\partial_{r_j} \psi\|_{C^{1,\alpha}(\mathcal{D}')}\leq {C} (\|{F}_2 \|_{C^{1,\alpha}(\overline{\mathcal{D}})}+\|\psi \|_{C^{0}(\overline{\mathcal{D}})}).
  \end{equation}

Finally, let us consider the estimate near the nozzle wall $|\mathbf{r}|=r_0$. Let $\mathcal{D}_r=\mathcal{D}\cap\{(z,\mathbf{r}),\,|\mathbf{r}|\geq r\}$. For any subdomain $\mathcal{D}'\Subset\overline{\mathcal{D}}\backslash\overline{\mathcal{D}_{r_0/2}}$, consider tangential derivative to the boundary of $\mathcal{D}'$, $\partial_{\tau}=r_i\partial_{r_j}-r_j\partial_{r_i}$, where $i,j=1,\cdots,4$ with $i\neq j$.
It follows from boundary condition \eqref{b5DP2} that
\begin{equation}\label{5.128x}
\partial_{\tau}\psi=0\qquad\mbox{on }\partial \mathcal{D}'\cap\overline{\mathcal{D}_{r_0/2}}.
\end{equation}
Then similarly, by following the argument to derive the equation for $\partial_{\tau}\psi$ in divergence form and the even/odd extension as before, based on the Dirichlet boundary condition \eqref{5.128x}, and by \cite[Corollary 8.36]{DNS}, we can have for $\mathcal{D}'\Subset\overline{\mathcal{D}}\backslash\overline{\mathcal{D}_{r_0/2}}$ and for $i,j=1,\cdots,4$ with $i\neq j$,
\begin{equation}\label{5.129x}
   \|(r_i\partial_{r_j}-r_j\partial_{r_i}) \psi\|_{C^{1,\alpha}(\mathcal{D}')}\leq {C} (\|{F}_2 \|_{C^{1,\alpha}(\overline{\mathcal{D}})}+\|\psi \|_{C^{0}(\overline{\mathcal{D}})}).
  \end{equation}
Let $\partial_{\tau_{ij}} = (r_i \partial_{r_j} - r_j \partial_{r_i}) $ and $\partial_{n}=\sum_{i=1}^4r_i \partial_{r_i}$. We have
   \begin{equation}
     \psi_{r_i} =\frac{r_i\psi_{n} - \sum_{j\neq i}r_j \psi_{\tau_{ij}}}{|\mathbf{r}|^2},\label{psiri=}
\end{equation}
and then
\begin{equation}
\sum_{i=1}^4  \psi_{r_ir_i}= \frac{\psi_{nn}}{|\mathbf{r}|^2}+\sum_{i=1}^4(\frac{r_i}{|\mathbf{r}|^2})_{r_i}\psi_n-\sum_{i=1}^4(\frac{\sum_{j\neq i}\psi_{\tau_{ij}}}{|\mathbf{r}|^2})_{r_i}. \label{psiriri=}
\end{equation}
By \eqref{5.120x}, we know that $R_2\geq C>0$. It follows from equation \eqref{5DP2r} that
\begin{align*}
\psi_{nn}=&\frac{|\mathbf{r}|^2}{R_2}\left( {F}^{\sharp}- R_1\psi_zz -{F}_0 (r) \psi \right)-|\mathbf{r}|^2(\sum_{i=1}^4(\frac{r_i}{|\mathbf{r}|^2})_{r_i}\psi_n-\sum_{i=1}^4(\frac{\sum_{j\neq i}\psi_{\tau_{ij}}}{|\mathbf{r}|^2})_{r_i}).
\end{align*}
Because $|\mathbf{r}|\geq r_0/2>0$ in $\mathcal{D}'$, it follows from estimates \eqref{psiC1alphal}, \eqref{psizC1alpha} and \eqref{5.129x} that
   \begin{align}\label{psinnin}
    \|\psi_{nn}\|_{C^{0,\alpha}(\mathcal{D}')}\leq {C} (\|{F}_2 \|_{C^{1,\alpha}(\overline{\mathcal{D}})}+\|\psi \|_{C^{0}(\overline{\mathcal{D}})}).
   \end{align}
So for any $\mathcal{D}'\Subset\overline{\mathcal{D}}\backslash\overline{\mathcal{D}_{r_0/2}}$,
\begin{equation}\label{5.133x}
  \|\psi\|_{C^{2,\alpha}(\mathcal{D}')}\leq {C} (\|{F}_2 \|_{C^{1,\alpha}(\overline{\mathcal{D}})}+\|\psi \|_{C^{0}(\overline{\mathcal{D}})}).
 \end{equation}
It further with estimate \eqref{5.127x}, we have
\begin{equation}\label{5.134x}
  \|\psi\|_{C^{2,\alpha}(\overline{\mathcal{D}})}\leq {C} (\|{F}_2 \|_{C^{1,\alpha}(\overline{\mathcal{D}})}+\|\psi \|_{C^{0}(\overline{\mathcal{D}})}).
 \end{equation}

 Finally, we will show the following estimate from \eqref{5.134x}
\begin{align}\label{psiC2alphal}
   \| \psi \|_{C^{0}(\overline{\mathcal{D}})}\leq {C} \|{F}_2 \|_{C^{1,\alpha}(\overline{\mathcal{D}})}.
   \end{align}
In fact, if it is not true, then for any $k>0$, there exists $\psi_k$ and $F^k_2$ with $\|\psi_k\|_{C^{0}(\overline{\mathcal{D}})}=1$ such that $ 1=\| \psi_k \|_{C^{0}(\overline{\mathcal{D}})}>k \|{F}^k_2 \|_{C^{1,\alpha}(\overline{\mathcal{D}})}$. Then
\begin{equation}\label{5.136x}
F_2^k\rightarrow0\qquad\mbox{as }k\rightarrow\infty\quad\mbox{in }C^{1,\alpha}(\overline{\mathcal{D}}).
\end{equation}
Moreover, it follows from \eqref{5.134x} that $ \|\psi\|_{C^{2,\alpha}(\overline{\mathcal{D}})}\leq {C}(\frac1k+1)$. So there exists a subsequence which we still denote as $\psi_k$ such that
\begin{equation}\label{5.137x}
\psi_k\rightarrow\psi^*\qquad\mbox{as }k\rightarrow\infty\quad\mbox{in }C^{2}(\overline{\mathcal{D}}).
\end{equation}
So it follows from \eqref{5.98x}, \eqref{5DP2}, \eqref{5.136x} and \eqref{5.137x} that $\psi^*$ satisfies the following boundary value problem
\begin{align}
\partial_z \Big(\bar{\rho}_+ \bar{q}_+^2 \bar{p}_+^{-\frac{2}{\gamma}}\partial_z \psi^*  \Big) + \sum_{i=1}^4 \partial_{r_i} \Big( \frac{\bar{\rho}_+ \bar{q}_+^2}{1-\overline{M}_+^2} \bar{p}_+^{-\frac{2}{\gamma}}\partial_{r_i} \psi^*\Big)+  F_0(r) \psi^* =0\qquad \text{in}\quad \mathcal{D}
\end{align}
and
\begin{equation}
\psi^*=0\qquad\mbox{on }\partial\mathcal{D}.
\end{equation}
By the uniqueness obtained in step 1, we know that $\psi^*\equiv0$. It is contradicts to $1=\lim_{k\rightarrow\infty}\|\psi_k\|_{C^0(\overline{\mathcal{D}})}=\|\psi^*\|_{C^0(\overline{\mathcal{D}})}$.

Hence, \eqref{5.133x} and \eqref{psiC2alphal} yield
\begin{equation}
  \|\psi\|_{C^{2,\alpha}(\overline{\mathcal{D}})}\leq {C} \|{F}_2 \|_{C^{1,\alpha}(\overline{\mathcal{D}})},
 \end{equation}
and then \eqref{psic3} follows.

\emph{Step 4.} In this step, we will prove \eqref{5.103x} and \eqref{p2theta2}.
First, it is easy to check that the problem \eqref{5DP2}-\eqref{b5DP2} is rotationally invariant. Thus, \eqref{5.103x} follows from the uniqueness of the solutions.
Then it follows from \eqref{psic3} and Lemma \ref{lem5.4x} that $\Psi$ satisfies \eqref{PsiBry0}, \eqref{5.93x}, \eqref{RRproblem2}, and estimate
\begin{align}\label{Psic3}
\|\Psi\|_{C^{2,\alpha}(\overline{\Omega_+^K})} \leq {C} \Big(\| f_2 \|_{C^{1,\alpha}(\overline{\Omega_+^K})} + \| \delta \theta_1 \|_{C^{1,\alpha}(\overline{\Omega_+^K})} \Big),
\end{align}
where the constant $C$ only depends on $\overline{U}_\pm$, $\gamma$, $r_0$, $L$, $K$ and $\alpha$.
Applying \eqref{Psic3}, for $\widehat{\Psi}=r^2\Psi$, we have
 \begin{align}\label{tildePsi}
    \|\widehat{\Psi}\|_{C^{2,\alpha}(\overline{\Omega_+^K})} \leq {C}_+  \Big(\| f_2 \|_{C^{1,\alpha}(\overline{\Omega_+^K})} + \| \delta \theta_1 \|_{C^{1,\alpha}(\overline{\Omega_+^K})} \Big).
  \end{align}
Define $( \delta{\theta}_2,\delta{p}_2)$ by \eqref{thetaphat}. Then  \eqref{p2theta2} follows from \eqref{theta1p13D} and \eqref{tildePsi}.
\end{proof}

To conclude the proof of Theorem \ref{BigThm}, we need to further raise the regularity of $\delta p$ and $\delta \theta$, which is the following lemma.
\begin{lem}\label{raiseRegularity}
Assume the assumptions in Theorem \ref{BigThm} hold. Let $(\delta p, \delta \theta)=(\delta p_1 +\delta p_2, \delta \theta_1 + \delta \theta_2)$, where $(\delta p_1,\delta\theta_1)$ and $(\delta p_2,\delta\theta_2)$ are obtained by Lemma \ref{3Dphi} and Lemma \ref{theta2p2}, respectively.  Then $(\delta p, \delta \theta)\in {C^{2,\alpha}(\overline{\Omega_+^K})}$ and satisfies
\begin{align}\label{C2alpha}
\| \delta p \|_{C^{2,\alpha}(\overline{\Omega_+^K})} + \| \delta \theta \|_{C^{2,\alpha}(\overline{\Omega_+^K})}
\leq C\Big(  \sum_{i=1}^2 \| f_i \|_{C^{1,\alpha}(\overline{\Omega_+^K})} + \| g_1 \|_{C^{2,\alpha}([0,r_0])} + \|\sigma  p_{ex} \|_{C^{2,\alpha}([0,r_0])}\Big),
\end{align}
where the constant $C$ only depends on $\overline{U}_{\pm}$, $K$, $L$, $r_0$ and $\alpha$.
Moreover, it holds that
\begin{align}
  \partial_r\delta p(z,0) =0,\quad \partial_r\delta p(z,r_0) =0,\quad
\partial_r^2\delta \theta (z,0)=0.
\end{align}
\end{lem}

\begin{proof}
This proof is divided into two steps.

\emph{Step 1:} In this step, we will prove the $C^{2,\alpha}$ regularity of $\delta p$.
Let
\begin{align}
\mathcal{A}(r)\defs \frac{\bar{\rho}_+\bar{w}_+}{r}\Big(\frac{ \partial_r\bar{s}_+}{\gamma c_v}\bar{w}_+ - 2\big(\partial_r \bar{w}_+ + \frac{\bar{w}_+}{r}\big) \Big).
\end{align}
It follows from \eqref{F1eq1} and \eqref{F2eq2} that
\begin{align}\label{psecond}
&\partial_z^2 \Big( \frac{1-\overline{M}_+^2}{\bar{\rho}_+\bar{q}_+^2}  \bar{p}_+^{\frac{1}{\gamma}} \delta p \Big) + \partial_r \Big(  \frac{1}{\bar{\rho}_+\bar{q}_+^2}  \bar{p}_+^{\frac{2}{\gamma}} \partial_r \big( \bar{p}_+^{-\frac{1}{\gamma}} \delta p \big)\Big) -\bar{p}_+^{\frac{1}{\gamma}} \frac{\partial_z\delta \theta}{r}\notag\\
=&\bar{p}_+^{\frac{1}{\gamma}} \partial_z f_1 + \partial_r \Big( \frac{1}{\bar{\rho}_+\bar{q}_+^2}  \bar{p}_+^{\frac{1}{\gamma}} f_2\Big)
+  \partial_r \Big(\frac{1}{\bar{\rho}_+\bar{q}_+^2} \bar{p}_+^{\frac{1}{\gamma}} \mathcal{A}(r) \int_{K}^{z} \delta{\theta}(\tau,r)\dif \tau\Big).
\end{align}
By employing \eqref{F2eq2}, it holds that
\begin{align}\label{zthetaeq}
  - \partial_z \delta \theta = \frac{1}{\bar{\rho}_+\bar{q}_+^2}  \bar{p}_+^{\frac{1}{\gamma}} \Big(\partial_r \big(  \bar{p}_+^{-\frac{1}{\gamma}}\delta{p} \big) - \mathcal{A}(r)\bar{p}_+^{-\frac{1}{\gamma}}\int_{K}^{z} \delta{\theta}(\tau,r)\dif \tau -\bar{p}_+^{-\frac{1}{\gamma}} f_2\Big).
\end{align}
Substituting \eqref{zthetaeq} into \eqref{psecond}, one has
\begin{align}\label{peqeq}
&\partial_z^2 \Big( \frac{1-\overline{M}_+^2}{\bar{\rho}_+\bar{q}_+^2}  \bar{p}_+^{\frac{1}{\gamma}} \delta p \Big) + \partial_r \Big(  \frac{1}{\bar{\rho}_+\bar{q}_+^2}  \bar{p}_+^{\frac{2}{\gamma}} \partial_r \big( \bar{p}_+^{-\frac{1}{\gamma}} \delta p \big)\Big) + \frac{1}{r} \frac{1}{\bar{\rho}_+\bar{q}_+^2}  \bar{p}_+^{\frac{2}{\gamma}} \partial_r \big( \bar{p}_+^{-\frac{1}{\gamma}} \delta p \big)
={F}(z,r)
\end{align}
where
\begin{align*}
{F}(z,r)\defs&\bar{p}_+^{\frac{1}{\gamma}} \partial_z f_1 + \partial_r \Big( \frac{1}{\bar{\rho}_+\bar{q}_+^2}  \bar{p}_+^{\frac{1}{\gamma}} f_2\Big)
+ \frac{1}{r}\frac{1}{\bar{\rho}_+\bar{q}_+^2}  \bar{p}_+^{\frac{1}{\gamma}} f_2 \notag\\
& +\partial_r \Big(\frac{1}{\bar{\rho}_+\bar{q}_+^2} \bar{p}_+^{\frac{1}{\gamma}}\mathcal{A}(r) \int_{K}^{z} \delta{\theta}(\tau,r)\dif \tau\Big)+ \frac{1}{r} \frac{1}{\bar{\rho}_+\bar{q}_+^2} \bar{p}_+^{\frac{1}{\gamma}}\mathcal{A}(r) \int_{K}^{z} \delta{\theta}(\tau,r)\dif \tau\notag.
\end{align*}
The boundary conditions of $\delta p$ on $z=K$ and $z=L$ are
\begin{align}\label{bc:deltap}
\delta p(K,r)= g_1(K,r)\qquad\mbox{and}\qquad \delta p(L,r) = \sigma p_{ex}(L,r).
\end{align}
For the boundary conditions on the nozzle wall and the symmetry axis, it follows from $f_2(z,0) = f_2(z,r_0) =0$, $\delta \theta(z,0)=\delta (z,r_0)=0$, $\delta w(z,0)=0$, $\bar{w}_+(0)=\bar{w}(r_0)=0$, and \eqref{F2eq2} that
\begin{align}\label{pbryN}
\partial_r \delta p(z,0) =0\qquad\mbox{and}\qquad \partial_r \delta p(z,r_0) =0.
\end{align}
Obviously, the mixed boundary problem \eqref{peqeq}-\eqref{pbryN} has singularity near the axis $r=0$.
Similarly as done for {\bf {$\llbracket \textit{Problem I}\rrbracket $}},
let $\delta P(x_1,x_2,x_3)\defs \delta p(x_1, \sqrt{x_2^2 + x_3^2})$. By boundary condition \eqref{pbryN} and Lemma \ref{lem5.4x}, $\delta P\in C^{2,\alpha}(\mathcal{P})$ if and only if $\delta p\in C^{2,\alpha}(\Omega^K_+)$.
Then problem \eqref{peqeq}-\eqref{pbryN} can be rewritten as the following problem:
\begin{align}\label{P3Deq}
\partial_{x_1} \Big(\frac{1-\overline{M}_+^2}{\bar{\rho}_+\bar{q}_+^2}  \bar{p}_+^{\frac{1}{\gamma}} \partial_{x_1}\delta P \Big) + \sum_{i=2}^3\partial_{x_i}
\Big( \frac{\bar{p}_+^{\frac{2}{\gamma}}}{\bar{\rho}_+\bar{q}_+^2 } \partial_{x_i}\big(  \bar{p}_+^{-\frac{1}{\gamma}} \delta P  \big)\Big ) = {F}(x_1,r)\qquad \text{in} \quad \mathcal{P}
\end{align}
with the boundary conditions:
\begin{align}
&\delta P(K,x_2,x_3)= g_1(K,r)&\quad &\text{on}&\quad \mathcal{P}_K\\
&\delta P(L,x_2,x_3) =\sigma p_{ex}(r)&\quad &\text{on}&\quad \mathcal{P}_{ex}\\
&x_2 \partial_{x_2}\delta P(x_1,x_2,x_3) +  x_3 \partial_{x_3}\phi(x_1,x_2,x_3) =0 &\quad &\text{on}&\quad \mathcal{P}_w.
\end{align}
By $\partial_r g_1(K,0) = \partial_r g_1(K,r_0) =0$ and $\partial_r p_{ex}(0) = \partial_r p_{ex}(r_0) =0$, it is easy to see that $g_1$ and $p_{ex}$ are $C^{2,\alpha}$ functions in $\mathcal{P}_K$ and $\mathcal{P}_{ex}$ respecitvely. Because
\[
\|F\|_{C^{0,\alpha}(\mathcal{P})}\leq C\|F\|_{C^{0,\alpha}(\overline{\Omega_+^K})}\leq C\Big(  \sum_{i=1}^2 \| f_i \|_{C^{1,\alpha}(\overline{\Omega_+^K})} + \| g_1 \|_{C^{2,\alpha}([0,r_0])} + \|\sigma  p_{ex}\|_{C^{2,\alpha}([0,r_0])}\Big),
\]
following a similar argumant as given in the proof of Lemma \ref{3Dphi}, we can obtain
\begin{align}
  \| \delta p \|_{C^{2,\alpha}(\overline{\Omega_+^K})}
\leq C\Big(  \sum_{i=1}^2 \| f_i \|_{C^{1,\alpha}(\overline{\Omega_+^K})} + \| g_1 \|_{C^{2,\alpha}([0,r_0])} + \|\sigma  p_{ex}\|_{C^{2,\alpha}([0,r_0])}\Big).
\end{align}

\emph{Step 2:} In this step, we will estimate $\delta \theta$. Applying \eqref{F2eq2}, we have
\begin{align}\label{zthetaEQ}
 \partial_z \delta \theta = -\frac{1}{\bar{\rho}_+\bar{q}_+^2}  \bar{p}_+^{\frac{1}{\gamma}} \Big(\partial_r \big(  \bar{p}_+^{-\frac{1}{\gamma}}\delta{p} \big) - \mathcal{A}(r)\bar{p}_+^{-\frac{1}{\gamma}}\int_{K}^{z} \delta{\theta}(\tau,r)\dif \tau -\bar{p}_+^{-\frac{1}{\gamma}} f_2\Big)
 \in C^{1,\alpha}(\overline{\Omega_+^K}),
\end{align}
where we use the facts that $\delta p \in C^{2,\alpha}(\overline{\Omega_+^K})$, $\delta \theta\in C^{1,\alpha}(\overline{\Omega_+^K})$ and $f_2 \in C^{1,\alpha}(\overline{\Omega_+^K})$.
Next, by equation \eqref{F1eq1}, we have
\begin{align*}
\partial_r \big( r \bar{p}_+^{\frac{1}{\gamma}}\delta{\theta}
\big ) = \partial_z \big(\frac{1-\overline{M}_+^2}{\bar{\rho}_+\bar{q}_+^2} r \bar{p}_+^{\frac{1}{\gamma}} \delta{p}  \big) - r \bar{p}_+^{\frac{1}{\gamma}} f_1.
\end{align*}
So it follows from the condition $\delta \theta(z,0) =0$ that
\begin{align*}
\delta{\theta}(z,r) = \frac{\bar{p}_+^{-\frac{1}{\gamma}}}{r}\int_0^r \tau \Big( \partial_z \big(\frac{1-\overline{M}_+^2}{\bar{\rho}_+\bar{q}_+^2} \bar{p}_+^{\frac{1}{\gamma}} \delta{p}  \big) - \bar{p}_+^{\frac{1}{\gamma}} f_1\Big)\dif \tau.
\end{align*}
Direct calculations yield that
\begin{align}
\partial_r \delta{\theta}(z,r) =& J_1 + J_2,\label{eq:par1theta}
\end{align}
where $J_1\defs\bar{p}_+^{-\frac{1}{\gamma}} J_*$,
$$
J_2\defs (\frac{1}{r} \partial_r \big(\bar{p}_+^{ - \frac{1}{\gamma}}\big) -  \frac{1}{r^2} \bar{p}_+^{-\frac{1}{\gamma}} )\int_0^r \tau J_*(z,\tau)\dif \tau=(r \partial_r \big(\bar{p}_+^{ - \frac{1}{\gamma}}\big) -  \bar{p}_+^{-\frac{1}{\gamma}} )\int_0^1 \tau J_*(z,r\tau)\dif \tau,
$$
and
$$
J_*\defs  \partial_z \big(\frac{1-\overline{M}_+^2}{\bar{\rho}_+\bar{q}_+^2} \bar{p}_+^{\frac{1}{\gamma}} \delta{p}  \big) - \bar{p}_+^{\frac{1}{\gamma}} f_1.
$$
By $\delta p\in C^{2,\alpha}(\overline{\Omega_+^K})$ and $f_1\in C^{1,\alpha}(\overline{\Omega_+^K})$, it is easy to see that $J_* \in C^{1,\alpha}(\overline{\Omega_+^K})$. Hence
$J_1 \in C^{1,\alpha}(\overline{\Omega_+^K})$,
$J_2 \in C^{1,\alpha}(\overline{\Omega_+^K})$,
and it holds that
\begin{align}
  \| \delta \theta \|_{C^{2,\alpha}(\overline{\Omega_+^K})}
\leq C\Big(  \sum_{i=1}^2 \| f_i \|_{C^{1,\alpha}(\overline{\Omega_+^K})} + \| g_1 \|_{C^{2,\alpha}([0,r_0])} + \|\sigma p_{ex} \|_{C^{2,\alpha}([0,r_0])}\Big).
\end{align}

Finally, by equation \eqref{F1eq1}, we know that
\begin{equation*}
\partial_r\delta \theta=\frac{1-\overline{M}_+^2}{\bar{\rho}_+\bar{q}^2_+}\partial_z\delta p-(\frac1r+\frac{\partial_r\bar{p}_+}{\gamma \bar{p}_+})\delta\theta-f_1.
\end{equation*}
So
\begin{equation}\label{5.168x}
\partial_r^2\delta \theta+\partial_r(\frac{\delta\theta}{r})=\partial_r\left(\frac{1-\overline{M}_+^2}{\bar{\rho}_+\bar{q}^2_+}\partial_z\delta p-\frac{\partial_r\bar{p}_+}{\gamma \bar{p}_+}\delta\theta-f_1\right).
\end{equation}
Because $\partial_r\overline{M}_+(0)=\partial_r\bar{\rho}_+(0)=\partial_r\bar{p}_+(0)=\delta\theta(z,0)=\partial_r \delta p(z,0)=\partial_r f_1(z,0)=0$, it is easy to see that
\[
\partial_r\left(\frac{1-\overline{M}_+^2}{\bar{\rho}_+\bar{q}^2_+}\partial_z\delta p-\frac{\partial_r\bar{p}_+}{\gamma \bar{p}_+}\delta\theta-f_1\right)(z,0)=0.
\]
Finally, notice that
\[
\lim_{r\rightarrow0}\partial_r(\frac{\delta\theta}{r})=\lim_{r\rightarrow0}\frac{r\partial_r\delta\theta-\delta\theta}{r^2}=\lim_{r\rightarrow0}\frac{r\partial_{r}^2\delta\theta}{2r}=\frac{1}{2}\partial_r^2\delta \theta(z,0).
\]
Then it follows from \eqref{5.168x} that $\partial_r^2\delta \theta(z,0)=0$.
\end{proof}
Therefore, Theorem \ref{BigThm} follows from Lemma \ref{3Dphi}, Lemma \ref{theta2p2} and Lemma \ref{raiseRegularity}.

\subsection{Cauchy problem for the three linearized transport equations}
Once $\delta\theta$ is determined, we will solve $(\delta w, \delta q,\delta s)$ by considering the following Cauchy problems  in the subsonic domain $\Omega_+^K$.
\begin{eqnarray}
&&\begin{cases}\label{eq:g2}
\big( \partial_z +  H(z,r)\partial_r\big) (r\delta w)+ \partial_r( r\bar{w}_+ )\cdot \delta \theta ={f}_3(z,r)&\text{in}\quad {\Omega}_+^K,\\[3pt]
\delta w(K, r)=g_2(K,r);
\end{cases}\\
&&\begin{cases}\label{eq:g3}
 \big( \partial_z +  H(z,r)\partial_r\big) \delta B + \partial_r\bar{B}_+ \cdot \delta \theta =f_4(z,r)&\text{in}\quad {\Omega}_+^K,\\[3pt]
\delta q(K, r)=g_3(K,r);
\end{cases}\\
&&\begin{cases}\label{eq:g4}
\big( \partial_z +  H(z,r)\partial_r\big) \delta s + \partial_r\bar{s}_+ \delta \theta={f}_5(z,r)& \text{in}\quad {\Omega}_+^K,\\[3pt]
\delta s(K, r)=g_4(K,r).
\end{cases}
\end{eqnarray}
The slop of the shock front $\delta {\varphi}'$ can be easily determined by equation
\begin{align}\label{eq:g5}
\delta {\varphi}' = g_5(K,r).
\end{align}
We further assume that
\begin{align}
&(f_i, H) \in C^{2,\alpha}(\overline{\Omega_+^K}),(i=3,4,5),\quad
g_j\in C^{2,\alpha}([0,r_0]),(j=2,3,4,5),\label{con:regu}\\
&{f}_3(z,0)={f}_3(z,r_0)=0,\quad H(z,0)=H(z,r_0)=0,\label{con:fh}\\
&\partial_r f_j(z,0)=0, \quad  \partial_r f_j(z,r_0)=0, (j=4,5),\label{con:f45}\\
&g_2(K,0)=0,\quad  g_2(K,r_0)=0,\label{con:g2}\\
&\partial_r g_j(K,0)=0,\quad  \partial_r g_j(K,r_0)=0, (j=3,4),\label{con:g34}\\
&g_5(K,0)=0,\quad  g_5(K,r_0)=0,\quad \partial_r^2 g_5(K,0)=0. \label{con:g5}
\end{align}
By the method of characteristics, we can obtain the existence of the solution $(\delta w, \delta q,\delta s)$ for the Cauchy problems \eqref{eq:g2}-\eqref{eq:g4} as follows.
\begin{thm}\label{wqsexistence}
Assume the conditions \eqref{con:regu}-\eqref{con:g5} hold, then the Cauchy problems \eqref{eq:g2}-\eqref{eq:g4} admit a unique solution
$(r\delta w, \delta q,\delta s)\in C^{2,\alpha}({\Omega}_+^K)$ with $\delta w\in C^{1,\alpha}({\Omega}_+^K)$,
which satisfy
\begin{align}\label{est:rwqs}
&\|(r\delta w, \delta q,\delta s)\|_{C^{2,\alpha}(\overline{\Omega_+^K})}+
\|\delta w\|_{C^{1,\alpha}(\overline{\Omega_+^K})}\notag\\
\leq & C\left(
\sum_{i=1}^2 \| f_i \|_{C^{1,\alpha}(\Omega_+^K)} +\sum_{i=3}^5 \| f_i \|_{C^{2,\alpha}(\Omega_+^K)}+\sum_{i=1}^4\| g_i \|_{C^{2,\alpha}([0,r_0])} + \|\sigma p_{ex} \|_{C^{2,\alpha}([0,r_0])}\right),
\end{align}
where constant $C$ depends only on $\overline{U}_{\pm}$, $H$, $K$, $L$, $r_0$ and $\alpha$.
Moreover, it holds that
\begin{align}\label{rqsz0r0}
\partial_r(\delta q)(z,0)=\partial_r(\delta s)(z,0) = \partial_r( \delta q)(z,r_0)=\partial_r(\delta s)(z,r_0) =\delta w(z,r_0)=0
\end{align}
provided that $\delta w(K,r_0)=\partial_r \delta{s}(K,0)=\partial_r\delta s(K,r_0)=\partial_r\delta q(K,0)=\partial_r\delta q(K,r_0)=0$.
For the shock front $\delta\varphi$ satisfying \eqref{eq:g5}, it holds that
\begin{align}\label{eq:deltavar1}
\| \delta\varphi' \|_{C^{2,\alpha}([0,r_0])}\leq\| g_5 \|_{C^{2,\alpha}([0,r_0])}
\end{align}
and
\begin{align}\label{eq:deltavar2}
&\delta\varphi'(0)=\delta\varphi^{(3)}(0)= \delta\varphi'(r_0)= 0.
\end{align}
\end{thm}

\begin{proof}
Note that \eqref{eqdeltaw} and \eqref{eqdeltas} are the solution experssions of problems \eqref{eq:g2} and \eqref{eq:g4} respectively.
Moreover, it follows from \eqref{eq:g3} that
\begin{align}\label{eqdeltaB}
\delta B(z,r) = & B(R(K;z,{r})) + \int_{K}^{z} {f}_4(\tau, R(\tau;z,{r}) )\dif \tau\notag\\
&- \int_{K}^{z} \partial_r\bar{B}_+(R(\tau;z,{r})) \delta \theta(\tau, R(\tau;z,{r}) )\dif \tau.
\end{align}
So estimate \eqref{est:rwqs} follows immediately. Here we use the fact that $W(z,0)=0$ for $W=r\delta w$ such that
$$
\delta w(z,r)=\frac{W(z,r)-W(z,0)}{r}=\int_0^1\partial_{r}W(z,\tau r)d\tau.
$$

Since $H(z,r_0)=\delta\theta(z,r_0)={f}_3(z,r_0)=0$, it follows from
\eqref{eq:g2} that $\partial_z\delta w=0$ on the boundary $r=r_0$. Therefore, $\delta w(z,r_0)=0$ provided that  $\delta w(K,r_0)=0.$

Taking derivative with respect to the variable $r$ on the both sides of the equation \eqref{eq:g4}, we have
\begin{align*}
\big( \partial_z +  H(z,r)\partial_r\big)(\partial_r \delta{s} )
+\partial_r H(z,r)\partial_r \delta{s}+
 \partial_r\bar{s}_+\partial_r \delta{\theta} +\partial_r^2\bar{s}_+\delta{\theta} =\partial_r f_5.
\end{align*}
By applying $H(z,0)=0$, $\partial_r \bar{s}_+(0) =0$, $\delta{\theta}(z,0) =0$ and $\partial_r f_5(z,0)=0$, we have
\begin{align*}
\partial_z (\partial_r \delta{s})(z,0)+\partial_r H(z,0)\partial_r \delta{s}(z,0) =0.
\end{align*}
Thus, $\partial_r \delta{s}(z,0) =0$ provided that $\partial_r \delta{s}(K,0) =0$.

Similarly, we can show $\partial_r\delta s(z,r_0) =0$ and
$\partial_r\delta q(z,0)=\partial_r\delta q(z,r_0)=0$, provided that  $\partial_r\delta s(K,r_0)=\partial_r\delta q(K,0)=\partial_r\delta q(K,r_0)=0$.

For the shock front, \eqref{eq:deltavar1} and \eqref{eq:deltavar2} follow from \eqref{eq:g5} and \eqref{con:g5}.
\end{proof}

\section{Existence of shock solutions to the nonlinear problem}
Based on Section 4 and Section 5, we will design a delicate nonlinear iteration scheme such that  \eqref{assumeregularity}-\eqref{partialrg1pe} and \eqref{con:regu}-\eqref{con:g5} hold to solve the nonlinear boundary value problem \eqref{eqf1}-\eqref{eq11000}.
Then we can apply the Banach fixed-point theorem.
Before doing so, in order to determine the shock front position approximately based on the solvability condition \eqref{thmsolvabilityeq}, we will construct approximate shock solutions, around which the iteration scheme will be introduced.
\subsection{Approximate shock solutions}\label{sec:6.1}
Let $\dot{U}_- =  (\dot{p}_-, \dot{\theta}_-, \dot{w}_-, \dot{q}_-, \dot{s}_-)^\top$ be defined in  $ \dot{\Omega}_- \defs \{ (z,r)\in \mathbb{R}^2 : 0 < z < \dot{z}_*, \, 0 < r < r_0\}$ and satisfy the following linearized Euler equations:
\begin{align}
  & \partial_z \Big(\frac{1-\overline{M}_-^2}{\bar{\rho}_-\bar{q}_-^2} \dot{p}_- \Big) -\Big( \partial_r \dot{\theta}_- + \frac{1}{r}\big(   1+ \frac{\overline{M}_-^2\bar{w}_-^2 }{\bar{q}_-^2}\big)\dot{\theta}_-  \Big) =0\label{CDEI-}\\
  &\partial_z (\bar{\rho}_-\bar{q}_-^2\dot{\theta}_-) +\partial_r \dot{p}_- -\frac{1}{r} \frac{\overline{M}_-^2 \bar{w}_-^2}{\bar{q}_-^2}\dot{p}_- -\frac{2\bar{\rho}_- \bar{w}_-}{r}
  \dot{w}_- + \frac{1}{\gamma c_v}\frac{\bar{\rho}_- \bar{w}_-^2}{r}
  \dot{s}_-  =0\label{CDEII-}\\
  &\partial_z (r\dot{w}_-) + \partial_r( r\bar{w}_- ) \cdot \dot{\theta}_- =0\label{CDEIII-}\\
  &\partial_z\big(\dot{p}_-  + \bar{\rho}_-\bar{q}_- \dot{q}_- \big) + \bar{\rho}_- \bar{q}_- \partial_r \bar{q}_-\cdot\dot{\theta}_- =0\label{CDEIV-}\\
  &\partial_z \dot{s}_- + \partial_r\bar{s}_- \cdot \dot{\theta}_- =0\label{CDEV-}
\end{align}
with initial-boundary conditions:
\begin{align}
	&\dot{U}_- = (0,0, \sigma w_{en}(r), \sigma q_{en}(r) ,0)&\quad &\text{on} \quad \Gamma_{en}\label{U-0}\\
	&\dot\theta_- = 0 &\quad &\text{on} \quad \Gamma_w\cap\overline{\dot\Omega_-}\label{theta-0}\\
&\dot\theta_- = 0\quad \dot{w}_- =0 &\quad &\text{on} \quad \Gamma_a\cap\overline{\dot\Omega_-}.\label{gammaabry}
	\end{align}
Let $ \dot{U}_+ = (\dot{p}_+, \dot{\theta}_+, \dot {w}_+,\dot{q}_+,\dot{s}_+)$ satisfy \eqref{eqf1}, \eqref{eqf3}-\eqref{eqf5}, \eqref{GAMMA4theta}-\eqref{eq11000} and \eqref{eqf22} with
\begin{align}
K\defs \dot{z}_*, \,\, H =0\label{LetUKH=} 
\end{align}
and
\begin{align}
  f_j\defs 0,(j=1,3,4,5), \,\, f_2\defs \frac{2\bar{\rho}_+ \bar{w}_+}{r}
  \dot{w}_+ (\dot{z}_*, r) - \frac{1}{\gamma c_v}\frac{\bar{\rho}_+ \bar{w}_+^2}{r}\dot{s}_+ (\dot{z}_*, r).\label{Letfjf2=}
\end{align}
Then condition \eqref{thmsolvabilityeq} in Theorem \ref{BigThm} become
 \begin{align}\label{solvability0}
  \int_0^{r_0} \frac{1-\overline{M}_+^2}{\bar{\rho}_+\bar{q}_+^2}  r\bar{p}_+^{\frac{1}{\gamma}}\dot{p}_+(\dot{z}_*, r)\dif r =  \int_0^{r_0} \frac{1-\overline{M}_+^2}{\bar{\rho}_+\bar{q}_+^2}  r\bar{p}_+^{\frac{1}{\gamma}} \sigma p_{ex} (r)\dif r.
\end{align}
Notice that $\dot{p}_+(\dot{z}_*, r)$ is an unknown quantity in \eqref{solvability0}, which can be determined by the supersonic solution $\dot{U}_-$ ahead of the shock front and the R.-H. conditions \eqref{eq251RH}-\eqref{eq11000}, which can be rewritten as on $\dot{\Gamma}_s\defs \{(z,r) : z = \dot{z}_*, \, 0< r < r_0\}$,
\begin{align}
	&{\mathbf{\alpha}}_j^+ \cdot {\dot{U}}_+ = \dot{G}_j\defs {\mathbf{\alpha}}_j^- \cdot {\dot{U}}_-, \quad j = 1,2,3, 4,\label{initialRH}\\
	&{\mathbf{\alpha}}_5^+\cdot {\dot{U}}_+  - [\bar{p}]{\dot{\varphi}}' = {\mathbf{\alpha}}_5^- \cdot {\dot{U}}_-,\label{initialsf}
	\end{align}
where $\alpha_j^+$ are given in \eqref{alpha1}-\eqref{alpha5} and
 \begin{align}
	& {\mathbf{\alpha}}_1^- = \Big(\frac{\bar{q}_-}{\bar{c}_-^2},\, 0,\,0,\, \bar{\rho}_-,\, -\frac{\bar{\rho}_- \bar{q}_-}{\gamma c_v} \Big)^\top,\label{alpha-1}\\
	& {\mathbf{\alpha}}_2^-  = \Big(1+\overline{M}_-^2 ,\,0,\, 0,\, 2\bar{\rho}_- \bar{q}_-, \, -\frac{\bar{\rho}_- \bar{q}_-^2}{\gamma c_v}\Big)^\top,\quad {\mathbf{\alpha}}_3^- = \Big(0,\,0,\,1,\,0 ,\,0 \Big)^\top, \\
	& {\mathbf{\alpha}}_4^-  = \Big(\frac{1}{\bar{\rho}_-},\,0,\,0,\,\bar{q}_-,\, \frac{\bar{p}_-}{(\gamma -1)c_v \bar{\rho}_-}\Big)^\top,\quad  {\mathbf{\alpha}}_5^-  = \Big(0,\, \bar{\rho}_- \bar{q}_-^2,\, 0,\,0,\, 0\Big)^\top.\label{alpha-5}
	 \end{align}
Following the argument in \eqref{AMRS}-\eqref{varphi}, we can solve \eqref{initialRH} and \eqref{initialsf} so that
\begin{align}
&\big(\dot{p}_+,\,  \dot{w}_+, \,  \dot{q}_+,\, \dot{s}_+\big)(\dot{z}_*, r) = \big(\dot{g}_1(r), \, \dot{g}_2(r),\,   \dot{g}_3(r),\,\dot{g}_4(r) \big)\defs A_s^{-1}\dot{\mathbf{G}},\label{RHinitial}\\
&\dot{\varphi}' = \frac{\bar{\rho}_+\bar{q}_+^2 \dot{\theta}_+ -\bar{\rho}_-\bar{q}_-^2 \dot{\theta}_-}{[\bar{p}]}\defs \dot{g}_5(r),\label{dotvar}
\end{align}
where $\dot{\mathbf{G}}=(\dot{G}_1,\dot{G}_2,\dot{G}_3,\dot{G}_4)$, and
\begin{align}
\dot{g}_1(r)
 =& \frac{\bar{\rho}_+\bar{q}_+^2}{\overline{M}_+^2 -1} \Big(\frac{\overline{M}_-^2-1}{\bar{\rho}_-\bar{q}_-^2}( 1 - \kappa_1(r) )\dot{p}_- + \kappa_2(r) \sigma q_{en}(r) + \kappa_3(r)\int_0^{\dot{z}_*}\dot{\theta}_-(z,r)\dif z\Big),\label{dotprh}\\
 \dot{g}_2(r)=& \sigma w_{en}(r) -  \big(\partial_r \bar{w}_- + \frac{\bar{w}_-}{r}\big) \int_0^{\dot{z}_*} \dot{\theta}_-(z,r)\dif z,\label{dotw}\\
\dot{g}_3(r) =& \frac{1}{\bar{\rho}_+ \bar{q}_+} \frac{\overline{M}_-^2-1}{\bar{\rho}_-\bar{q}_-^2} \Big( [\bar{p}] - \frac{\bar{\rho}_+\bar{q}_+^2}{\overline{M}_+^2 -1}  ( 1 - \kappa_1(r) )  \Big)\dot{p}_-\notag\\
& + \Big( 1  + \frac{[\bar{p}]}{\bar{\rho}_- \bar{q}_-^2}  -  \frac{1}{\bar{\rho}_+ \bar{q}_+} \frac{\bar{\rho}_+\bar{q}_+^2}{\overline{M}_+^2 -1} \kappa_2(r)  \Big)\sigma q_{en}(r)\notag\\
& +\Big( \Big( \frac{\partial_r\bar{s}_- }{\gamma c_v}  - \frac{\partial_r \bar{q}_- }{\bar{q}_- }  \Big) \frac{[\bar{p}]}{\bar{\rho}_+ \bar{q}_+} - \partial_r \bar{q}_- -  \frac{1}{\bar{\rho}_+ \bar{q}_+} \frac{\bar{\rho}_+\bar{q}_+^2}{\overline{M}_+^2 -1}\kappa_3(r)\Big) \int_0^{\dot{z}_*}\dot{\theta}_-(z,r)\dif z,\label{dotq}\\
\dot{g}_4(r)= & - \frac{(\gamma-1)c_v}{ \bar{p}_+}\frac{\overline{M}_-^2-1}{\bar{\rho}_-\bar{q}_-^2}[\bar{p}] \dot{p}_- + \frac{(\gamma-1)c_v}{ \bar{p}_+}\Big(\bar{\rho}_+ \bar{q}_- -  \bar{\rho}_- \bar{q}_-  - \frac{[\bar{p}]}{\bar{q}_-}  \Big)\sigma {q}_{en}(r)\notag\\
 &-\frac{(\gamma-1)c_v}{ \bar{p}_+}\Big( \Big(\frac{\bar{\rho}_+\bar{p}_-}{(\gamma-1)c_v\bar{\rho}_-} + \frac{[\bar{p}]}{\gamma c_v} \Big)\partial_r\bar{s}_-  + \Big(  \bar{\rho}_+ \bar{q}_- - \bar{\rho}_-\bar{q}_- - \frac{ [\bar{p}] }{\bar{q}_- }  \Big) \partial_r \bar{q}_- \Big)\notag\\
 &\quad \times \int_0^{\dot{z}_*}\dot{\theta}_-(z,r)\dif z,\label{dots}
\end{align}
with
\begin{align}
   \kappa_1(r)\defs &  \Big(\frac{1}
  {\bar{\rho}_+\bar{q}_+^2} +
  \frac{\gamma-1}{\gamma \bar{p}_+}\Big)[\bar{p}],\label{kappa1r}\\
   \kappa_2 (r)\defs &\Big( \frac{1}{\bar{q}_+\bar{q}_-} - \frac{\gamma-1}{\bar{c}_+^2} \Big) (\bar{q}_+ - \bar{q}_-) - \frac{\kappa_1(r)}{\bar{q}_-},\\
    \kappa_3 (r)\defs &\Big( 1- \kappa_1(r) - \frac{\bar{\rho}_+ \bar{p}_-}{\bar{\rho}_- \bar{p}_+}\Big)\frac{\partial_r\bar{s}_- }{\gamma c_v}\notag\\
    & + \Big( \kappa_1(r)+
  \big(\frac{1}{\bar{q}_+} - \frac{(\gamma-1)\bar{q}_-}{\bar{c}_+^2}   \big)\big( \bar{q}_-  - \bar{q}_+ \big) \Big) \frac{\partial_r \bar{q}_-}{\bar{q}_-}.\label{kappa3r}
\end{align}
Notice that $\dot{g}_1=\dot{p}_+(\dot{z}_*, r)$ given by \eqref{dotprh} involves with $\dot{p}_-$ and $\dot{\theta}_-$, so we need to solve the initial-boundary value problem \eqref{CDEI-}-\eqref{gammaabry} at the same time.
Similarly as done in Section 4, we solve the initial-boundary value problem \eqref{CDEI-}-\eqref{gammaabry} as follows.
\begin{lem}\label{U-esti}
Assume that $(w_{en}, q_{en})\in C^7([0,r_0])^2$,  \eqref{assumew0r0=1}-\eqref{assumewenqen} hold, and the equations \eqref{CDEI-}-\eqref{CDEV-} with the boundary conditions \eqref{U-0}-\eqref{gammaabry} satisfy the compatibility conditions up to the seventh order, then there exists a sufficiently small positive constant $\dot{\sigma}_L$ depending on $\overline{U}_-$ and $L$, such that for any $0<\sigma<\dot{\sigma}_L$, there exists a unique solution $\dot{U}_-$ to problem \eqref{CDEI-}-\eqref{gammaabry}, which satisfies
  \begin{align}\label{dotU-estimate}
	\|\dot{U}_-\|_{C^3(\overline{\Omega})}  \leq  C_{(\overline{U}_-)} \sigma \big(\| w_{en} \|_{H^7(\Gamma_{en})} + \| q_{en} \|_{H^7(\Gamma_{en})}\big),
	\end{align}
where the positive constant $ C_{(\overline{U}_-)}$ only depends on $\overline{U}_-$ and $L$. Moreover,
\begin{align}
  &\partial_r^j (\dot{p}_- , \dot{q}_- , \dot{s}_-)(z,0) = 0, \quad j=1,3,\label{j=13==0}\\
  &\partial_r^2 \dot{\theta}_- (z,0) =0,\quad  \partial_r^2 \dot{w}_- (z,0) =0, \label{Iaxis-}\\
 &\dot{w}_- (z,r_0)=0, \quad \partial_r (\dot{p}_- , \dot{q}_- , \dot{s}_-)(z,r_0)=0.\label{upper-}
\end{align}
\end{lem}
\begin{proof}
This proof is divided into two steps.

\emph{Step 1:} In this step, we prove the unique existence of the solution $\dot{U}_-$ that satisfies \eqref{dotU-estimate}.
Applying \eqref{CDEIII-}, \eqref{CDEV-} and \eqref{U-0}, direct calculations yield that
\begin{align}
 \dot{w}_-(z,r) =& \sigma w_{en} (r) - \big(\partial_r \bar{w}_- + \frac{\bar{w}_-}{r}\big) \int_{0}^{z} \dot{\theta}_-(\tau,r)\dif \tau,\label{dotw--}\\
 \dot{s}_-(z,r) =&  -  \partial_r\bar{s}_- \int_{0}^{z} \dot{\theta}_-(\tau,r)\dif \tau.\label{dots--}
\end{align}
Substituting \eqref{dotw--} and \eqref{dots--} into \eqref{CDEII-}, further calculations yield that
 \begin{align}\label{CDEIIFur}
    &\partial_z \Big(\bar{\rho}_-\bar{q}_-^2 \bar{p}_-^{-\frac{1}{\gamma}}\dot{\theta}_-\Big) +\partial_r \Big( \bar{p}_-^{-\frac{1}{\gamma}}\dot{p}_-\Big)\notag\\
    &- \frac{\bar{\rho}_-\bar{w}_-}{r}\bar{p}_-^{-\frac{1}{\gamma}}
  \Big(\frac{\partial_r\bar{s}_-}{\gamma c_v} \bar{w}_- - 2\big(\partial_r \bar{w}_- + \frac{\bar{w}_-}{r}\big) \Big)\int_{0}^{z} \dot{\theta}_-(\tau,r)\dif \tau\notag\\
=& \frac{2\bar{\rho}_- \bar{w}_-}{r}\bar{p}_-^{-\frac{1}{\gamma}} \sigma w_{en}(r).
 \end{align}
 In addition, \eqref{CDEI-} can be rewritten as the following form:
 \begin{align}\label{CDEI-re}
\partial_z \Big(\frac{1-\bar{M}_-^2}{\bar{\rho}_-\bar{q}_-^2} r\bar{p}_-^{\frac{1}{\gamma}}\dot{p}_-\Big) -\partial_r \Big(  r\bar{p}_-^{\frac{1}{\gamma}} \dot{\theta}_-\Big ) =0.
  \end{align}
So there exists a potential function $\widehat{\Psi}_-$ with $\widehat{\Psi}_-(0,0)=0$ such that
\begin{align}\label{hatPsi-}
  \nabla \widehat{\Psi}_- = (\partial_z \widehat{\Psi}_-, \partial_r \widehat{\Psi}_-) =& \Big( r \bar{p}_-^{\frac{1}{\gamma}} \dot{\theta}_-, \,\, \frac{1-\overline{M}_-^2}{\bar{\rho}_-\bar{q}_-^2}  r\bar{p}_-^{\frac{1}{\gamma}}\dot{p}_- \Big).
\end{align}
That is
\begin{align}\label{thetapwidepsi}
  \dot{\theta}_- = \frac{\bar{p}_-^{-\frac{1}{\gamma}} }{r }\partial_z \widehat{\Psi}_-, \quad
  \dot{p}_- = \frac{\bar{\rho}_-\bar{q}_-^2}{1-\overline{M}_-^2}
  \frac{\bar{p}_-^{-\frac{1}{\gamma}}}{r}\partial_r \widehat{\Psi}_-.
\end{align}
By  \eqref{U-0}-\eqref{gammaabry} and \eqref{thetapwidepsi}, we have
\begin{align}\label{hatPsibry}
  \widehat{\Psi}_-(0,r)=0, \quad \partial_z \widehat{\Psi}_-(0,r)=0,\quad  \widehat{\Psi}_-(z,0)=0,\quad  \widehat{\Psi}_-(z,r_0)=0.
\end{align}
Let
\begin{align*}
  \Psi_- = \frac{\widehat{\Psi}_-}{r^2}.
\end{align*}
Then \eqref{thetapwidepsi} implies that
\begin{align}\label{thetapPsi-}
  \dot{\theta}_- = r \bar{p}_-^{-\frac{1}{\gamma}}\partial_z {\Psi}_-, \quad
  \dot{p}_- = \frac{\bar{\rho}_- \bar{q}_-^2}{1-\overline{M}_-^2}
 \bar{p}_-^{-\frac{1}{\gamma}}(2\Psi_- + r \partial_r {\Psi}_-),
\end{align}
and \eqref{hatPsibry} yields that
\begin{align}\label{psi-D}
   {\Psi}_-(0,r)=0, \quad \partial_z {\Psi}_-(0,r)=0,\quad \partial_r{\Psi}_-(z,0)=0,\quad {\Psi}_-(z,r_0)=0.
\end{align}
Substituting \eqref{thetapPsi-} into \eqref{CDEIIFur}, we have
\begin{align}\label{Psi--eq}
  &\partial_z \Big(\bar{\rho}_-\bar{q}_-^2 \bar{p}_-^{-\frac{2}{\gamma}}\partial_z {\Psi}_-\Big) +\frac{1}{r}\partial_r \Big( \bar{p}_-^{-\frac{2}{\gamma}} \frac{\bar{\rho}_- \bar{q}_-^2}{1-\overline{M}_-^2}
 (2\Psi_- + r \partial_r {\Psi}_-)  \Big)\notag\\
 & - \frac{\bar{\rho}_-\bar{w}_-}{r}\bar{p}_-^{-\frac{2}{\gamma}}
  \Big(\frac{\partial_r\bar{s}_-}{\gamma c_v} \bar{w}_- - 2\big(\partial_r \bar{w}_- + \frac{\bar{w}_-}{r}\big) \Big){\Psi}_-\notag\\
  =&\frac{2\bar{\rho}_- \bar{w}_-}{r^2}\bar{p}_-^{-\frac{1}{\gamma}}\sigma w_{en}(r).
\end{align}
Let $\mathcal{D}_- \defs \{(z,\mathbf{r}): z\in (0, L), \mathbf{r}\in \mathbb{R}^4, |\mathbf{r}|< r_0 \}$,
where $\mathbf{r} = (r_1, r_2, r_3, r_4)$ and $\sum\limits_{i=1}^4 r_i^2 = r^2$. Let $\psi_- (z, r_1, r_2, r_3, r_4) = \Psi_- (z, r)$, then by boundary condition \eqref{psi-D} on $\{r=0\}$, it follows from Lemma \ref{lem5.4x} that $\psi_-\in C^{2,\alpha}(\overline{\Omega})$ if and only if $\Psi_-\in C^{2,\alpha}(\mathcal{D}_-)$.
Moreover, \eqref{Psi--eq} becomes the following equation in $\mathcal{D}_-$
\begin{align}\label{5DP2-}
  &\partial_z \Big(\bar{\rho}_- \bar{q}_-^2 \bar{p}_-^{-\frac{2}{\gamma}}\partial_z \psi_-  \Big) + \sum_{i=1}^4 \partial_{r_i} \Big( \frac{\bar{\rho}_- \bar{q}_-^2}{1-\bar{M}_-^2} \bar{p}_-^{-\frac{2}{\gamma}}\partial_{r_i} \psi_-\Big)\notag\\
  & + \Big\{ \frac{2}{r}\partial_r\Big( \frac{\bar{\rho}_-\bar{q}_-^2}{1-\bar{M}_-^2} \bar{p}_-^{-\frac{2}{\gamma}}\Big) - \frac{\bar{\rho}_-\bar{w}_-}{r}\Big(\frac{ \partial_r\bar{s}_-}{\gamma c_v} \bar{w}_- - 2\big(\partial_r \bar{w}_- + \frac{\bar{w}_-}{r}\big)  \Big)\bar{p}_-^{-\frac{2}{\gamma}}\Big\}\psi_- \notag\\
  =&\frac{2\bar{\rho}_- \bar{w}_-}{r^2}\bar{p}_-^{-\frac{1}{\gamma}}\sigma w_{en}(r),
\end{align}
with the following boundary conditions
\begin{align}
   &{\psi}_-(0,r_1,r_2,r_3,r_4)=0, \quad \partial_z {\psi}_-(0,r_1,r_2,r_3,r_4)=0,\\
  &{\psi}_-(z,r_1,r_2,r_3,r_4)=0,\quad \text{for}\quad  |\mathbf{r}| = r_0.\label{psibryr0=}
\end{align}
Following the proof of Theorem \ref{mainthmU-} in Section 4, we can obtain the unique existence of the solution $\dot{U}_-$ that satisfies \eqref{dotU-estimate}, if the assumptions stated in the lemma hold.

\emph{Step 2:} In this step, we prove the conditions \eqref{j=13==0}-\eqref{upper-} hold.
 First, by the equation \eqref{CDEIII-} and the condition $\dot{\theta}_-(z,r_0)=0$, we have
 \begin{align}\label{dotw-r0=0}
    \dot{w}_-(z,r_0) = \dot{w}_-(0,r_0) = \sigma w_{en}(r_0) = 0.
 \end{align}
By \eqref{CDEII-}, $\bar{w}_-(0)=0$,  $\dot{w}_-(z,0) =0$ and $\dot{\theta}_-(z,0)=0$, we can obtain
 \begin{align}\label{rp-0}
  \partial_r \dot{p}_- (z,0)=0.
 \end{align}
Similarly, it follows from \eqref{CDEIII-}, $\bar{w}_-(r_0)=0$, $\dot{\theta}_-(z,r_0)=0$, and \eqref{dotw-r0=0} that
 \begin{align}\label{rp-r0}
  \partial_r \dot{p}_- (z,r_0)=0.
 \end{align}
 Take derivative on the both sides of the equation \eqref{CDEI-} with respect to $r$,
\begin{align}\label{derir-}
    &\frac{1-\bar{M}_-^2}{\bar{\rho}_-\bar{q}_-^2}\partial_{zr} \dot{p}_-
    + \partial_r\Big(  \frac{1-\bar{M}_-^2}{\bar{\rho}_-\bar{q}_-^2}  \Big)\partial_z \dot{p}_- -\partial_r^2\dot{\theta}_- -\frac{1}{r}\partial_r \dot{\theta}_- + \frac{1}{r^2} \dot{\theta}_- \notag\\
    &- \frac{\bar{M}_-^2\bar{w}_-^2}{r \bar{q}_-^2} \partial_r\dot{\theta}_-
    -\partial_r\Big(\frac{\bar{M}_-^2\bar{w}_-^2}{r\bar{q}_-^2}\Big) \cdot \dot{\theta}_- =0.
 \end{align}
 Applying \eqref{pro1}, \eqref{pro2}, \eqref{rp-0} and $\lim\limits_{r\rightarrow 0}\frac{r\partial_r \dot{\theta}_- (z,r) -\dot{\theta}_- (z,r) }{r^2} = \frac12 \partial_r^2\dot{\theta}_-(z,0)$, then \eqref{derir-} implies that
 \begin{align}
   \partial_r^2\dot{\theta}_-(z,0) =0.
 \end{align}
Taking derivative with respect to $r$ on the both sides of the equation \eqref{CDEV-}, we have
 \begin{align}
    \partial_z (\partial_r \dot{s}_- )+ \partial_r\bar{s}_- \cdot \partial_r \dot{\theta}_- + \dot{\theta}_-\cdot \partial_r^2\bar{s}_- =0.
 \end{align}
Applying $\dot{\theta}_-(z,0) =0$ and $\partial_r \bar{s}_-(0) =0$, we have
$ \partial_z (\partial_r \dot{s}_- )(z,0) =0$.
 Thus
 \begin{align}\label{parrsz0=}
   \partial_r \dot{s}_- (z,0) = \partial_r \dot{s}_- (0,0) = 0.
 \end{align}
 Similarly, applying $\dot{\theta}_-(z,r_0) =0$ and $\partial_r \bar{s}_-(r_0) =0$, we have
 \begin{align}\label{parszr0=}
   \partial_r \dot{s}_- (z,r_0) = \partial_r \dot{s}_- (0,r_0) = 0.
 \end{align}
By similar calculations as done for \eqref{parrsz0=} and \eqref{parszr0=}, we have
 \begin{align}
    \partial_r \dot{q}_- (z,0)  = 0, \quad  \partial_r \dot{q}_- (z,r_0)  = 0.
 \end{align}

Next, we will prove
\begin{align}\label{6.55x}
  \partial_r^2 \dot{w}_- (z,0) =0.
\end{align}
\eqref{CDEIII-} yields that
\begin{align}\label{rcdeiii-}
  \partial_z \dot{w}_- + \big( \partial_r \bar{w}_- + \frac{\bar{w}_-}{r} \big) \dot{\theta}_- =0.
\end{align}
Taking twice derivatives with respect to $r$ on the both sides of the equation \eqref{rcdeiii-}, we have
\begin{align}\label{wzrr=0}
  \partial_{zrr} \dot{w}_- + 2\big( \partial_{rr} \bar{w}_- + \partial_r(\frac{\bar{w}_-}{r}) \big) \partial_r \dot{\theta}_-  + \partial_r^2 \big( \partial_r \bar{w}_- + \frac{\bar{w}_-}{r} \big) \dot{\theta}_- + \big( \partial_r \bar{w}_- + \frac{\bar{w}_-}{r} \big) \partial_{rr}\dot{\theta}_- =0.
\end{align}
Then it follows from $\bar{w}_-(0) =0$, $\partial_r^2 \bar{w}_-(0)=0$, $\dot{\theta}_-(z,0)=0$, $\partial_r^2 \dot{\theta}_-(z,0)=0$, and \eqref{wzrr=0} that
\begin{align}\label{zrrdotwz0=0}
   \partial_{zrr} \dot{w}_- (z,0) =0.
\end{align}
So by $\partial_r^2 w_{en}(0,0)=0$, \eqref{6.55x} holds. 

Now we will prove that $\partial_r^3 \dot{p}_-(z,0) =0$.
Taking twice derivatives with respect to $r$ on the both sides of \eqref{CDEII-}, we have
\begin{align}\label{rrrdotp-==0}
  \partial_r^2(\partial_z \big(\bar{\rho}_-\bar{q}_-^2\dot{\theta}_-) \big) +\partial_r^3 \dot{p}_- - \partial_r^2\big(\frac{1}{r} \frac{\overline{M}_-^2 \bar{w}_-^2}{\bar{q}_-^2}\dot{p}_-\big) -\partial_r^2\big(\frac{2\bar{\rho}_- \bar{w}_-}{r}
\dot{w}_-\big) + \partial_r^2\big(\frac{1}{\gamma c_v}\frac{\bar{\rho}_- \bar{w}_-^2}{r}
\dot{s}_- \big) =0.
\end{align}
Note that
\begin{align}\label{r^2dottheta-=0}
\partial_r^2(\partial_z \big(\bar{\rho}_-\bar{q}_-^2\dot{\theta}_-) \big)\big|_{r=0} =\Big((\bar{\rho}_-\bar{q}_-^2)_{rr}\partial_z \dot{\theta}_-  + 2(\bar{\rho}_-\bar{q}_-^2)_r \partial_{zr} \dot{\theta}_- + \bar{\rho}_-\bar{q}_-^2 \partial_{zrr} \dot{\theta}_-\Big)\big|_{r=0}
 = 0.
\end{align}
So it follows from \eqref{rrw2r=}, \eqref{rp-0}, and Lemma \ref{barvalue} that
\begin{equation}
\begin{array}{rl}
&\partial_r^2\big(\frac{1}{r} \frac{\overline{M}_-^2 \bar{w}_-^2}{\bar{q}_-^2}\dot{p}_-\big)\big|_{r=0}\\[2mm]
 = &\Big(\partial_r^2(\frac{ \bar{w}_-^2}{r})\cdot (\frac{\overline{M}_-^2 }{\bar{q}_-^2}\dot{p}_-) + 2\partial_r(\frac{ \bar{w}_-^2}{r})\cdot \partial_r(\frac{\overline{M}_-^2 }{\bar{q}_-^2}\dot{p}_-) +  \frac{ \bar{w}_-^2}{r} \partial_r^2 (\frac{\overline{M}_-^2 }{\bar{q}_-^2}\dot{p}_-)\Big)\big|_{r=0}\\[2mm]
=&0.
\end{array}
\end{equation}
Similarly, we can show
\begin{align}
  \partial_r^2\big(\frac{1}{\gamma c_v}\frac{\bar{\rho}_- \bar{w}_-^2}{r}
\dot{s}_- \big)\big|_{r=0} =0.
\end{align}
Moreover, it follows from $\dot{w}_-(z,0)=0$, $\bar{w}_-(0)=0$, $\partial_r^2 \bar{w}_-(0)=0$, and $\partial_r \bar{\rho}_-(0)=0$, that
\begin{equation}\label{r^2wrho=0}
\begin{array}{rl}
&\partial_r^2\big(\frac{2\bar{\rho}_- \bar{w}_-}{r}
\dot{w}_-\big)\big|_{r=0}\\[2mm]
=& \Big( \partial_r^2(\frac{ \bar{w}_-}{r})\cdot (\bar{\rho}_-
\dot{w}_-)  + 2\partial_r(\frac{ \bar{w}_-}{r}) \partial_r (\bar{\rho}_-
\dot{w}_-) + \frac{ \bar{w}_-}{r} \partial_r^2( \bar{\rho}_-
\dot{w}_- ) \Big)\big|_{r=0}\\[2mm]
 =& 0.
\end{array}
\end{equation}
Thus, \eqref{rrrdotp-==0}-\eqref{r^2wrho=0} yield that
\begin{align}\label{r^3dotp-=0}
  \partial_r^3 \dot{p}_- (z,0)=0.
\end{align}

Finally and similarly, by tanking the third derivatives with respect to $r$ on the both sides of the equations \eqref{CDEIV-} and \eqref{CDEV-}, and applying $\dot{\theta}_-(z,0)=0$, \eqref{r^2dottheta-=0}-\eqref{r^3dotp-=0} and $\partial_r^j (\bar{p}_-, \bar{q}_-, \bar{\rho}_-, \bar{s}_-)(0)=0$, $j=1,3$, it is easy to obtain
\begin{align}
  \partial_r^3(\dot{q}_-, \dot{s}_-)(z,0) =0.
\end{align}
It completes the proof of this lemma.
\end{proof}

Now we are ready to prove Lemma \ref{lemma:approxz*}.
\begin{proof}[Proof of Lemma \ref{lemma:approxz*}]
Direct calculations yield that
\begin{align}\label{ISOL}
{I}_1'(z)=&\int_0^{r_0} r\bar{p}_+^{\frac{1}{\gamma}}(r) ( 1- \kappa_1(r)) \frac{\overline{M}_-^2(r) -1}{\bar{\rho}_-(r)\bar{q}_-^2(r)}
   \partial_z\dot{p}_-(z, r)\dif r \notag\\
   &+ \int_0^{r_0} r\bar{p}_+^{\frac{1}{\gamma}}(r)\kappa_3(r) \dot{\theta}_-(z,r)\dif r.
\end{align}
By employing \eqref{CDEI-}, \eqref{theta-0} and \eqref{gammaabry}, it follows from \eqref{ISOL} that
 \begin{align}\label{RISOL}
 {I}_1'(z)
   =&\int_0^{r_0}r\bar{p}_+^{\frac{1}{\gamma}} (r)( \kappa_1(r)-1) \Big( \partial_r \dot{\theta}_-(z, r) + \frac{1}{r}\big(1+ \frac{\overline{M}_-^2(r)\bar{w}_-^2(r) }{\bar{q}_-^2(r)}\big)\dot{\theta}_-(z, r)\Big)\dif r\notag\\
   &+ \int_0^{r_0}  r\bar{p}_+^{\frac{1}{\gamma}}(r)\kappa_3(r)\dot{\theta}_-(z,r)\dif r\notag\\
  = &\int_0^{r_0} i_1(r) \dot{\theta}_-(z,r)\dif r.
\end{align}
By applying the equation \eqref{CDEII-}, it follows from \eqref{RISOL} that
\begin{align}
I_1''(z) =&  \int_0^{r_0}  \partial_z\big(i_1(r) \dot{\theta}_-(z,r)\big)\dif r\notag\\
=& - \int_0^{r_0}\frac{i_1(r)}{\bar{\rho}_-(r)\bar{q}_-^2(r)}
  \partial_r \dot{p}_-(z,r)\dif r + \int_0^{r_0} \frac{i_1(r)\overline{M}_-^2(r)}{\bar{\rho}_-(r)\bar{q}_-^4(r)}
  \frac{\bar{w}_-^2(r)}{r}\dot{p}_-(z,r)\dif r\notag\\
   &+ \int_0^{r_0}\frac{2i_1(r)}{\bar{q}_-^2(r)}\frac{\bar{w}_-(r)}{r} \dot{w}_-(z,r)\dif r - \int_0^{r_0} \frac{i_1(r)}{\gamma c_v \bar{q}_-^2(r)}\frac{\bar{w}_-^2(r)}{r}
  \dot{s}_-(z,r)\dif r.
\end{align}
It is easy to see that  ${I}_1(0)= {I}_1'(0)=0$ and
\begin{align}\label{Wper}
 {I}_1'' (0) =\int_0^{r_0}\frac{2i_1(r)}{\bar{q}_-^2(r)}\frac{\bar{w}_-(r)}{r} \dot{w}_-(0,r)\dif r = 2\sigma\int_0^{r_0}\frac{i_1(r)}{\bar{q}_-^2(r)}\frac{\bar{w}_-(r)}{r} w_{en}(r)\dif r = 2\sigma I_3.
\end{align}
So it follows from \eqref{RISOL} that there exists a $\hat{z}\in (0,z)$ such that
\begin{align}\label{IZZ}
{I}_1'(z) =& 2 \sigma I_3 z +  \frac12 z^2 \int_{0}^{r_0}i_1(r) \partial_z^2 \dot{\theta}_-(\hat{z},r)\dif r.
\end{align}
By applying \eqref{dotU-estimate}, we have
\begin{align*}
 {I}_1'(z)\geq  2\sigma I_3 z - \frac{ C_{(\overline{U}_-)} \sigma( \|w_{en}\|_{H^7(\Gamma_{en})} + \|q_{en}\|_{H^7(\Gamma_{en})} )}{2}  z^2 \int_{0}^{r_0} |i_1(r) | \dif r .
\end{align*}
By \eqref{letinterzzr-1}, we have ${I}_1'(z) > 0.$ Thus
\begin{align*}
&\min \limits_{z\in(0,L_*)} I_1(z) = I_1(0) =0,\\
&\max\limits_{z\in(0,L_*)} I_1(z) = I_1(L_*) = \int_0^{L_*} \Big(2 \sigma I_3 z +  \frac12 z^2 \int_{0}^{r_0}i_1(r) \partial_z^2 \dot{\theta}_-(\hat{z},r)\dif r\Big) \dif z \notag\\
&\qquad  = L_*^2 \Big(\sigma I_3 + \frac{L_*}{6} \int_{0}^{r_0}i_1(r) \partial_z^2 \dot{\theta}_-(\hat{z},r)\dif r   \Big)\notag\\
&\qquad \geq L_*^2 \Big( \sigma I_3 - \frac{ C_{(\overline{U}_-)} \sigma( \|w_{en}\|_{H^7(\Gamma_{en})} + \|q_{en}\|_{H^7(\Gamma_{en})} )}{6} L_*\int_{0}^{r_0} |i_1(r)| \dif r \Big).
\end{align*}
Applying \eqref{I2interval}, it is easy to see that there exists a unique $\dot{z}_* \in (0, L_*)$ such that  \eqref{eq0661} holds.
\end{proof}

Finally, we will construct the approximate solution in the quasi-subsonic domain, that the iteration space $\mathcal{N}$ below is based on. Applying \eqref{Iaxis-}-\eqref{upper-}, \eqref{pro1}-\eqref{pro2}, \eqref{LetUKH=}-\eqref{Letfjf2=} and \eqref{dotvar}-\eqref{dots}, it is easy to check that \eqref{assumeregularity}-\eqref{partialrg1pe} and \eqref{con:regu}-\eqref{con:g5} hold.
Then it follows from Theorem \ref{BigThm}, Theorem \ref{wqsexistence} and Lemma \ref{lemma:approxz*} that there exists a unique solution $\dot{U}_+$ in the domain $\dot{\Omega}_+ = \{ (z,r)\in \mathbb{R}^2 : \dot{z}_* < z < L,\, 0 < r< r_0\}$, which satisfies
\begin{align}\label{dotU+es=}
& \|(\dot{p}_+, \dot{\theta}_+, r\dot{w}_+, \dot{q}_+, \dot{s}_+)\|_{C^{2,\alpha}(\overline{\dot{\Omega}_+})}+
\|\dot{w}_+\|_{C^{1,\alpha}(\overline{\dot{\Omega}_+})}+
\| \dot{\varphi}'\|_{C^{2,\alpha}(\dot{\Gamma}_s)}\notag\\
\leq& \dot{C}_+ \sigma \Big(\| w_{en} \|_{H^7(\Gamma_{en})} + \| q_{en} \|_{H^7(\Gamma_{en})} +\|p_{ex} \|_{C^{2,\alpha}(\Gamma_{ex})}\Big),
\end{align}
where the constant $\dot{C}_+$ only depends on $\overline{U}_\pm$, $L$, $r_0$, $\dot{z}_*$ and $\alpha$, and $\dot{z}_*$ is given by Lemma \ref{lemma:approxz*}. Moreover
\begin{align}
&\partial_r^2 \dot{\theta}_+ (z,0) =0,\quad \partial_r (\dot{p}_+ , \dot{q}_+ , \dot{s}_+)(z,0) = 0,\label{aixs++}\\
 &\dot{w}_+ (z,r_0)=0, \quad \partial_r (\dot{p}_+ , \dot{q}_+, \dot{s}_+)(z,r_0)=0,\label{upper++}\\
&\dot{\varphi}'(0)= \dot{\varphi}'''(0) = \dot{\varphi}'(r_0)=0.\label{varphi''}
\end{align}

\subsection{Nonlinear iteration scheme}
In this subsection, we will design a nonlinear iteration scheme to determine the shock solution to the nonlinear boundary value problem \eqref{eqf1}-\eqref{eq11000}. Before doing that, an important property of the supersonic solution ahead of the shock front will be stated by the following lemma, which will be used to update the shock position.
\begin{lem}\label{U--dotU-}
Let $U_-\defs \overline{U}_- + \delta U_-$, then
\begin{align}
 &\|\delta U_- - \dot{U}_- \|_{C^2(\overline{\Omega})}
  \leq C_{L} \sigma^2,\label{eq838}
\end{align}
where the positive constant $C_{L}$ only depends on $\overline{U}_-$ and  $L$.
 \end{lem}
\begin{proof}
By applying the equations \eqref{CDE1}-\eqref{CDE5}, \eqref{CDEI-}-\eqref{CDEV-}, Theorem \ref{mainthmU-} and Lemma \ref{U-esti}, we have
\begin{align}
 & \| \delta U_- - \dot{U}_- \|_{C^2(\overline{\Omega})}\notag\\
 \leq&C\Big( \|\partial_z \delta U_- \|_{C^2(\overline{\Omega})}\cdot \|\delta U_- \|_{C^3(\overline{\Omega})} + \|\partial_r \delta U_- \|_{C^2(\overline{\Omega})}\cdot \|\delta U_- \|_{C^3(\overline{\Omega})} +\|\delta U_- \|_{C^3(\overline{\Omega})}^2\Big) \notag\\
 \leq& C_L \sigma^2( \|w_{en}\|_{H^7(\Gamma_{en})} + \|q_{en}\|_{H^7(\Gamma_{en})} )^2,
\end{align}
where the positive constant $C_{L}$ only depends on $\overline{U}_-$ and $L$.
\end{proof}

Now we can introduce the iteration scheme. Given a state $U=\overline{U}_+  + \delta U$ and the slope of the shock front $\varphi' = \delta\varphi'$, based on \eqref{thmsolvabilityeq} in Theorem \ref{BigThm}, we know the shock position $\varphi(r_0)={z}_*$
must satisfy the following identity
 \begin{align}\label{Rsolvability}
 \int_0^{r_0} \frac{1- \overline{M}_+^2}{\bar{\rho}_+\bar{q}_+^2}r \bar{p}_+^{\frac{1}{\gamma}} \Big(\sigma p_{ex} (r) - g_1({z}_*, r)\Big)\dif r = \int \int_{\dot{\Omega}_+} r \bar{p}_+^{\frac{1}{\gamma}}f_1(\delta U,\varphi)\dif z \dif r.
\end{align}
Take $K\defs \dot{z}_*$ in the transformation \eqref{Kfixed}, then denote the domain $\Omega_+^{\dot{z}_*}$ by $\dot{\Omega}_+$.
Define
\begin{equation*}
\mathcal{S}(\delta U, \delta\varphi')\defs
 \left\{
  \begin{aligned}
 &\delta \theta(z,0) = 0,\, \delta w(z,0) = 0,\, \partial_r ( \delta p, \delta q, \delta s) (z,0)=0,\, \partial_r^2\delta \theta(z,0) = 0, \\
 &\delta \theta(z,r_0) = 0,\,\delta w(z,r_0) = 0,\, \partial_r ( \delta p, \delta q, \delta s) (z,r_0) =0,\\
&\delta\varphi'(0)=\delta\varphi'''(0)= \delta\varphi'(r_0)= 0
 \end{aligned}
  \right\}
\end{equation*}
and
\begin{equation}
 \mathscr{N}\defs 
 \left\{
  \begin{aligned}
&( \delta U, \delta\varphi',{z}_*):\,( \delta U, \delta\varphi')\in\mathcal{S}(\delta U, \delta\varphi'),\,\mbox{\eqref{Rsolvability} holds},\,  |\dot{z}_*-{z}_*|< C \sigma,\\
&
 \|( r\delta w,\delta q,\delta s)\|_{C^{2,\alpha}(\overline{\dot{\Omega}_+})} +  \|\delta w\|_{C^{1,\alpha}(\overline{\dot{\Omega}_+})} + \| \delta\varphi'\|_{C^{2,\alpha}(\dot{\Gamma}_s)} < C\sigma,\\
&\|(\delta p-\dot{p}_+, \delta \theta-\dot{\theta}_+)\|_{C^{2,\alpha}(\overline{\dot{\Omega}_+})}
+ \|\delta w - \dot{w}_+\|_{C^{0,\alpha}(\overline{\dot{\Omega}_+})} +\| \delta \varphi' - \dot{\varphi}'  \|_{C^{1,\alpha}(\overline{\dot{\Omega}_+})}\notag\\
  & +\| (r\delta w - r\dot{w}_+, \delta q -\dot{q}_+, \delta s -\dot{s}_+)\|_{C^{1,\alpha}(\overline{\dot{\Omega}_+})} \leq \frac12 \sigma^{\frac32}
\end{aligned}
  \right\}
\end{equation}
where the positive constant $C$ depends only on $\overline{U}_{\pm}$, $p_{en}$, $w_{en}$, $p_{ex}$, $L$, $r_0$, $\dot{z}_*$ and $\alpha$, and will be given later.

For $(\delta U;  \delta \varphi',{z}_*)\in \mathscr{N}$ with  $\dot{\Omega}_+=\Omega_+^{{z}_*}$,
define the iteration mapping
\begin{align}\label{6.63x}
  \mathcal{T}: (\delta U; \delta \varphi',{z}_*) \mapsto (\delta {U}^*; \delta {\varphi^*}',\varphi^*(r_0)),
\end{align}
where $ (\delta {U}^*; \delta {\varphi^*}')$ is solved by problem \eqref{eqf1}-\eqref{eq11000} with \eqref{eqf22}-\eqref{f2itera}, and $\vec{\mathcal{F}}\defs(f_1, f_2, f_3, f_4, f_5)$ and $\vec{\mathcal{G}} \defs (g_1, g_2, g_3, g_4, g_5)$ are functions of $(\delta U; \delta \varphi')$ and $\varphi(r_0)={z}_*$. Then $\varphi^*(r_0)$ is solved by \eqref{Rsolvability} with $(\delta U; \delta \varphi',z_*)$ being replaced by $(\delta {U}^*; \delta {\varphi^*}',\varphi^*(r_0))$.
The mapping can be rewritten by an operator $\mathcal{L}$ as
\begin{align}\label{defineoperator}
  (\delta {U}^*; \delta {\varphi^*}'; \varphi^*(r_0)) = \mathcal{L}(H(z,r); \vec{\mathcal{F}}(\delta U, \varphi); \vec{\mathcal{G}}(\delta U, \delta U_-, \varphi); p_{ex}).
\end{align}
Obviously, \eqref{dotU+es=}-\eqref{varphi''} yield that $(\dot{U}_+, \dot{\varphi}',\dot{z}_*)\in \mathscr{N}$, so the function space $\mathscr{N}$ is non-empty.
We will prove that the mapping $ \mathcal{T}$ maps $\mathscr{N}$ into $\mathscr{N}$, and then the mapping is contractive, to finish the proof of Theorem \ref{mainthmU} by the Banach fixed point theorem.

First, we will prove that the iteration mapping $\mathcal{T}$ is well defined.
\begin{lem}\label{lemwelldefined}
There exists a positive constant $\sigma_1$ with $0<\sigma_1 \ll 1$, such that for any $0< \sigma \leq \sigma_1$, if $(\delta U;  \delta \varphi',{z}_*)\in \mathscr{N}$, then $(\delta {U}^*; \delta {\varphi^*}',\varphi^*(r_0))\in \mathscr{N}$.
 \end{lem}

\begin{proof} We divide the proof into two steps.

\smallskip
\emph{Step 1}. In the first step, we will solve $(\delta {U}^*; \delta {\varphi^*}')$.
It follows from  $(\delta {U}; \delta{\varphi}',{z}_*)\in \mathscr{N}$ that
\begin{align}
&\sum_{i=1}^2 \|f_i\|_{C^{1,\alpha}(\overline{\dot{\Omega}_+})} + \sum_{j=3}^5 \| f_j \|_{C^{2,\alpha}(\overline{\dot{\Omega}_+})} + \sum_{j=1}^{5}\| g_j\|_{C^{2,\alpha} (\dot{\Gamma}_s)} <C\sigma,\label{fifjin}\\
  &H(z,0)=H(z,r_0)=0, \quad \partial_r f_1(z,0) =0,\quad f_2(z,0) =f_2(z,r_0)= 0,\label{deltaUdeduce1}\\
  &f_3(z,0)=f_3(z,r_0) =0.\label{deltaUdeduce2}
\end{align}
So by Theorem \ref{BigThm} and Theorem \ref{wqsexistence},  there exists a unique solution $(\delta U^*, \delta {\varphi^*}')$ to the linearized problem \eqref{defineoperator} and
\begin{align}\label{mathcalS=*}
  (\delta {U}^*, \delta {\varphi^*}')\in \mathcal{S}(\delta U^*, \delta{\varphi^*}').
\end{align}

Moreover, it follows from Theorem \ref{BigThm} and Theorem \ref{wqsexistence} again that
\begin{align}\label{U*ES}
  &\|(\delta p^*,\delta \theta^* , \delta q^*, \delta s^*)\|_{C^{2,\alpha} (\overline{\dot{\Omega}_+})}+ \|r\delta w^*\|_{C^{2,\alpha} (\overline{\dot{\Omega}_+})} +  \|\delta w^*\|_{C^{1,\alpha} (\overline{\dot{\Omega}_+})} + \|\delta {\varphi^*}' \|_{C^{2,\alpha} (\dot{\Gamma}_s)}\notag \\
  \leq & C \Big(\sum_{i=1}^2 \|f_i\|_{C^{1,\alpha}(\overline{\dot{\Omega}_+})} + \sum_{j=3}^5 \| f_j \|_{C^{2,\alpha}(\overline{\dot{\Omega}_+})} + \sum_{j=1}^{5}\| g_j\|_{C^{2,\alpha} (\dot{\Gamma}_s)} + \sigma \|p_{ex}\|_{C^{2,\alpha}(\Gamma_{ex})}  \Big)\notag\\
 < & C\sigma.
\end{align}
Moreover, direct calculations yield that
\begin{align}
 & \|\delta p^* -\dot{p}_+\|_{C^{2,\alpha}(\overline{\dot{\Omega}_+})}
  + \|\delta \theta^* -\dot{\theta}_+\|_{C^{2,\alpha} (\overline{\dot{\Omega}_+})}\notag\\
  \leq &  C \Big(\| f_1 \|_{C^{1,\alpha}(\overline{\dot{\Omega}_+})}+ \| f_2 - \dot{f}_2\|_{C^{1,\alpha} (\overline{\dot{\Omega}_+})} + \sum_{j=1}^4\| g_j - \dot{g}_j\|_{C^{2,\alpha} (\dot{\Gamma}_s)}\Big)\label{esptheta1}
\end{align}
\begin{align}
  & \|r\delta w^* -r\dot{w}_+\|_{C^{1,\alpha} (\overline{\dot{\Omega}_+})} + \|\delta q^* -\dot{q}_+\|_{C^{1,\alpha} (\overline{\dot{\Omega}_+})}+ \|\delta s^* -\dot{s}_+\|_{C^{1,\alpha} (\overline{\dot{\Omega}_+})}\notag\\
\leq & C \Big(\sum_{j=3}^5\| f_j \|_{C^{2,\alpha}(\overline{\dot{\Omega}_+})} + \sum_{j=1}^4\| g_j - \dot{g}_j\|_{C^{2,\alpha} (\dot{\Gamma}_s)} +\|\delta \theta^* -\dot{\theta}_+\|_{C^{2,\alpha} (\overline{\dot{\Omega}_+})}\notag\\
& \quad + \|\delta \theta\|_{C^{2,\alpha}(\overline{\dot{\Omega}_+})}\big(\|\partial_r(r w^*)\|_{C^{1,\alpha}(\overline{\dot{\Omega}_+})} + \|\partial_r q^*\|_{C^{1,\alpha}(\overline{\dot{\Omega}_+})} + \|\partial_r s^*\|_{C^{1,\alpha}(\overline{\dot{\Omega}_+})}     \big)\Big)\notag\\
< & C \Big(\sum_{j=3}^5\| f_j \|_{C^{2,\alpha}(\overline{\dot{\Omega}_+})} + \sum_{j=1}^4\| g_j - \dot{g}_j\|_{C^{2,\alpha} (\dot{\Gamma}_s)} +\|\delta \theta^* -\dot{\theta}_+\|_{C^{2,\alpha} (\overline{\dot{\Omega}_+})}\Big)+C\sigma^2
\end{align}
and
\begin{align}
\|\delta {\varphi^*}' -\dot{\varphi}' \|_{C^{1,\alpha} (\dot{\Gamma}_s)}\leq \| g_5 - \dot{g}_5\|_{C^{1,\alpha} (\dot{\Gamma}_s)} 
\label{esvar4}
\end{align}
where the positive constant $C$ only depends on $\overline{U}_{\pm}$, $w_{en}$, $q_{en}$, $L$ and $\alpha$.

 By \eqref{deff1}, it is easy to see that
 \begin{align}\label{f1es}
&\|f_1(\delta U, \delta \varphi', \delta z_{*})\|_{C^{1,\alpha}(\overline{\dot{\Omega}_+})}\notag\\
\leq& C\Big(\big(\|\partial_z \delta p \|_{C^{1,\alpha}(\overline{\dot{\Omega}_+})} + \|\partial_r \delta p \|_{C^{1,\alpha}(\overline{\dot{\Omega}_+})} \big)  \|\delta U \|_{C^{2,\alpha}(\overline{\dot{\Omega}_+})}  + \|\partial_r \delta \theta \|_{C^{1,\alpha}(\overline{\dot{\Omega}_+})} \Big)\notag\\
& + C\Big(\|\delta \theta \|_{C^{2,\alpha}(\overline{\dot{\Omega}_+})}\|(\delta p, \delta q ,\delta s)\|_{C^{2,\alpha}(\overline{\dot{\Omega}_+})} + \|\delta \theta \|_{C^{2,\alpha}(\overline{\dot{\Omega}_+})}^2 \Big)\notag\\
& + C  \Big(\|(\partial_z\delta p, \partial_z \delta \theta) \|_{C^{1,\alpha}(\overline{\dot{\Omega}_+})} \|\delta \varphi'\|_{C^{2,\alpha}(\dot{\Gamma}_s)} +  |\varphi(r_0) - \dot{z}_*| \|\partial_z\delta p \|_{C^{1,\alpha}(\overline{\dot{\Omega}_+})} \Big)\notag\\
 <& C \sigma^2.
 \end{align}
For $j=3,4,5$, we have
\begin{align}
  &\|f_j(\delta U, \delta \varphi', \delta z_{*})\|_{C^{2,\alpha} (\dot{\Omega}_+)}\leq  C \big\|H(z,r)-\delta \theta\big\|_{C^{2,\alpha} (\overline{\dot{\Omega}_+})}\notag\\
  =& \Big\|\frac{-\delta z_* + \int_r^{r_0} \delta \varphi' (\tau)\dif \tau - (z-L) \delta \varphi'\tan\delta \theta }{L - \dot{z}_* + (z - L)\delta \varphi'\tan\delta\theta}\tan\delta \theta+ (\tan\delta\theta -\delta \theta)\Big\|_{C^{2,\alpha} (\overline{\dot{\Omega}_+})}\notag\\
  \leq &C \Big( \big(  |\delta z_*| + \|\delta\varphi'\|_{L^\infty(\dot{\Gamma}_s)}  \big) \|\delta \theta\|_{C^{2,\alpha} (\overline{\dot{\Omega}_+})}  + \big( \|\delta \theta\|_{C^{2,\alpha} (\overline{\dot{\Omega}_+})} \big)^2\Big)\notag\\
< &  C \sigma^2.\label{6.73x}
\end{align}

Next, let us consider the estimates of $\| g_j - \dot{g}_j\|_{C^{2,\alpha} (\dot{\Gamma}_s)} $, $(j= 1, \cdots, 4)$. 
By \eqref{5.35xw} and \eqref{RHinitial}, we have
 \begin{align}\label{gj-dotgj=}
   g_i - \dot{g}_i = \sum_{j=1}^{4} a_{ij} (G_j^{\sharp} -\dot{G}_j) \qquad \text{for}\quad i=1,\cdots, 4.
 \end{align}
Thus, it suffices to estimate $(G_j^{\sharp} -\dot{G}_j)$. By \eqref{5.29xw} and \eqref{initialRH}, we have  for $j=1,2,3,4$
\begin{align}\label{eq0104}
&G_j^{\sharp} -  \dot{G}_j
\notag\\
=& \alpha_{j}^+\cdot \delta U - G_j \big(U, U_-(\varphi' , \varphi(r_0))\big)-{\mathbf{\alpha}}_j^- \cdot {\dot{U}}_-
  \notag\\
   =& \Big(\alpha_j^+\cdot \delta U  - \alpha_j^-\cdot \delta U_-\big( \varphi(r_0) - \int_{r}^{r_0}\delta \varphi'(\tau)\dif \tau,r \big) - G_j \big(U, U_-(\varphi(r_0) - \int_{r}^{r_0}\delta \varphi'(\tau)\dif \tau,r) \big)\Big)
   \notag\\
   &+\alpha_j^-\cdot \Big(\delta U_-\big(\varphi(r_0) - \int_{r}^{r_0}\delta \varphi'(\tau)\dif \tau,r\big) - \dot{U}_-(\varphi(r_0),r)\Big)\notag\\
   & + \alpha_j^- \cdot \dot{U}_-(\varphi(r_0),r) - {\mathbf{\alpha}}_j^- \cdot {\dot{U}}_-(\dot{z}_*, r)\notag\\
   =& \frac{1}{2} (\delta U, \delta U_-) \int_{0}^{1} D^2 G_j (\overline{U}_+ + t \delta U; \overline{U}_- + t\delta U_-)\dif t (\delta U, \delta U_-)^{\top}\notag\\
   &+ \alpha_j^-\cdot \Big\{ \Big(\delta U_-\big(\varphi(r_0) - \int_{r}^{r_0}\delta \varphi'(\tau)\dif \tau,r\big)- \dot{U}_-\big(\varphi(r_0)- \int_{r}^{r_0}\delta \varphi'(\tau)\dif \tau,r \big)\Big)\notag\\
    &\qquad\qquad  +\Big(\dot{U}_-(\varphi(r_0) - \int_{r}^{r_0}\delta \varphi'(\tau)\dif \tau
     ,r) - \dot{U}_-(\varphi(r_0),r)\Big) \Big\}\notag\\
     & + \alpha_j^- \cdot \big( \dot{U}_-(\varphi(r_0),r) - {\dot{U}}_-(\dot{z}_*, r) \big).
\end{align}
So it follows from \eqref{eq837}, \eqref{alpha1}-\eqref{alpha5}, \eqref{alpha-1}-\eqref{alpha-5}, \eqref{dotU-estimate}, \eqref{eq838} and the fact that $(\delta U ;  \delta \varphi',z_*)\in  \mathscr{N}$ that  for $j=1,2,3,4$
\begin{align}
\| g_j - \dot{g}_j\|_{C^{2,\alpha} (\dot{\Gamma}_s)}
\leq & C \Big(\|\partial_z\dot{U}_-\|_{C^{1,\alpha}(\overline{\Omega})}\big(|\varphi(r_0) - \dot{z}_*| + \|\delta \varphi'\|_{C^{2,\alpha}(\dot{\Gamma}_s)} \big) + \| \delta U_- - \dot{U}_- \|_{C^{1,\alpha}(\overline{\Omega})}\Big)\notag\\
& + C\big( \|\delta U\|_{C^{2,\alpha}(\overline{\dot{\Omega}_+})}^2  + \|\delta U_-\|_{C^{2,\alpha}(\overline{\Omega})}^2\big)\notag\\
< & C\sigma^2,\label{eq0105}
\end{align}
where the positive constant $C$ only depends on $\overline{U}_{\pm}$, $L$ and $r_0$.

Then let us consider the estimate of $\| g_5 - \dot{g}_5\|_{C^{2,\alpha} (\dot{\Gamma}_s)}$. By \eqref{varphi} and \eqref{dotvar}, we have
\begin{align*}
&g_5 - \dot{g}_5\notag\\
 =& \frac{ \bar{\rho}_+\bar{q}_+^2 \delta \theta^* - \big( \bar{\rho}_+\bar{q}_+^2  \delta \theta - [\bar{p}]\delta \varphi' - G_5(U, U_-(\varphi', \varphi(r_0));\varphi')  \big)}{[\bar{p}]} -  \frac{\bar{\rho}_+\bar{q}_+^2 \dot{\theta}_+ -\bar{\rho}_-\bar{q}_-^2 \dot{\theta}_-}{[\bar{p}]}\notag\\
 =& \frac{ \bar{\rho}_+\bar{q}_+^2 (\delta \theta^* - \dot{\theta}_+)}{[\bar{p}]} + \frac{([\bar{p}] - [p]) - [\rho q^2 \sin^2\theta]}{[\bar{p}]}\delta \varphi' + \frac{\rho q^2 \cos\theta \sin\theta - \bar{\rho}_+\bar{q}_+^2  \delta \theta}{[\bar{p}]}\notag\\
 &-\frac{(\rho_- q_-^2 \cos\theta_- \sin\theta_- -  \bar{\rho}_-\bar{q}_-^2 \dot{\theta}_-)(\varphi(r_0), r)+ \bar{\rho}_-\bar{q}_-^2 \big(\dot{\theta}_- (\varphi(r_0),r) - \dot{\theta}_- (\dot{z}_*,r)\big)}{[\bar{p}]}.
\end{align*}
So 
\begin{align}
& \| g_5 - \dot{g}_5\|_{C^{2,\alpha} (\dot{\Gamma}_s)}\notag\\
\leq & C \Big(\|\delta \theta^* -\dot{\theta}_+\|_{C^{2,\alpha} (\overline{\dot{\Omega}_+})}  +\big(\|\delta p\|_{C^{2,\alpha} (\overline{\dot{\Omega}_+})} + \|\delta \theta\|_{C^{2,\alpha} (\overline{\dot{\Omega}_+})}^2   \big)  \|\delta \varphi'\|_{C^{2,\alpha}(\dot{\Gamma}_s)}\Big)\notag\\
&+ C \Big(  \|\delta \theta\|_{C^{2,\alpha} (\overline{\dot{\Omega}_+})} \|(\delta p, \delta q,  \delta s)\|_{C^{2,\alpha} (\overline{\dot{\Omega}_+})} + \|\delta \theta\|_{C^{2,\alpha} (\overline{\dot{\Omega}_+})}^2 + \| \delta \theta_- - \dot{\theta}_- \|_{C^{1,\alpha}(\overline{\Omega})}  \Big)\notag\\
& + C \Big(  \|\dot{\theta}_-\|_{C^{2,\alpha} (\overline{\Omega})} \|(\delta p_-, \delta q_-,  \delta s_-)\|_{C^{2,\alpha} (\overline{\Omega})} + \|\partial_z\dot{\theta}_-\|_{C^{1,\alpha}(\overline{\Omega})}|\varphi(r_0) - \dot{z}_*| \Big)\notag\\
< & C \Big(\|\delta \theta^* -\dot{\theta}_+\|_{C^{2,\alpha} (\overline{\dot{\Omega}_+})} + \sigma^2\Big),\label{6.77x}
\end{align}
where the positive constant $C$ only depends on $\overline{U}_{\pm}$, $L$ and $r_0$.

Finally, let us consider the estimate of $\|f_2 - \dot{f}_2\|_{C^{1,\alpha}(\overline{\dot{\Omega}_+})} $.
By \eqref{itemf2}, \eqref{f2itera}, \eqref{Letfjf2=} and \eqref{RHinitial}, we have
\begin{align}
&f_2 - \dot{f}_2\notag\\
=& {f}_2^{\sharp}(\delta U, \varphi)
 + \frac{2\bar{\rho}_+ \bar{w}_+}{r^2}\Big\{ R(\dot{z}_*;z,r)g_2(\dot{z}_*, R(\dot{z}_*;z,r)) -  r\dot{g}_2(\dot{z}_*, r)
+ \int_{\dot{z}_*}^{z} f_3(\tau, R(\tau;z,{r}) )\dif \tau\notag\\
& -  \int_{\dot{z}_*}^{z} \Big( R(\tau;z,{r})\partial_r \bar{w}_+(R(\tau;z,{r})) + \bar{w}_+(R(\tau;z,{r})) \Big)\delta \theta (\tau, R(\tau;z,{r}) )  \dif \tau \notag\\
&+\Big(r \partial_r \bar{w}_+(r) + \bar{w}_+(r)\Big)\int_{\dot{z}_*}^{z}\delta \theta (\tau, r ) \dif \tau  \Big\}\notag\\
&-\frac{1}{\gamma c_v}\frac{\bar{\rho}_+ \bar{w}_+^2}{r}
\Big\{ g_4(\dot{z}_*, R(\dot{z}_*;z,r) ) - \dot{g}_4 (\dot{z}_*, r) + \int_{\dot{z}_*}^{z} f_5(\tau, R(\tau;z,{r}) )\dif \tau\notag\\
&- \int_{\dot{z}_*}^{z} \partial_r\bar{s}_+( R(\tau;z,{r}) ) \delta \theta(\tau, R(\tau;z,{r}) )
  -\partial_r\bar{s}_+(r )\delta \theta(\tau, r)\dif \tau  \Big\}.\label{6.78x}
\end{align}

Because $(\delta U ;  \delta \varphi',z_*)\in  \mathscr{N}$, it is easy to see that
\begin{align}\label{f2sharpsigma2}
  \|f_2^{\sharp}\|_{C^{1,\alpha}(\overline{\dot{\Omega}_+})} < C \sigma^2.
\end{align}

Next, let us consider
\begin{align}\label{Rg-rdotg=}
  &R(\dot{z}_*;z,r) g_2(\dot{z}_*, R(\dot{z}_*;z,r)) -  r \dot{g}_2(\dot{z}_*, r)\notag\\
 =& g_2(\dot{z}_*, R(\dot{z}_*;z,r))\big(R(\dot{z}_*;z,r) - R(z;z,r)    \big)\notag\\
 & + r \Big( {g}_2 (\dot{z}_*, R(\dot{z}_*;z,r))   - \dot{g}_2 (\dot{z}_*, R(z;z,r))\Big) \notag\\
 =& g_2(\dot{z}_*, R(\dot{z}_*;z,r))\big(R(\dot{z}_*;z,r) - R(z;z,r)    \big)\notag\\
 & + r \Big( \dot{g}_2 (\dot{z}_*, R(\dot{z}_*;z,r))   - \dot{g}_2 (\dot{z}_*, R(z;z,r))\Big)\notag\\
 & + r \Big( {g}_2 (\dot{z}_*, R(\dot{z}_*;z,r))   - \dot{g}_2 (\dot{z}_*, R(\dot{z}_*;z,r))\Big) \notag\\
=& \Big(r \int_0^1 \dot{g}_2' \big(\dot{z}_*, \tau R(\dot{z}_*;z,r)+(1-\tau) R(z;z,r)   \big)\dif \tau  -  g_2(\dot{z}_*, R(\dot{z}_*;z,r))    \Big)\notag\\
& \times \int_{\dot{z}_*}^z \frac{(L- \varphi)\tan\delta \theta}{L - \dot{z}_* + (z - L)\varphi'\tan\delta \theta}(\tau, R(\tau;z,r))\dif \tau\notag\\
& + r \Big( {g}_2 (\dot{z}_*, R(\dot{z}_*;z,r))   - \dot{g}_2 (\dot{z}_*, R(\dot{z}_*;z,r))\Big).
\end{align}
So it follows from $(\delta U ;  \delta \varphi',z_*)\in  \mathscr{N}$, \eqref{dotw} and \eqref{eq0105} that
\begin{align}
&\| R(\dot{z}_*;z,r) g_2(\dot{z}_*, R(\dot{z}_*;z,r)) -  r \dot{g}_2(\dot{z}_*, r)\|_{C^{1,\alpha} (\dot{\Gamma}_s)}\notag\\
 \leq & C\big(  \sigma \|\partial_r w_{en}\|_{H^6([0, r_0])} + \| \dot{U}_-\|_{C^{2,\alpha}(\overline{\Omega})} + \| \delta{U}_-\|_{C^{2,\alpha}(\overline{\Omega})}+ \| \delta{U}\|_{C^{2,\alpha}(\overline{\dot{\Omega}_+})} \big) \|\dot{\theta}\|_{C^{2,\alpha}(\dot{\Gamma}_s)}\notag\\
 & + \| g_2- \dot{g}_2\|_{C^{2,\alpha} (\dot{\Gamma}_s)}\notag\\
 < & C\sigma^2.
\end{align}
Similarly, we can show that
\begin{align}
 \| g_4(\dot{z}_*, R(\dot{z}_*;z,r)) -  \dot{g}_4(\dot{z}_*, r)\|_{C^{1,\alpha} (\dot{\Gamma}_s)} <  C\sigma^2.
\end{align}

Then let us consider
\begin{align}\label{Rrwtheta}
   & \Big( R(\tau;z,r)\partial_r \bar{w}_+(R(\tau;z,r)) + \bar{w}_+(R(\tau;z,r))\Big)\delta \theta (\tau, R(\tau;z,r) ) \notag\\
  & -\Big(r \partial_r \bar{w}_+(r) + \bar{w}_+(r)\Big)\delta \theta (\tau, r )\notag\\
   =& \Big( r \partial_r \bar{w}_+(r) + \bar{w}_+(r)\Big)\Big( \delta \theta (\tau, R(\tau;z,r) ) - \delta \theta (\tau, r )\Big)\notag\\
   &+ \Big(R(\tau;z,r) \partial_r \bar{w}_+(R(\tau;z,r)) + \bar{w}_+(R(\tau;z,r))- r\partial_r \bar{w}_+(r) -\bar{w}_+(r)\Big)
\delta \theta (\tau, R(\tau;z,r) ).
\end{align}
Notice that
\begin{align}
  \| \delta \theta (\tau, R(\tau;z,r) ) - \delta \theta (\tau, r )\|_{C^{2,\alpha} (\overline{\dot{\Omega}_+})} 
 \leq & C \|\partial_r \theta  \|_{C^{1,\alpha} (\overline{\dot{\Omega}_+})} |R(\tau;z,r) - R(z;z,r)|\notag\\
 \leq &  C \|\partial_r \theta  \|_{C^{1,\alpha} (\overline{\dot{\Omega}_+})}\|\theta  \|_{C^{2,\alpha} (\overline{\dot{\Omega}_+})} \notag\\
< & C \sigma^2
\end{align}
and
\begin{align}
&\|R(\tau;z,r)\partial_r \bar{w}_+(R(\tau;z,r)) + \bar{w}_+(R(\tau;z,r))- r\partial_r \bar{w}_+(r) -\bar{w}_+(r)\|_{L^\infty(\overline{\dot{\Omega}_+})} \notag\\
\leq& C \Big( \|\partial_r^2 \bar{w}_+\|_{L^\infty(0,1)} + \|\partial_r \bar{w}_+\|_{L^\infty([0, r_0])} \Big) |R(\tau;z,r) - R(z;z,r)|\notag\\
\leq &C \|\theta  \|_{C^{2,\alpha} (\overline{\dot{\Omega}_+})} \notag\\
 < &  C \sigma.\label{termf2es}
\end{align}
Thus, \eqref{Rrwtheta}-\eqref{termf2es} yield that
\begin{align}
   & \Big\|\Big( R(\tau;z,r)\partial_r \bar{w}_+(R(\tau;z,r)) + \bar{w}_+(R(\tau;z,r))\Big)\delta \theta (\tau, R(\tau;z,r) ) \notag\\
   &\quad- \Big( r \partial_r \bar{w}_+(r) +\bar{w}_+(r)\Big)\delta \theta (\tau, r )\Big\|_{C^{1,\alpha} (\overline{\dot{\Omega}_+})}\notag\\
  < & C \sigma^2.
\end{align}
Similarly, we have
\begin{align}\label{sR-dot}
  \| \partial_r\bar{s}_+( R(\tau;z,r) ) \delta \theta(\tau, R(\tau;z,r) )  - \partial_r\bar{s}_+(r ) \delta \theta(\tau, r)\|_{C^{1,\alpha} (\dot{\Omega}_+)}
   <  C \sigma^2.
\end{align}
Therefore, it follows from \eqref{6.73x}, \eqref{6.78x}-\eqref{sR-dot} that
\begin{align}\label{f2es}
  \|f_2 - \dot{f}_2 \|_{C^{1,\alpha}(\dot{\Omega}_+)} <  C \sigma^2.
\end{align}

Hence, it follows from \eqref{esptheta1}-\eqref{6.73x}, \eqref{eq0105}-\eqref{6.77x} and \eqref{f2es} that
\begin{align}\label{U-dotUsigma32}
 & \|\delta p^* -\dot{p}_+\|_{C^{2,\alpha}(\overline{\dot{\Omega}_+})}
  + \|\delta \theta^* -\dot{\theta}_+\|_{C^{2,\alpha} (\overline{\dot{\Omega}_+})} + \|r\delta w^* -r\dot{w}_+\|_{C^{1,\alpha} (\overline{\dot{\Omega}_+})}\notag\\
   &+ \|\delta q^* -\dot{q}_+\|_{C^{1,\alpha} (\overline{\dot{\Omega}_+})}+ \|\delta s^* -\dot{s}_+\|_{C^{1,\alpha} (\overline{\dot{\Omega}_+})}+
\|\delta {\varphi^*}' -\dot{\varphi}' \|_{C^{1,\alpha} (\dot{\Gamma}_s)} \notag\\
\leq& C \sigma^2 < \frac12 \sigma^{\frac32},
\end{align}
for sufficiently small $\sigma>0$.

\smallskip
\emph{Step 2}. In this step, we will show that there exists a constant $\varphi^*(r_0)$ to the equation \eqref{Rsolvability} for $(\delta {U}^*; \delta {\varphi^*}')$ with the following estimate:
\begin{equation}\label{d}
  |\varphi^*(r_0) -\dot{z}|\leq C \sigma,
\end{equation}
where the constant $C$ only depends on $\overline{U}_{\pm}$, $\dot{z}_*$, $L$, $r_0$ and $\alpha$.

Let $\delta z_*:=\varphi^*(r_0) -\dot{z}_*$. Define
\begin{align}\label{I*}
{I}^*(\delta z_*;\delta {U}^*,\delta {\varphi^*}', \delta U_-)
\defs &
 \int_0^{r_0} \frac{1-\bar{M}_+^2}{\bar{\rho}_+\bar{q}_+^2} r  \bar{p}_+^{\frac{1}{\gamma}} \Big(\sigma p_{ex} (r) - g_1(\dot{z}_*+\delta z_*, r)\Big)\dif r\notag\\
&-\int \int_{\dot{\Omega}_+} r  \bar{p}_+^{\frac{1}{\gamma}} f_1(\delta U^*, \delta {\varphi^*}',\dot{z}_*+\delta z_*)\dif z \dif r.
\end{align}
First, it easily follows from Lemma \ref{lemma:approxz*} that
\begin{equation}\label{6.92x}
{I}^*(0; \delta \dot{U}_+, \dot{\varphi}', \dot{U}_-) = 0.
\end{equation}
By \eqref{gj-dotgj=}-\eqref{eq0105}, we have
\begin{align}\label{g1z*dotg1}
  g_1(\dot{z}_*+\delta z_*, r) =& \sum_{j=1}^{4} a_{1j} G_j^{\sharp}(\dot{z}_*+\delta z_*, r)  \notag\\
  =& \sum_{j=1}^{4} a_{1j}\dot{G}_j(\dot{z}_*+\delta z_*, r) + \sum_{j=1}^{4} a_{1j} \big(G_j^{\sharp}(\dot{z}_*+\delta z_*, r) - \dot{G}_j(\dot{z}_*+\delta z_*, r) \big)\notag\\
  =& \dot{g}_1(\dot{z}_*+\delta z_*, r)+  \sum_{j=1}^{4} a_{1j} \big(G_j^{\sharp}(\dot{z}_*+\delta z_*, r) - \dot{G}_j(\dot{z}_*+\delta z_*, r) \big),
\end{align}
where
\begin{align}
&\|G_j^{\sharp}(\dot{z}_*+\delta z_*, r) - \dot{G}_j(\dot{z}_*+\delta z_*, r) \|_{C^{2,\alpha} ([0,r_0])}\notag\\
\leq & C \Big(\|\partial_z\dot{U}_-\|_{C^{1,\alpha}(\overline{\Omega})}\|\delta {\varphi^*}'\|_{C^{2,\alpha}(\dot{\Gamma}_s)} + \| \delta U_- - \dot{U}_- \|_{C^{1,\alpha}(\overline{\Omega})}\Big)\notag\\
& + C\big( \|\delta U^*\|_{C^{2,\alpha}(\overline{\dot{\Omega}_+})}^2  + \|\delta U_-\|_{C^{2,\alpha}(\overline{\Omega})}^2\big)\notag\\
< & C\sigma^2,
\end{align}
where the positive constant $C$ only depends on $\overline{U}_{\pm}$, $L$ and $r_0$.


Next, for $f_1$ (see its definition in \eqref{deff1} with $K=\dot{z}_*$), notice that
\begin{align}\label{termoff1}
  &\frac{\delta z_*}{L-\varphi} \frac{1-M^2 \cos^2 \theta}{\rho q^2} \partial_z \delta p \notag\\
  =&\frac{\delta z_*}{L-\varphi} \Big(\frac{1-M^2 \cos^2 \theta}{\rho q^2} - \frac{1-\overline{M}_+^2}{\bar{\rho}_+ \bar{q}_+^2} \Big) \partial_z \delta p + \frac{\delta z_*}{L-\varphi}\frac{1-\overline{M}_+^2}{\bar{\rho}_+ \bar{q}_+^2}\partial_z(\delta p - \dot{p}_+)\notag\\
  & -\frac{\delta z_*\int_r^{r_0} \delta \varphi'(\tau)\dif \tau}{(L-\varphi)(L-\varphi(r_0))}\frac{1-\overline{M}_+^2}{\bar{\rho}_+ \bar{q}_+^2}\partial_z \dot{p}_+ + \frac{\delta z_*}{L-\varphi(r_0)}\frac{1-\overline{M}_+^2}{\bar{\rho}_+ \bar{q}_+^2}\partial_z \dot{p}_+
\end{align}
and
\begin{align}\label{intint=0}
  \int \int_{\dot{\Omega}_+}r  \bar{p}_+^{\frac{1}{\gamma}} \frac{1-\overline{M}_+^2}{\bar{\rho}_+ \bar{q}_+^2}\partial_z \dot{p}_+\dif z \dif r =   \int \int_{\dot{\Omega}_+} \partial_r \big(  r\bar{p}_+^{\frac{1}{\gamma}} \dot{\theta}_+\big) \dif z \dif r=0,
\end{align}
where \eqref{intint=0} comes from \eqref{eqf1} with $f_1=0$ and boundary conditions \eqref{GAMMA4theta} and \eqref{GAMMA2}. Here we drop $~^*$ in \eqref{termoff1} for the notational simplicity. By \eqref{termoff1} and \eqref{intint=0} with estimates \eqref{mathcalS=*}-\eqref{U*ES} and \eqref{U-dotUsigma32}, we have that  for sufficiently small $\delta z_*$,  
\begin{align}\label{f2r}
 &\Big\|\int \int_{\dot{\Omega}_+} r \bar{p}_+^{\frac{1}{\gamma}} f_1(\delta U^*, \delta {\varphi^*}', \delta z_{*})\dif z \dif r \Big\|_{C^{1,\alpha}(\overline{\dot{\Omega}_+})}\notag\\
 \leq & C\Big(\big(\|\partial_z \delta p^* \|_{C^{1,\alpha}(\overline{\dot{\Omega}_+})} + \|\partial_r \delta p^* \|_{C^{1,\alpha}(\overline{\dot{\Omega}_+})} \big)  \|\delta U^* \|_{C^{2,\alpha}(\overline{\dot{\Omega}_+})}  + \|\partial_r \delta \theta^* \|_{C^{1,\alpha}(\overline{\dot{\Omega}_+})} \Big)\notag\\
& + C\Big(\|\delta \theta^* \|_{C^{2,\alpha}(\overline{\dot{\Omega}_+})}\|(\delta p^*, \delta q^* ,\delta s^*)\|_{C^{2,\alpha}(\overline{\dot{\Omega}_+})} + \|\delta \theta^* \|_{C^{2,\alpha}(\overline{\dot{\Omega}_+})}^2 \Big)\notag\\
& + C \Big( \|(\partial_z\delta p^*, \partial_z \delta \theta^*) \|_{C^{1,\alpha}(\overline{\dot{\Omega}_+})} \|\delta {\varphi^*}'\|_{C^{2,\alpha}(\dot{\Gamma}_s)} + \|\delta {\varphi^*}'\|_{C^{2,\alpha}(\dot{\Gamma}_s)}  \|\partial_z\dot{p}_+ \|_{C^{1,\alpha}(\overline{\dot{\Omega}_+})}\Big)\notag\\
&+ C \|\partial_z(\delta p^* - \dot{p}_+)\|_{C^{2,\alpha}(\overline{\dot{\Omega}_+})} | \delta z_*|\notag\\
< & C\big(\sigma^{\frac32}\cdot |\delta z_*| + \sigma^2\big),
\end{align}
where the positive constant $C$ only depends on $\overline{U}_{\pm}$, $L$ and $r_0$.

Applying  Lemmas \ref{lemma:approxz*} and \ref{U--dotU-} and estimtes \eqref{6.92x}-\eqref{f2r}, \eqref{I*} yields that
\begin{align}\label{rRI*}
&{I}^*(\delta z_*; \delta {U}^*,\delta {\varphi^*}', \delta U_-)\notag\\
=&{I}^*(\delta z_*; \delta {U}^*,\delta {\varphi^*}', \delta U_-)-{I}^*(0; \delta \dot{U}_+, \dot{\varphi}', \dot{U}_-)\notag\\
=&\int_0^{r_0} \frac{1-\overline{M}_+^2}{\bar{\rho}_+\bar{q}_+^2} r \bar{p}_+^{\frac{1}{\gamma}}\Big( \dot{g}_1 (\dot{z}_*, r) - \dot{g}_1 (\dot{z}_* +\delta z_*, r)\Big)\dif r + O(1)\sigma^{\frac32}\cdot \delta z_* + O(1) \sigma^2\notag\\
= &\int_0^{r_0} \frac{\overline{M}_+^2 -1}{\bar{\rho}_+\bar{q}_+^2} r \bar{p}_+^{\frac{1}{\gamma}}\partial_z \dot{p}_+ (\dot{z}_*, r)\dif r  \delta z_* + O(1)\sigma  (\delta z_*)^2 + O(1)\sigma^{\frac32}\cdot \delta z_* + O(1) \sigma^2\notag\\
=& I_1'(\dot{z}_*) \delta z_* + O(1)\sigma  (\delta z_*)^2 + O(1)\sigma^{\frac32}\cdot \delta z_* + O(1) \sigma^2 ,
\end{align}
where $O(1)$ are bounded functions in $C^{1,\alpha}$, which only depend on $\overline{U}_{\pm}$, $L$ and $r_0$.
Then \eqref{rRI*} implies that, as long as $\sigma$ is small enough, we have
\begin{align}
  \frac{\partial {I}^*}{\partial{\delta z_*}}(0; 0, 0, \dot{ U}_-) = I_1'(\dot{z}_*) + O(1)\sigma^{\frac32}
  \neq 0.
\end{align}
By applying the implicit function theorem, for sufficiently small $\sigma>0$, there exists a solution $\delta z_{*}$ satisfying ${I}^*(\delta z_*;\delta {U}^*,\delta {\varphi^*}', \delta U_-)=0$, that is, $\delta z_{*}$ satisfies equation \eqref{Rsolvability}, and estimate
\begin{align}
   |\delta z_*| \leq \Big|\frac{O(1) \sigma^2}{ I_1'(z_*)} \Big| \leq C\sigma,
 \end{align}
where constant $C$ only depends on $\overline{U}_{\pm}$, $w_{en}$, $q_{en}$, $L$, $r_0$ and $\gamma$. So \eqref{d} follows.

Therefore, \eqref{mathcalS=*}-\eqref{U*ES} and \eqref{U-dotUsigma32}-\eqref{d} yield that $(\delta {U}^*; \delta {\varphi^*}',\varphi^*(r_0))\in \mathscr{N}$.
\end{proof}

Finally, Theorem \ref{mainthmU} will be proven if we show that the mapping $\mathcal{T}$ is contractive in $\mathscr{N}$, which will be done in the following lemma.

\begin{lem}\label{lemcontrac}
    There exists a positive constant $\sigma_3$ with $0<\sigma_3 \ll 1$, such that for any $0< \sigma \leq \sigma_3$, the mapping $\mathcal{T}$ defined in \eqref{6.63x} is contractive.
  \end{lem}
\begin{proof}

Suppose that $(\delta U_k, \delta \varphi_k^{'}, z_k)\in \mathscr{N}$ for $k=1,2$. By Lemma \ref{lemwelldefined}, there exist $(\delta {U}_k^{*}, \delta \varphi_k^{*'}, z_{*k})\in \mathscr{N}$ with estimate \eqref{d}.
Finally, by a similar argument as done in Lemma \ref{lemwelldefined} to establish estimate \eqref{U-dotUsigma32} via replacing $(\delta U^*, \delta {\varphi^{*}}',\varphi^*(r_0))$ and $(\delta \dot{U}^+, \delta \dot{\varphi}^{'},\dot{z})$ by $(\delta {U}_k^{*}, \delta \varphi_k^{*'}, z_{*k})$ for $k=1,2$, we can prove for sufficiently small $\sigma$, it holds that
  \begin{align}\label{ee}
     &\| (\delta p_2^* - \delta p_1^*,\delta \theta_2^* - \delta \theta_1^*)\|_{C^{2,\alpha}(\overline{\dot{\Omega}_+})} + \| ( \delta q_2^* - \delta q_1^*,\delta s_2^* - \delta s_1^*,r\delta w_2^* - r\delta w_1^*)\|_{C^{1,\alpha}(\overline{\dot{\Omega}_+})} \notag\\
     & + \| \delta w_2^* - \delta w_1^*\|_{C^{0,\alpha}(\overline{\dot{\Omega}_+})} + \| \delta \varphi_{2}^{*'} - \delta \varphi_{1}^{*'}\|_{C^{1,\alpha}(\dot{\Gamma}_s)}
      \notag \\
     \leq& C\sigma\Big(\| (\delta p_2 - \delta p_1,\delta \theta_2 - \delta \theta_1)\|_{C^{2,\alpha}(\overline{\dot{\Omega}_+})}+\| (\delta q_2 - \delta q_1,\delta s_2- \delta s_1, r\delta w_2 - r\delta w_1)\|_{C^{1,\alpha}(\overline{\dot{\Omega}_+})} \notag\\
     & \qquad + \| \delta w_2 - \delta w_1\|_{C^{0,\alpha}(\overline{\dot{\Omega}_+})} + \| \delta {\varphi_2}' - \delta {\varphi_1}'\|_{C^{1,\alpha}(\dot{\Gamma}_s)}+|z_1- z_2|\Big).
   \end{align}
Because the proof is long and similar, we omit the details for the shortness. To conclude the proof of this lemma, we need to estimate $|z_{*2} - z_{*1}|$. Let $\delta z_{*k}=z_{*k}-\dot{z}$. Notice that
 \begin{align}\label{I*Z1Z2}
  0 =& {I}^*(\delta z_{*2}, \delta {U}^*_2,\delta {\varphi^*_2}', \delta U_-) - {I}^*(\delta z_{*1},  \delta {U}^*_1,\delta {\varphi_1^*}', \delta U_-)
  \notag \\
 = &{I}^*(\delta z_{*2}, \delta U^*_2 , \delta {\varphi^*_2}' , \delta U_-) - {I}^*(\delta z_{*1}, \delta U^*_2 , \delta {\varphi^*_2}' , \delta U_-)
  \notag \\
& + {I}^*(\delta z_{*1}, \delta U^*_2 , \delta {\varphi^*_2}' , \delta U_-) - {I}^*(\delta z_{*1}, \delta U^*_1 , \delta {\varphi^*_1}', \delta U_-)
 \notag \\
= &\int_{0}^{1} \frac{\partial{I}^* }{\partial (\delta z_*)} (\delta z_{*t}, \delta U^*_2,\delta {\varphi^*_2}', \delta U_-)\dif t (\delta z_{*2} - \delta z_{*1})
 \notag \\
& + \int_{0}^{1} \nabla_{(\delta U^*, \delta {\varphi^*}')} {I}^* (\delta z_{*1} , \delta U^*_t, \delta {\varphi^*_t}' , \delta U_-)\dif t \cdot (\delta U^*_2 - \delta U^*_1, \delta {\varphi^*_2}' - \delta {\varphi^*_1}' ),
\end{align}
where
\begin{align*}
 \delta \xi_{*t} : = t \delta z_{*2} + (1 - t) \delta z_{*1}\quad
    \delta U^*_t : = t \delta U^*_2 + (1 - t) \delta U^*_1 \quad
  \delta {\varphi^*_t}' : = t \delta {\varphi^*_2}' + (1 - t)\delta {\varphi^*_1}'.
\end{align*}
Similarly as done in step 2 of the proof of Lemma  \ref{lemwelldefined}, we have
\begin{align}\label{eq305}
\frac{\partial {I}^* }{\partial (\delta z_*)} (\delta z_{*t}, \delta U^*_2,\delta {\varphi^*_2}', \delta U_-)
 = \frac{\partial {I}^* }{\partial (\delta z_*)}(0 , 0, 0, \dot{U}_-) 
+ O(1) \sigma^{\frac32} + O(1) \sigma^2.
\end{align}
Moreover, it is easy to see that
\begin{align}\label{6.105x}
  \nabla_{(\delta U^*, \delta {\varphi^*}')} I^* (\delta z_{*1} , \delta U^*_t, \delta {\varphi^*_t}' , \delta U_-) = O(1)\sigma,
\end{align}
where $O(1)$ only depends on $ {I}'(\dot{z}_*)$, $\dot{z}_*$, $\overline{U}_\pm$, $L$, $p_{ex}$ ,$w_{en}$, $q_{en}$ and $\alpha$.
Therefore, \eqref{I*Z1Z2}-\eqref{6.105x} yield that
\begin{align}\label{eq306}
&|\delta z_{*2} - \delta z_{*1}| \notag\\
 \leq& C \Big(\| (\delta p^*_2 - \delta p^*_1,\delta \theta^*_2 - \delta \theta^*_1)\|_{C^{2,\alpha}(\overline{\dot{\Omega}_+})}+\| (\delta q^*_2 - \delta q^*_1,\delta s^*_2- \delta s^*_1, r\delta w^*_2 - r\delta w^*_1)\|_{C^{1,\alpha}(\overline{\dot{\Omega}_+})} \notag\\
     & \qquad + \| \delta w^*_2 - \delta w^*_1\|_{C^{0,\alpha}(\overline{\dot{\Omega}_+})} + \| \delta {\varphi^*_2}' - \delta {\varphi^*_1}'\|_{C^{1,\alpha}(\dot{\Gamma}_s)}\Big),
\end{align}
where the constant $C$ only depends on $\dot{z}_*$, $\overline{U}_\pm$, $L$, $p_{ex}$ ,$w_{en}$, $q_{en}$ and $\alpha$.
Therefore, it follows from \eqref{ee} and  \eqref{eq306} that there exists  a positive constant $\sigma_3$ with $0<\sigma_3 \ll 1$, such that for any $0< \sigma \leq \sigma_3$,
\begin{align*}
     &\| (\delta p_2^* - \delta p_1^*,\delta \theta_2^* - \delta \theta_1^*)\|_{C^{2,\alpha}(\overline{\dot{\Omega}_+})} + \| ( \delta q_2^* - \delta q_1^*,\delta s_2^* - \delta s_1^*,r\delta w_2^* - r\delta w_1^*)\|_{C^{1,\alpha}(\overline{\dot{\Omega}_+})} \notag\\
     & + \| \delta w_2^* - \delta w_1^*\|_{C^{0,\alpha}(\overline{\dot{\Omega}_+})} + \| \delta \varphi_{2}^{*'} - \delta \varphi_{1}^{*'}\|_{C^{1,\alpha}(\dot{\Gamma}_s)}+| z_{*2} - z_{*1}|
      \notag \\
     \leq& \frac12\Big(\| (\delta p_2 - \delta p_1,\delta \theta_2 - \delta \theta_1)\|_{C^{2,\alpha}(\overline{\dot{\Omega}_+})}+\| (\delta q_2 - \delta q_1,\delta s_2- \delta s_1, r\delta w_2 - r\delta w_1)\|_{C^{1,\alpha}(\overline{\dot{\Omega}_+})} \notag\\
     & \qquad + \| \delta w_2 - \delta w_1\|_{C^{0,\alpha}(\overline{\dot{\Omega}_+})} + \| \delta {\varphi_2}' - \delta {\varphi_1}'\|_{C^{1,\alpha}(\dot{\Gamma}_s)}+|z_1- z_2|\Big).
   \end{align*}
So the mapping $\mathcal{T}$ defined in \eqref{6.63x} is contractive.
\end{proof}

{\bf Acknowledgments.}
The research of Beixiang Fang was supported in part by National Key R\&D Program of China under Grant No. 2024YFA1013302, Natural Science
Foundation of China under Grant No. 12331008, and the Fundamental Research Funds for the Central Universities. The research of Xin Gao was supported in part by
Scholarship of Shanghai Jiao Tong University.
The research of Wei Xiang was supported in part by the
Research Grants Council of the HKSAR, China (Project No. CityU 11300021, CityU 11311722, CityU 11305523 and CityU 11305625), and in part by the Research Center for Nonlinear Analysis of the Hong Kong Polytechnic University and the PolyU Internal project P0045335.
The research of Qin Zhao was supported in part by Natural Science Foundation of China under Grant Nos. 12101471 and 12272297.

{\bf Data Availability.}
Data sharing not applicable to this article as no datasets were generated or analysed during
the current study.

{\bf Conflict of Interest.}
The authors declare no conflict of interest.


\begin{thebibliography}{HHH}\addtolength{\itemsep}{-2.5ex}

\bibitem{BF} M. Bae, M. Feldman; Transonic shocks in multidimensional divergent nozzles.\emph{ Arch. Ration. Mech. Anal.} 201 (2011), no. 3, 777-840. \\


\bibitem{GCM2007}G.-Q. Chen, J. Chen, M. Feldman; Transonic shocks and free boundary problems for the full
Euler equations in infinite nozzles. \emph{J. Math. Pures Appl.} (9) 88 (2007), no. 2, 191-218.\\


\bibitem{GK2006} G.-Q. Chen, J. Chen, K. Song; Transonic nozzle flows and free boundary problems for the full
Euler equations. \emph{J. Differential Equations} 229 (2006), no. 1, 92-120.\\

\bibitem{GM2003} G.-Q. Chen, M. Feldman; Multidimensional transonic shocks and free boundary problems for
nonlinear equations of mixed type. \emph{J. Amer. Math. Soc.} 16 (2003), no. 3, 461-494.\\


\bibitem{GM2004} G.-Q. Chen, M. Feldman; Steady transonic shocks and free boundary problems for the Euler
equations in infinite cylinders. \emph{Comm. Pure Appl. Math.} 57 (2004), no. 3, 310-356.\\



\bibitem{GM2007} G.-Q. Chen, M. Feldman; Existence and stability of multidimensional transonic flows through
an infinite nozzle of arbitrary cross-sections. \emph{Arch. Ration. Mech. Anal.} 184 (2007), no. 2, 185-242.\\


\bibitem{SC2005} S. Chen; Stability of transonic shock fronts in two-dimensional Euler systems. \emph{Trans. Amer. Math. Soc.} 357 (2005), no. 1, 287-308.\\


\bibitem{Chen_S2008TAMS} S. Chen; Transonic shocks in 3-d compressible
flow passing a duct with a general section for Euler systems. \emph{Trans. Amer. Math. Soc.} 360 (2008), no. 10, 5265-5289.\\

\bibitem{SC2009} S. Chen; Compressible flow and transonic shock in a diverging nozzle. \emph{Comm. Math. Phys.} 289 (2009), no. 1, 75-106.\\


 \bibitem{CY2008}S. Chen, H. Yuan; Transonic shocks in compressible flow passing a duct for three-dimensional
Euler systems. \emph{Arch. Ration. Mech. Anal.} 187 (2008), no. 3, 523-556.\\

 \bibitem{CR} R. Courant, K.O. Friedrichs; Supersonic flow and shock waves. Springer-Verlag, New York, 1948.\\

\bibitem{DWX} X. Deng, T. Wang, W. Xiang; Three-dimensional full Euler flows with nontrivial swirl in axisymmetric nozzles. \emph{SIAM J. Math. Anal.} 50 (2018), no. 3, 2740-2772.\\

\bibitem{DW} B. Duan, S. Weng; Global smooth axisymmetric subsonic flows with nonzero swirl in an infinitely long axisymmetric nozzle. \emph{Z. Angew. Math. Phys.} 69 (2018), no. 5, Paper No. 135, 17 pp.\\

\bibitem{FG1} B. Fang, X. Gao;
On admissible positions of transonic shocks for steady Euler flows in a 3-d axisymmetric cylindrical nozzle. \emph{J. Differential Equations} 288 (2021), 62-117.\\

\bibitem{FG3} B. Fang, X. Gao; On admissible positions of transonic shocks for steady isothermal Euler flows in a horizontal flat nozzle under vertical gravity. \emph{SIAM J. Math. Anal.} 54 (2022), no. 5, 5223-5267.\\

\bibitem{FHXX} B. Fang, F. Huang, W. Xiang, F. Xiao; Persistence of the steady planar normal shock structure in 3-d unsteady potential flows.
\emph{J. Lond. Math. Soc.} 107 (2023), no. 5, 1692-1753.\\

\bibitem{FL} B. Fang, L. Liu, and H. Yuan; Global uniqueness of transonic shocks in two-dimensional
steady compressible Euler flows, \emph{Arch. Ration. Mech. Anal.} 207 (2013), 317-345\\

\bibitem{FX} B. Fang, W. Xiang; The uniqueness of transonic shocks in supersonic flow past a 2-D wedge, \emph{J. Math. Anal. Appl.} 437(2016), 194-213\\

\bibitem{FXX} B. Fang, W. Xiang, F. Xiao; Persistence of the steady normal shock structure for the unsteady potential flow.
\emph{SIAM J. Math. Anal.} 52(2020), no. 6, 6033-6104.\\

\bibitem{FB63}  B. Fang, Z. Xin; On admissible locations of transonic shock fronts for steady Euler flows in an almost flat finite nozzle with prescribed receiver pressure. \emph{Comm. Pure Appl. Math.} 74 (2021), no. 7, 1493-1544.\\

\bibitem{GM} M. Giaquinta, L. Martinazzi; An introduction to the regularity theory for elliptic systems, harmonic maps and minimal graphs. Second edition. Appunti. Scuola Normale Superiore di Pisa (Nuova Serie) [Lecture Notes. Scuola Normale Superiore di Pisa (New Series)], 11. Edizioni della Normale, Pisa, 2012.\\

 \bibitem{DNS} D. Gilbarg, N.S. Trudinger; Elliptic partial differential equations of second order. 3rd ed.
Grundlehren Math. Wiss. 224, Springer, Berlin, New York, 1998.\\

\bibitem{Lars} L. Hormander; Lectures on nonlinear hyperbolic differential equations. Mathematics \& Applications, 26. Springer-Verlag, Berlin, 1997.\\


\bibitem{HKWX2021}  F. Huang, J. Kuang, D. Wang, W. Xiang; Stability of transonic contact discontinuity for two-dimensional steady compressible {E}uler flows in a
  finitely long nozzle. \emph{Ann. PDE} 7 (2021), no. 2, Paper No. 23, 96 pp.\\


\bibitem{LXY2009} J. Li, Z. Xin, H. Yin; On transonic shocks in a nozzle with variable end pressures. \emph{Comm.
Math. Phys.} 291 (2009), no. 1, 111-150.\\



\bibitem{LiXY2010}  J. Li, Z. Xin, H. Yin; On transonic shocks in a conic divergent nozzle with axi-symmetric exit
pressures. \emph{J. Differential Equations} 248 (2010), no. 3, 423-469.\\


\bibitem{LiXinYin2011PJM} J. Li, Z. Xin, H. Yin;  Monotonicity
and uniqueness of a 3D transonic shock solution in a conic nozzle
with variable end pressure. \emph{Pacific J. Math.} 254 (2011), no. 1, 129-171.\\

\bibitem{LXY2013} J. Li, Z. Xin, H. Yin; Transonic shocks for the full compressible Euler system in a general
two-dimensional de Laval nozzle. \emph{Arch. Ration. Mech. Anal.} 207 (2013), no. 2, 533-581.\\


\bibitem{GM13} G. Lieberman; Oblique derivative problems for elliptic equations. World Scientific Publishing Co. Pte. Ltd., Hackensack, NJ, 2013.\\


\bibitem{LiuXuYuan2016AdM} L. Liu, G. Xu, H. Yuan; Stability
of spherically symmetric subsonic flows and transonic shocks under
multidimensional perturbations. \emph{Adv. Math.} 291 (2016), 696-757.\\

\bibitem{LiuYuan2009SIAMJMA} L. Liu, H. Yuan; Global uniqueness
of transonic shocks in divergent nozzles for steady potential flows.
\emph{SIAM J. Math. Anal.} 41 (2009), no. 5, 1816-1824.\\


\bibitem{Liu1982ARMA} T. Liu; Transonic gas flow in a duct
of varying area. \emph{Arch. Rational Mech. Anal.} 80 (1982), no. 1, 1-18.\\

\bibitem{Liu1982CMP} T. Liu; Nonlinear stability and instability
of transonic flows through a nozzle. \emph{Comm. Math. Phys.}  83 (1982),
no. 2, 243-260.\\

\bibitem{LyonsZumbrun} R. Lyons, K. Zumbrun;
Homogeneous partial derivatives of radial functions.
\emph{Proc. Amer. Math. Soc.}, 121 (1994), 315-316.\\

\bibitem{HPark} H. Park;
Transonic shocks for three-dimensional axisymmetric flows in divergent nozzles. arXiv:2304.08100, 2023.\\

\bibitem{ParkRyu_2019arxiv} H. Park, H. Ryu; Transonic shocks for 3-d axisymmetric compressible inviscid flows in cylinders. \emph{J. Differential Equations} 269 (2020), no. 9, 7326-7355.\\

\bibitem{Yong}Y. Park; 3-d axisymmetric transonic shock solutions of the full Euler system in divergent nozzles. \emph{Arch. Ration. Mech. Anal.} 240 (2021), no. 1, 467-563.\\

\bibitem{WS94}S. Weng, C. Xie, Z. Xin; Structural stability of the transonic shock problem in a divergent three-dimensional axisymmetric perturbed nozzle. \emph{SIAM J. Math. Anal.} 53 (2021), no. 1, 279-308.\\

\bibitem{WengXin} S. Weng, Z. Xin; Existence and stability of cylindrical transonic shock solutions under three dimensional perturbations. \emph{arXiv:2304.02429}, 2023.\\

\bibitem{XinYanYin2011ARMA} Z. Xin, W. Yan, H. Yin;
Transonic shock problem for the Euler system in a nozzle. \emph{Arch. Ration.	Mech. Anal.}  194 (2009), no. 1, 1-47. \\


\bibitem{XinYin2005CPAM} Z. Xin, H. Yin; Transonic shock
in a nozzle. I. Two-dimensional case. \emph{Comm. Pure Appl. Math.}  58 (2005),
no. 8, 999-1050.\\

\bibitem{XinYin2008JDE} Z. Xin, H. Yin; The transonic
shock in a nozzle, 2-d and 3-d complete Euler systems. \emph{J. Differential Equations} 245 (2008), no. 4, 1014-1085.\\


\bibitem{HY2007} H. Yuan; Transonic shocks for steady Euler flows with cylindrical symmetry.
\emph{Nonlinear Anal.} 66 (2007), no. 8, 1853--1878.\\


\bibitem{Yuan2012NA} H. Yuan; Persistence of shocks in ducts.
\emph{Nonlinear Anal.} 75 (2012), no. 9, 3874-3894.\\


\bibitem{HZ1} H. Yuan and Q. Zhao; Stabilization effect of frictions for transonic shocks in steady compressible Euler flows passing three-dimensional ducts, \emph{Acta Math. Sci. Ser. B Engl. Ed.},
40 (2020), pp. 470-502.


\end{thebibliography}
\end{document}